\documentclass[reqno]{amsart}

\usepackage{mathtools,amssymb,mathrsfs}
\usepackage[english]{babel}
\usepackage[utf8]{inputenc}
\usepackage[linktocpage,colorlinks=false]{hyperref}
\usepackage{tikz-cd}
\usepackage{upgreek}
\usepackage{textcomp}
\usepackage{amsmath}
\usepackage{graphicx}
\usepackage{float}
\usepackage{amsthm}
\usepackage{amsfonts}
\usepackage{enumitem}
\usepackage[symbol]{footmisc}

\numberwithin{equation}{section}

\usepackage{tikz-cd, mathtools}

\newtheorem{thm}{Theorem}[section]
\newtheorem{cor}[thm]{Corollary}
\newtheorem{lem}[thm]{Lemma}
\newtheorem{prop}[thm]{Proposition}
\newtheorem{conj}[thm]{Conjecture}
\theoremstyle{definition}
\newtheorem{defn}[thm]{Definition}
\newtheorem{exmp}[thm]{Example}
\newtheorem{assm}[thm]{Assumption}
\newtheorem{rem}[thm]{Remark}
\DeclareMathOperator{\Hom}{Hom}
\DeclareMathOperator{\HC}{HC}
\DeclareMathOperator{\CH}{CH}
\DeclareMathOperator{\pr}{pr}
\newcommand{\wh}{\widehat}
\newcommand{\ot}{\otimes}
\newcommand{\vphi}{\varphi}
\newcommand{\fd}[2]{{}_{#1}\hspace{-0.33333em}\times_{#2}\hspace{-0.16667em}}

\numberwithin{equation}{section}

\begin{document}
	
    \author{Yi Wang}
    \address{Department of Mathematics, Purdue University, West Lafayette, IN 47907, USA}
    \email{wang6206@purdue.edu}
	\title[cyclic homology, $S^1$-equivariant homology, string topology]{A cocyclic construction of $S^1$-equivariant homology and application to string topology}


	\begin{abstract}
		Given a space with a circle action, we study certain cocyclic chain complexes and prove a theorem relating cyclic homology to $S^1$-equivariant homology, in the spirit of celebrated work of Jones. As an application, we describe a chain level refinement of the gravity algebra structure on the (negative) $S^1$-equivariant homology of the free loop space of a closed oriented smooth manifold, based on work of Irie on chain level string topology and work of Ward on an $S^1$-equivariant version of operadic Deligne's conjecture.
	\end{abstract}
	
	\maketitle

	
	\section{Introduction}
	Let $M$ be a closed oriented smooth manifold and $\mathcal{L}M=C^\infty(S^1,M)$ be the smooth free loop space of $M$. In a seminal paper \cite{Chas-Sullivan string topology} (and a sequal \cite{Chas-Sullivan gravity}), Chas-Sullivan discovered rich algebraic structures on the ordinary homology and $S^1$-equivariant homology of $\mathcal{L}M$, initiating the study of string topology. In particular, there is a Batalin-Vilkovisky (BV) algebra structure on (shifted) $H_*(\mathcal{L}M)$ (\cite[Theorem 5.4]{Chas-Sullivan string topology}), which naturally induces a gravity algebra structure on (shifted) $H^{S^1}_*(\mathcal{L}M)$ (\cite[Section 6]{Chas-Sullivan string topology}, \cite[page 18]{Chas-Sullivan gravity}).
	
	The goal of this paper is to describe a chain level refinement of the string topology gravity algebra, and compare it with an algebraic counterpart related to the de Rham dg algebra $\Omega(M)$. Along the way we also obtain results on the relation between cyclic homology and $S^1$-equivariant homology, and an $S^1$-equivariant version of Deligne's conjecture.
	
	In spirit, this paper may be compared with work of Westerland \cite{Westerland}. Westerland gave a homotopy theoretic generalization of the gravity operations on the (negative) $S^1$-equivariant homology of $\mathcal{L}M$, whereas we describe a chain level refinement.

	\subsubsection*{Cyclic homology and $S^1$-equivariant homology}
	
	The close connection between cyclic homology (algebra) and $S^1$-equivariant homology (topology) was first systematically studied by Jones in \cite{Jones cyclic}. One of the main theorems in that paper (\cite[Theorem 3.3]{Jones cyclic}) says that the singular chains  $\{S_k(X)\}_{k\geq0}$ of an $S^1$-space $X$ can be made into a cyclic module, such that there are natural isomorphisms between three versions of cyclic homology (positive, periodic, negative) of $\{S_k(X)\}_{k\geq0}$ and three versions of $S^1$-equivariant homology of $X$, in a way compatible with long exact sequences. 
	
	The first result in this paper is a theorem ``cyclic dual'' to Jones' theorem. As far as the author knows, such a result did not appear in the literature.
	\begin{thm}[See Theorem {\ref{thm:cocyclic topological isom}}]\label{thm:main result 1}
		Let $X$ be a topological space with an $S^1$-action. Then $\{S_*(X\times\Delta^k)\}_{k\geq0}$ can be made into a cocyclic chain complex, such that there are natural isomorphisms between three versions of cyclic homology of $\{S_*(X\times\Delta^k)\}_{k\geq0}$ and three versions of $S^1$-equivariant homology of $X$, in a way compatible with long exact sequences.
	\end{thm}
	
	Jones dealt with the cyclic set $\{\mathrm{Map}(\Delta^k,X)\}_{k}$ and the cyclic module $\{S_k(X)\}_k$, while we deal with the cocyclic space $\{X\times\Delta^k\}_{k}$ and the cocyclic complex $\{S_*(X\times\Delta^k)\}_{k}$. It is in this sense that these two theorems are ``cyclic dual'' to each other. In the special case that $X$ is the free loop space of a topological space $Y$, Theorem \ref{thm:main result 1} may also be viewed as ``cyclic dual'' to a result of Goodwillie (\cite[Lemma V.1.4]{Goodwillie}). As does Jones' theorem, Theorem \ref{thm:main result 1} has the advantage that it works for all $S^1$-spaces.
	
	The cyclic structure on singular chains plays no role in Theorem \ref{thm:main result 1}; what matters is the cocyclic space. 
	Indeed, the main motivation for the author to seek for a result like Theorem \ref{thm:main result 1} is to study the $S^1$-equivariant homology of $\mathcal{L}M$, using a novel chain model of loop space homology defined via certain ``de Rham chains'', introduced by Irie \cite{Irie loop}. 
	
	\subsubsection*{Deligne's conjecture}
	
	What is called Deligne's conjecture asks whether there is an action of a certain chain model of the little disks operad on the Hochschild cochain complex of an associative algebra, inducing the Gerstenhaber algebra structure on Hochschild cohomology discovered by Gerstenhaber \cite{Gerstenhaber}. This conjecture, as well as some variations and generalizations, has been answered affirmatively by many authors, to whom we are apologetic not to list here. What is of most interest and importance to us is work of Ward \cite{Ward}.
	
	Ward (\cite[Theorem C]{Ward}) gave a general solution to the question when certain complex of cyclic (co)invariants admits an action of a chain model of the gravity operad, inducing the gravity algebra structure on cyclic cohomology. Recall that the gravity operad was introduced by Getzler \cite{Getzler gravity} and is the $S^1$-equivariant homology of the little disks operad. So Ward's result can be viewed as an $S^1$-equivariant version of operadic Deligne's conjecture (\cite[Corollary 5.22]{Ward}).
	
	The second result in this paper is an extension, in a special case, of Ward's theorem. To state our result, let $A$ be a dg algebra equipped with a symmetric, cyclic, bilinear form $\langle,\rangle:A\ot A\to\mathbb{R}$ of degree $m\in\mathbb{Z}$ satisfying Leibniz rule (see Example \ref{exmp:end operad}). Then $\langle,\rangle$ induces a dg $A$-bimodule map $\theta:A\to A^\vee[m]$, and hence a cochain map $\Theta:\CH(A,A)\to\CH(A,A^\vee[m])$ between Hochschild cochains. Let $\CH_{\mathrm{cyc}}(A,A^\vee[m])$ be the subcomplex of cyclic invariants in $\CH(A,A^\vee[m])$. Let $\mathsf{M}_{\circlearrowleft}$ be the chain model of the gravity operad that Ward constructed (see also Example \ref{exmp:operads}\eqref{item:example operad 3}).
	\begin{thm}[See Corollary \ref{cor:homotopy gravity action}]\label{thm:main result 2}
		Given $A,\langle,\rangle,\theta,\Theta$ as above, there is an action of $\mathsf{M}_{\circlearrowleft}$ on $\Theta^{-1}(\CH_{\mathrm{cyc}}(A,A^\vee[m]))$, giving rise to a structure of a gravity algebra up to homotopy. If $\theta$ is a quasi-isomorphism and $\Theta$ restricts to a quasi-isomrphism $\Theta^{-1}(\CH_{\mathrm{cyc}}(A,A^\vee[m]))\to\CH_{\mathrm{cyc}}(A,A^\vee[m])$, this descends to a gravity algebra structure on the cyclic cohomology of $A$, which is compatible with the BV algebra structure on Hochschild cohomology.
	\end{thm}
	
	Here the BV algebra structure on the Hochschild cohomology of $A$ (when $\theta$ is a quasi-isomorphism) is well-known (e.g. Menichi \cite[Theorem 18]{Menichi cyclic}), where the BV operator is given by Connes' operator (Example \ref{exmp:cocyclic complex}). By compatibility with a BV algebra structure we mean the content of Lemma \ref{lem:BV induce gravity}. Note that Ward's original theorem only applies to the situation that $\theta$ is an isomorphism (\cite[Corollary 6.2]{Ward}). 
	
	\subsubsection*{Chain level structures in $S^1$-equivariant string topology}
	Let us say more about Irie's work \cite{Irie loop}. Using his chain model and results of Ward (\cite[Theorem A, Theorem B]{Ward}), Irie obtained an operadic chain level refinement of the string topology BV algebra, and compared it with a solution to the ordinary Deligne's conjecture via a chain map defined by iterated integrals of differential forms. 
	
	The third result in this paper is a similar story in the $S^1$-equivariant context. Note that the string topology BV algebra induces gravity algebra structures on two versions (positive i.e. ordinary, and negative) of $S^1$-equivariant homology of $\mathcal{L}M$ (Example \ref{exmp:string topology gravity}).
	
	\begin{thm}[See Theorem \ref{thm:chain level gravity string topology})]\label{thm:main result 3}
		For any closed oriented $C^\infty$-manifold $M$, there exists a chain complex $\tilde{\mathcal{O}}^{\mathrm{cyc}}_M$ satisfying the following properties. Firstly, the homology of $\tilde{\mathcal{O}}^{\mathrm{cyc}}_M$ is isomorphic to the negative $S^1$-equivariant homology of $\mathcal{L}M$, and $\tilde{\mathcal{O}}^{\mathrm{cyc}}_M$ admits an action of $\mathsf{M}_{\circlearrowleft}$ (hence an up-to-homotopy gravity algebra structure) which lifts the gravity algebra structure mentioned above. Secondly, there is a morphism of $\mathsf{M}_{\circlearrowleft}$-algebras 
		\begin{equation}\label{eqn:main result morphism}
			\tilde{\mathcal{O}}^{\mathrm{cyc}}_M\to\Theta^{-1}(\CH_{\mathrm{cyc}}(\Omega(M),\Omega(M)^\vee[-\dim M]))
		\end{equation}
		which is induced by iterated integrals of differential forms, where the structure on right-hand side follows from Theorem \ref{thm:main result 2} and $\Theta$ comes from the Poincar\'e pairing. At homology level, the morphism \eqref{eqn:main result morphism} descends to a map (part of arrow 4 below) which fits into a commutative diagram of gravity algebra homomorphisms			\begin{equation}\label{eqn:main result diagram}				\begin{tikzcd}					
				A\arrow[r,"1"]\arrow[d,"2"] & B\arrow[d,"3"] \\					C\arrow[r,"4"] & D.			
			\end{tikzcd}			
		\end{equation}		    Here $A$ is the $S^1$-equivariant homology of $\mathcal{L}M$, $B$ is the negative cyclic cohomology of $\Omega(M)$, $C$ is the negative $S^1$-equivariant homology of $\mathcal{L}M$, $D$ is the cyclic cohomology of $\Omega(M)$. Arrows 1, 4 are defined by iterated integrals on free loop space, and arrow 2 (resp. 3) is the connecting map in the tautological long exact sequence for $S^1$-equivariant homology theories (resp. cyclic homology theories).
	\end{thm}
	
	The crucial part of Theorem \ref{thm:main result 3} is, of course, the chain level statement that fits well with structures on homology. The first part of Theorem \ref{thm:main result 3} was conjectured by Ward in \cite[Example 6.12]{Ward}, but the correct statement turns out to be more complicated, as we actually lift gravity algebra structures on negative $S^1$-equivariant homology rather than $S^1$-equivariant homology, whereas they are naturally related by a morphism (arrow 2).
	
	Other than the chain level statement, part of the results at homology level is known. For example, the fact that arrow 1 is a Lie algebra homomorphism appeared in work of Abbaspour-Tradler-Zeinalian as \cite[Theorem 11]{ATZ}; The fact that \eqref{eqn:main result diagram} commutes was of importance to Cieliebak-Volkov \cite{CV cyclic} (the arrows are only treated as linear maps there). 	
	
	In a forthcoming paper, the author is going to apply results in this paper to Lagrangian Floer theory, in view of cyclic symmetry therein (Fukaya \cite{Fukaya cyclic}).
	
	\subsection*{Outline} In Section \ref{section:cyclic homology}, we review cyclic homology of mixed complexes. In Section \ref{section:A cocyclic complex and an infty-quasi-isomorphism}, we prove Theorem \ref{thm:main result 1}. In Section \ref{section:differentiable space}, we review Irie's de Rham chain complex of differentiable spaces and apply Theorem \ref{thm:main result 1} to it. In Section \ref{section:operads and algebraic structures}, we review basics of operads and algebraic structures. In Section \ref{section:cyclic brace operations}, we prove Theorem \ref{thm:main result 2}. 
	In Section \ref{section:application to string topology}, we prove Theorem \ref{thm:main result 3}.
	
	\subsection*{Conventions} Vector spaces are over $\mathbb{R}$, algebras are associative and unital, graded objects are $\mathbb{Z}$-graded. Homological and cohomological gradings are mixed by the understanding $C_*=C^{-*},C^*=C_{-*}$. As for sign rules, see Appendix \ref{section:sign rule}. For the sake of convenience, we may write $(-1)^{\varepsilon}$ for a sign that is apparent from Koszul sign rule (Appendix \ref{subsection:Koszul sign rule}).
	
	\subsection*{Acknowledgements} This work was completed when I was a graduate student at Stony Brook University. I am grateful to my advisor, Kenji Fukaya, for guidance, encouragement and support. I also thank Kei Irie for comments on an earlier version of the paper, and thank Ben Ward for a useful email.
	
	\section{Preliminaries on cyclic homology}\label{section:cyclic homology} A convenient way to study different versions of cyclic homology is to work in the context of \emph{mixed complexes}, which was introduced by Kassel \cite{Kassel}. By definition, a mixed cochain complex is a triple $(C^*,b,B)$ consisting of a graded vector space $C^*$ and  linear maps
	$b:C^*\to C^{*+1}$, $B:C^*\to C^{*-1}$ such that	
	\begin{equation*}b^2=0,\hspace{.33333em}B^2=0,\hspace{.33333em}bB+Bb=0.\end{equation*}
	Let $u$ be a formal variable of degree $2$. Define graded $\mathbb{R}[u]$-modules $C[[u]]^*$, $C[[u,u^{-1}]^*$, $C[u^{-1}]^*$ by
	\begin{align*}
		C[[u]]^n&:=\Big\{\sum_{i\geq0} c_iu^i\mid c_i\in C^{n-2i}\Big\},\\ 
		C[[u,u^{-1}]^n&:=\Big\{\sum_{i\geq-k} c_iu^i\mid k\in\mathbb{Z}_{\geq0},\hspace{.16667em}c_i\in C^{n-2i}\Big\},\\
		C[u^{-1}]^n&:=\Big\{\sum_{-k\leq i\leq0} c_iu^i\mid k\in\mathbb{Z}_{\geq0},\hspace{.16667em}c_i\in C^{n-2i}\Big\}.
	\end{align*}
	Here the $\mathbb{R}[u]$-module structure on $C[u^{-1}]$ is induced by the identification $C[u^{-1}]=C[[u,u^{-1}]/uC[[u]]$.
	Then $b+uB$ is a differential on $C[[u]]^*$, $C[[u,u^{-1}]^*$, $C[u^{-1}]^*$, resulting in cohomology groups denoted by 
	\begin{equation*}
		\HC^*_{[[u]]}(C),\hspace{.33333em}\HC^*_{[[u,u^{-1}]}(C),\hspace{.33333em}\HC^*_{[u^{-1}]}(C).
	\end{equation*} 
	These are three classical versions of \emph{cyclic homology} of mixed complexes, called the negative, periodic and ordinary (positive) cyclic homology of $(C^*,b,B)$, respectively.  We prefer to distinguish them by suggestive symbols ($[[u]]$, $[[u,u^{-1}]$, $[u^{-1}]$) rather than names, as did in \cite{CV cyclic}. Here cohomological grading is used for cyclic homology since we deal with cochain complexes. If we move to homological grading $C_*:=C^{-*}$ and replace $u$ by $v$ (a formal variable of degree $-2$), then the mixed chain complex $(C_*,b,B)$ gives negative, periodic and ordinary (positive) cyclic homology theories \begin{equation*}\HC_*^{[[v]]}=\HC^{-*}_{[[u]]}, \hspace{.33333em}\HC_*^{[[v,v^{-1}]}=\HC^{-*}_{[[u,u^{-1}]}, \hspace{.33333em}\HC_*^{[v^{-1}]}=\HC^{-*}_{[u^{-1}]}.\end{equation*} (\cite{CV cyclic} also takes the Hom dual of $C$ to define cyclic cohomology theories of $(C,b,B)$, which we try to avoid in this article.)
	
	For any mixed cochain complex $(C^*,b,B)$, there is a \emph{tautological exact sequence}
	\begin{equation}\label{eqn:cyclic exact seq tautological}
		\cdots\to  \HC^*_{[[u]]}(C)\xrightarrow{i_*} \HC^*_{[[u,u^{-1}]}(C)\xrightarrow{u\cdot p_*} \HC^{*+2}_{[u^{-1}]}(C)\xrightarrow{B_{0*}} \HC^{*+1}_{[[u]]}(C)\to\cdots
	\end{equation}
	which is induced by the short exact sequence
	\begin{equation*}
		0\to C[[u]]\xrightarrow{i} C[[u,u^{-1}]\xrightarrow{p} C[[u,u^{-1}]/C[[u]]\to 0
	\end{equation*}
	and the $(b+uB)$-cochain isomorphism \begin{equation*}(C[[u,u^{-1}]/C[[u]])^*\underset{\cong}{\xrightarrow{\cdot u}}C[u^{-1}]^{*+2};\hspace{.5em}\sum_{-k\leq i\leq -1}c_i u^i\mapsto \sum_{-k\leq i\leq -1}c_i u^{i+1}.\end{equation*}
	The connecting map $B_{0*}:\HC^{*+2}_{[u^{-1}]}(C)\to \HC^{*+1}_{[[u]]}(C)$ is induced by an anti-cochain map
	\begin{equation*}B_0:C[u^{-1}]^{*+2}\to C[[u]]^{*+1};\hspace{.16667em}\sum_{-k\leq i\leq 0}c_iu^i\mapsto B(c_0).\end{equation*}
	Similarly, from the short exact sequences 
	\begin{align*}
		&0\to C[[u]]/uC[[u]]\xrightarrow{i} C[[u,u^{-1}]/uC[[u]]\xrightarrow{p}C[[u,u^{-1}]/C[[u]]\to 0\\
		&0\to uC[[u]]\xrightarrow{i^+} C[[u]]\xrightarrow{p_0} C[[u]]/uC[[u]]\to 0
	\end{align*} 
	one obtains the \emph{Gysin-Connes exact sequences} 
	\begin{subequations}\label{eqn:cyclic exact seq gysin}
		\begin{align}
			\label{eqn:cyclic exact seq gysin [u-1]}&\cdots\to H^*(C,b)\xrightarrow{i_*} \HC_{[u^{-1}]}^*(C)\xrightarrow{u\cdot p_*} \HC_{[u^{-1}]}^{*+2}(C)\xrightarrow{B_{0*}} H^{*+1}(C,b)\to\cdots\\
			\label{eqn:cyclic exact seq gysin [[u]]}&\cdots\to \HC_{[[u]]}^{*-2}(C)\xrightarrow{i^+_*\cdot u} \HC_{[[u]]}^*(C)\xrightarrow{p_{0*}} H^*(C,b)\xrightarrow{B_*} \HC_{[[u]]}^{*-1}(C)\to\cdots
		\end{align}
	\end{subequations}
	The connecting maps $\HC_{[u^{-1}]}^{*+2}(C)\xrightarrow{B_{0*}} H^{*+1}(C,b)$ and $H^*(C,b)\xrightarrow{B_*} \HC_{[[u]]}^{*-1}(C)$ are induced by anti-cochain maps $B_0$ and $B$, respectively.
	\begin{lem}\label{lem:cyclic exact seq gysin morphism}
		The map $B_{0*}:\HC^{*+2}_{[u^{-1}]}(C)\to \HC^{*+1}_{[[u]]}(C)$ in \eqref{eqn:cyclic exact seq tautological} and the exact sequences \eqref{eqn:cyclic exact seq gysin} fit into the following commutative diagram:
		\begin{center}
			\begin{tikzcd}
				\cdots \to H^*(C,b) \arrow[r,"i_*"] \arrow[d,equal,"\mathrm{id}"] & \HC_{[u^{-1}]}^*(C) \arrow[r,"u\cdot p_*"] \arrow[d,"B_{0*}"] & \HC_{[u^{-1}]}^{*+2}(C)\arrow[r,"B_{0*}"]\arrow[d,"B_{0*}"] & H^{*+1}(C,b)\arrow[d,equal,"\mathrm{id}"]  \to\cdots\\
				\cdots \to H^*(C,b) \arrow[r,"B_*"] & \HC_{[[u]]}^{*-1}(C) \arrow[r,"i^+_*\cdot u"] & \HC_{[[u]]}^{*+1}(C)\arrow[r,"p_{0*}"]& H^{*+1}(C,b) \to\cdots
			\end{tikzcd}
		\end{center}
		\begin{proof}
			The left and the right squares commute since they commute at the level of cocycles. As for the middle square, let $c=\sum_{j=-k}^0 c_ju^j\in Z^*(C[u^{-1}])$, then 
			$B_0(u\cdot p(c))=B(c_{-1})$ and
			$i^+(u\cdot B_0(c))=B(c_0)u$.
			Since $c$ is a cocycle, $$0=(b+uB)(c)=\sum_{j=-k}^0 (b(c_j)+B(c_{j-1}))u^j\in C[u^{-1}].$$ In particular, $b(c_0)+B(c_{-1})=0$, so $B(c_0)u-B(c_{-1})=(b+uB)(c_0)$ is exact. This proves $B_{0*}\circ(u\cdot p_*)=(i^+_*\cdot u)\circ B_{0*}$.
		\end{proof}

	\end{lem}
	
	\begin{defn}Let $(C^*,b,B),(C^{\prime\prime*},b^{\prime\prime},B^{\prime\prime})$ be mixed cochain complexes.
		\begin{enumerate}
			\item A series of linear maps $\{f_i:C^*\to (C^{\prime\prime})^{*-2i}\}_{i\in\mathbb{Z}_{\geq0}}$ is called an \emph{$\infty$-morphism} from $C^*$ to $C^{\prime\prime*}$ if $\sum_{i\geq0}u^if_i:(C[[u,u^{-1}]^*,b+uB)\to (C^{\prime\prime}[[u,u^{-1}]^*,b^{\prime\prime}+uB^{\prime\prime})$ is a cochain map, or equivalently, if $\{f_i\}_{i\geq0}$ satisfies $b^{\prime\prime} f_0=f_0 b$ and
			$B^{\prime\prime} f_{i-1}+b^{\prime\prime} f_i=f_{i-1}B+f_ib$ ($i\geq1$).
			\item An $\infty$-morphism $f=\{f_i\}_{i\geq0}:C^*\to C^{\prime\prime*}$ is called an \emph{$\infty$-quasi-isomorphism} if $f_0:(C^*,b)\to (C^{\prime\prime*},b^{\prime\prime})$ is a cochain quasi-isomorphism.
			\item Given two $\infty$-morphisms $\{f_i\}_{i\geq0},\{g_i\}_{i\geq0}:C^*\to C^{\prime\prime*}$, a series of linear maps $\{h_i:C^*\to (C^{\prime\prime})^{*-2i-1}\}_{i\in\mathbb{Z}_{\geq0}}$ is called an \emph{$\infty$-homotopy} between them if $h=\sum_{i\geq0}u^ih_i:C[[u,u^{-1}]^*\to C^{\prime\prime}[[u,u^{-1}]^*$ is a $(b+uB,b^{\prime\prime}+uB^{\prime\prime})$-cochain homotopy between $\sum_{i\geq0}u^if_i$ and $\sum_{i\geq0}u^ig_i$, or equivalently, if $\{h_i\}_{i\geq0}$ satisfies $f_0-g_0=b^{\prime\prime} h_0+h_0 b$ and 
			$f_i-g_i=b^{\prime\prime} h_i+h_i b +B^{\prime\prime} h_{i-1}+h_{i-1}B$ ($i\geq1$).
		\end{enumerate}
	\end{defn}
	A morphism between mixed complexes is an $\infty$-morphism $\{f_i\}_{i\geq0}$ such that $f_i=0$ for all $i>0$, namely a single degree 0 linear map that commutes with both $b$ and $B$. A quasi-isomorphism between mixed complexes is a morphism that is also a $(b,b^{\prime\prime})$-quasi-isomorphism. A homotopy between two morphisms $f,g:(C^*,b,B)\to(C^{\prime\prime*},b^{\prime\prime},B^{\prime\prime})$ is an $\infty$-homotopy $\{h_i\}_{i\geq0}$ such that $h_i=0$ for all $i>0$, namely a single degree $-1$ linear map $h$ satisfying $f-g=b^{\prime\prime} h+hb$ and $B^{\prime\prime} h+hB=0$.
	
	The following important lemma goes back to \cite[Lemma 2.1]{Jones cyclic}, and is a special case of \cite[Lemma 2.3]{Zhao cyclic} which is stated for  $S^1$-complexes (an $\infty$-version of mixed complexes). The proof is a spectral sequence argument using the $u$-adic filtration on $C[[u]]^*$ etc.
	\begin{lem}\label{lem:mixed complex quasi}
		Let $\{f_i\}_{i\geq0}:(C^*,b,B)\to(C^{\prime\prime*},b^{\prime\prime},B^{\prime\prime})$ be an $\infty$-quasi-isomorphism. Then $\sum_{i\geq0}u^if_i$ induces isomorphisms on $\HC^*_{[[u]]}$, $\HC^*_{[[u,u^{-1}]}$ and $\HC^*_{[u^{-1}]}$.\qed
	\end{lem}
	The following lemma illustrates the naturality of the tautological exact sequence and Connes-Gysin exact sequences for cyclic homology, with respect to $\infty$-morphisms between mixed complexes.
	\begin{lem}\label{lem:cyclic long exact seq naturality}
		Let $f=\{f_i\}_{i\geq0}:(C^*,b,B)\to(C^{\prime\prime*},b^{\prime\prime},B^{\prime\prime})$ be an $\infty$-morphism. Then $f=\sum_i u^i f_i$ induces a morphism between the exact sequence \eqref{eqn:cyclic exact seq tautological} for $C$ and $C^{\prime\prime}$, namely there is a commutative diagram
		\begin{center}
			\begin{tikzcd}
				\cdots \HC^*_{[[u]]}(C)\arrow[r,"i_*"]\arrow[d,"f_*"] & \HC^*_{[[u,u^{-1}]}(C)\arrow[r,"u\cdot p_*"]\arrow[d,"f_*"] & \HC^{*+2}_{[u^{-1}]}(C)\arrow[r,"B_{0*}"]\arrow[d,"f_*"] & \HC^{*+1}_{[[u]]}(C)\arrow[d,"f_*"]  \cdots\\
				\cdots \HC^*_{[[u]]}(C^{\prime\prime})\arrow[r,"i_*^{\prime\prime}"] & \HC^*_{[[u,u^{-1}]}(C^{\prime\prime})\arrow[r,"u\cdot p_*^{\prime\prime}"] & \HC^{*+2}_{[u^{-1}]}(C^{\prime\prime})\arrow[r,"B_{0*}^{\prime\prime}"] & \HC^{*+1}_{[[u]]}(C^{\prime\prime}) \cdots.
			\end{tikzcd}
		\end{center}
		Similarly, for the exact sequence  \eqref{eqn:cyclic exact seq gysin [u-1]}, there is a commutative diagram
		\begin{center}
			\begin{tikzcd}
				\cdots  H^*(C,b) \arrow[r,"i_*"] \arrow[d,"f_{0*}"] & \HC_{[u^{-1}]}^*(C) \arrow[r,"u\cdot p_*"] \arrow[d,"f_*"] & \HC_{[u^{-1}]}^{*+2}(C)\arrow[r,"B_{0*}"]\arrow[d,"f_*"] & H^{*+1}(C,b)\arrow[d,"f_{0*}"]  \cdots\\
				\cdots H^*(C^{\prime\prime},b^{\prime\prime}) \arrow[r,"i_*^{\prime\prime}"] & \HC_{[u^{-1}]}^{*-1}(C^{\prime\prime}) \arrow[r,"u\cdot p_*^{\prime\prime}"] & \HC_{[u^{-1}]}^{*+1}(C^{\prime\prime})\arrow[r,"B_{0*}^{\prime\prime}"]& H^{*+1}(C^{\prime\prime},b^{\prime\prime})  \cdots.
			\end{tikzcd}
		\end{center}
		The case of the exact sequence \eqref{eqn:cyclic exact seq gysin [[u]]} is also similar.
	\end{lem}
	\begin{proof}We only write proof for the first diagram since the others are similar. The left and the middle squares commute since they commute at the level of cocycles. Now let $c=\sum_{j=0}^k c_{-j}u^{-j}\in Z^{*+2}(C[u^{-1}])$. Then, $(b+uB)(c)=0$ says $b(c_{-j})+B(c_{-j-1})=0$ for all $j\in\{0,\dots,k\}$. Also recall the $\infty$-morphism $f$ satisfies $B^{\prime\prime} f_{i-1}+b^{\prime\prime}f_i=f_{i-1}B+f_ib$. Using these relations, it is a straightforward computation to see
		\begin{equation*}
			\sum_{i\geq0}f_i(B(c_0))\cdot u^i-B^{\prime\prime}\Big(\sum_{0\leq j\leq k}f_j(c_{-j})\Big)=(b^{\prime\prime}+uB^{\prime\prime})\Big(\sum_{i\geq0}\sum_{0\leq j\leq k}f_{i+j+1}(c_{-j})\cdot u^i\Big).
		\end{equation*}
		The left-hand side is $(f\circ B_0-B_0^{\prime\prime}\circ f)(c)$, and the right-hand side is exact, so commutativity of the right square is proved. 
	\end{proof}
	
	We now discuss some important examples of mixed (co)chain complexes and their cyclic homologies. Recall that a \emph{cosimplicial object} in some category is a sequence of objects $C(k)$ $(k\in\mathbb{Z}_{\geq0})$ together with morphisms 
	\begin{equation*}
		\delta_i:C(k-1)\to C(k)\hspace{.16667em}(0\leq i\leq k),\hspace{.33333em}\sigma_{i}:C(k+1)\to C(k)\hspace{.16667em}(0\leq i\leq k)
	\end{equation*}
	satisfying the following relations:
	\begin{align*}
		\delta_j\delta_i&=\delta_i\delta_{j-1}\hspace{.16667em}(i<j);\\
		\sigma_j\sigma_i&=\sigma_i\sigma_{j+1}\hspace{.16667em}(i\leq j);\\
		\sigma_j\delta_i&=\begin{cases}\delta_i\sigma_{j-1} & (i<j)\\
			\mathrm{id}&(i=j,j+1)\\
			\delta_{i-1}\sigma_j&(i>j+1).
		\end{cases}
	\end{align*}
	A \emph{cocyclic object} is a cosimplicial object $\{C(k)\}_k$ together with morphisms $\tau_k:C(k)\to C(k)$ satisfying the following relations:
	\begin{align*}
		\tau_k^{k+1}&=\mathrm{id};\\
		\tau_k\delta_i&=\delta_{i-1}\tau_{k-1}\hspace{.16667em}(1\leq i\leq k),\hspace{.66667em}\tau_k\delta_0=\delta_k;\\
		\tau_k\sigma_i&=\sigma_{i-1}\tau_{k+1}\hspace{.16667em}(1\leq i\leq k),\hspace{.66667em}\tau_k\sigma_0=\sigma_k\tau_{k+1}^2.
	\end{align*}
	For example, let $\Delta^0:=\mathbb{R}^0$, $\Delta^k:=\{(t_1,\dots,t_k)\in\mathbb{R}^k\mid0\leq t_1\leq\dots\leq t_k\leq1\}$ $(k>0)$ be the standard simplices, then $\{\Delta^k\}_{k\in\mathbb{Z}_{\geq0}}$ is a cocyclic set (topological space, etc.) with standard cocyclic maps $\delta_i:\Delta^{k-1}\to \Delta^k$, $\sigma_i:\Delta^{k+1}\to \Delta^k$, $\tau_k:\Delta^k \to \Delta^k$ defined by
	\begin{subequations}\label{eqn:cocyclic simplex}
		\begin{align}
			\delta_i(t_1,\dots,t_{k-1})&:=
			\begin{cases}
				(0,t_1,\dots,t_{k-1}) & (i=0)\\
				(t_1,\dots,t_i,t_i,\dots,t_{k-1}) & (1\leq i\leq k-1)\\
				(t_1,\dots,t_{k-1},1) & (i=k),
			\end{cases}\\
			\sigma_i(t_1,\dots,t_{k+1})&:=
			(t_1,\dots,\widehat{t_{i+1}},\dots,t_{k+1})\hspace{.16667em}(\hbox{miss }t_{i+1})\hspace{.16667em}(0\leq i\leq k),\\
			\tau_k(t_1,\dots,t_k)&:=
			(t_2-t_1,\dots,t_k-t_1,1-t_1).
		\end{align}
	\end{subequations}
	\begin{rem}
		Equivalently, for $\tilde{\Delta}^k:=\{(s_0,s_1,\dots,s_k)\in[0,1]^{k+1}\mid s_0+s_1\cdots+s_k=1\}$ ($k\geq0$), $\tau_k:\tilde{\Delta}^k\to\tilde{\Delta}^k$ reads $\tau_k(s_0,s_1,\dots,s_k)=(s_1,\dots,s_k,s_0)$.
	\end{rem}
	
	\begin{exmp}[Cocyclic complex and Connes' version of cyclic cohomology]\label{exmp:cocyclic complex} 
		Consider the category of cochain complexes where the morphisms are degree 0 cochain maps. Let $\big((C(k)^*,d),\delta_i,\sigma_i,\tau_k\big)$ be a cocyclic cochain complex, then a mixed cochain complex $(C,b,B)$ is obtained as follows. Let 
		\begin{equation}\label{eqn:cocyclic complex delta defn}
			\delta:C(k-1)^*\to C(k)^*;\hspace{.33333em}c_{k-1}\mapsto(-1)^{|c_{k-1}|+k-1}\sum_{0\leq i\leq k}(-1)^i\delta_i(c_{k-1}),
		\end{equation}
		then $\delta^2=0$, $\delta d+d\delta=0$. Let $(C^*,b)$ be the \emph{product} total complex of the double complex $\big(C(k)^l,d,\delta\big)^{k\in\mathbb{Z}_{\geq0}}_{l\in\mathbb{Z}}$:
		\begin{equation*}
			C^*:=\prod_{l+k=*}C(k)^l=\prod_{k\geq 0} C(k)^{*-k},\hspace{.66667em}b=d+\delta.
		\end{equation*}
		For later purpose we also introduce the \emph{normalized} subcomplex $(C_{\mathrm{nm}}^*,b)$ of $(C^*,b)$:
		\begin{equation*}
			C_{\mathrm{nm}}^*:=\prod_{k\geq0}C_{\mathrm{nm}}(k)^{*-k},\hspace{.5em}C_{\mathrm{nm}}(k):=\bigcap_{0\leq i\leq k-1}\mathrm{Ker}\big(\sigma_i:C(k)\to C(k-1)\big).
		\end{equation*}
		Note that the natural inclusion $\{C_{\mathrm{nm}}(k)\}\subset \{C(k)\}$ is not cosimplicial since $\delta_j$ does not restricts to $C_{\mathrm{nm}}(k)$. The natural inclusion $(C_{\mathrm{nm}}^*,b)\hookrightarrow(C^*,b)$ is a quasi-isomorphism (see \cite[Proposition 1.6.5]{Loday cyclic book} or \cite[Lemma 2.5]{Irie loop}). Next, define the operator $B:C^*\to C^{*-1}$ by
		\begin{equation*}
			B:=Ns(1-\lambda)\hspace{.33333em}(\hbox{\emph{Connes' operator}}),
		\end{equation*}
		where $\lambda,N,s$ are given by (here $|c|$ is the degree of $c=(c_k)_{k\geq0}\in\prod_{k\geq0}C(k)$ in $C^*$)
		\begin{equation*}
			\lambda|_{C(k)}:=(-1)^k\tau_k,\hspace{.16667em} N|_{C(k)}:=1+\lambda+\dots+\lambda^k,\hspace{.16667em} s(c):=(-1)^{|c|-1}(\sigma_{k}\tau_{k+1}(c_{k+1}))_{k\geq0}.
		\end{equation*}
		Although $C_{\mathrm{nm}}$ is not closed under $\lambda$, $N$, it is closed under $s$, $B$. For  $c_{k+1}\in C_{\mathrm{nm}}(k+1)^*$, there holds
		$$
		s(\lambda(c_{k+1}))=(-1)^{|c_{k+1}|}\sigma_k\tau_{k+1}^2(c_{k+1})=(-1)^{|c_{k+1}|}\tau_k\sigma_0(c_{k+1})=0,
		$$
		so Connes' operator $B$ has simpler form on normalized subcomplex:
		\begin{equation*}
			B|_{C_{\mathrm{nm}}}=Ns.
		\end{equation*}
		To see $(C^*,b,B)$ is a mixed complex, define 
		\begin{equation*}b^\prime:C^*\to C^{*+1},\hspace{.5em}c\mapsto b(c)-((-1)^{|c_{k-1}|-1}\delta_k(c_{k-1}))_{k\geq0}.\end{equation*}
		It is a routine calculation to see $(b^\prime)^2=0$,
		$N(1-\lambda)=(1-\lambda)N=0$, $(1-\lambda)b=b^\prime(1-\lambda)$, $bN=Nb^\prime$ and $b^\prime s+sb^\prime=1$.
		It follows that $B^2=Ns((1-\lambda)N)s(1-\lambda)=0$ and $bB+Bb=Nb^\prime s(1-\lambda)+Nsb^\prime(1-\lambda)=N(1-\lambda)=0$, as desired. The identity $(1-\lambda)b=b^\prime(1-\lambda)$ also implies that the space of \emph{cyclic invariants}, \begin{equation*}C_{\mathrm{cyc}}:=\mathrm{Ker}(1-\lambda)\subset(C,b),\end{equation*} forms a subcomplex (we denote this inclusion by $i_\lambda$). This leads to \emph{Connes' version of cyclic cohomology} of the cocyclic cochain complex,
		\begin{equation*}\HC_{\lambda}^*(C)=\HC_{\lambda}^*\big(C(k),d,\delta_i,\sigma_i,\tau_k\big):=H^*(C_{\mathrm{cyc}},b).\end{equation*}
		Since $B=Ns(1-\lambda)$ vanishes on $C_{\mathrm{cyc}}$, $(C_{\mathrm{cyc}}^*,b)$ is also naturally a subcomplex of $(C[[u]]^*,b+uB)$. By an argument similar to \cite[Theorem 2.1.5, 2.1.8]{Loday cyclic book} 
		one sees that this inclusion $I_{\lambda}:(C_{\mathrm{cyc}}^*,b)\hookrightarrow(C[[u]]^*,b+uB)$ 
		induces an isomorphism
		\begin{equation}\label{eqn:cyclic Connes isom}
			I_{\lambda*}:\HC_{\lambda}^*(C)\cong \HC_{[[u]]}^*(C).
		\end{equation}
		The short exact sequence $0\to(C_{\mathrm{cyc}},b)\xrightarrow{i_\lambda} (C,b)\xrightarrow{p_\lambda} (C/C_{\mathrm{cyc}},b)\to 0$ induces \emph{Connes' long exact sequence} (we follow the presentation of \cite[Section 3.7]{K})
		\begin{equation}\label{eqn:cyclic exact seq Connes}
			\cdots\to \HC_\lambda^*(C)\xrightarrow{i_{\lambda*}} H^*(C,b)\xrightarrow{B_\lambda}\HC_\lambda^{*-1}(C)\xrightarrow{S_\lambda}\HC_\lambda^{*+1}(C)\to\cdots.
		\end{equation}
		Here we have made use of an isomorphism $\HC_\lambda^{*-1}(C)\cong H^*(C/C_{\mathrm{cyc}},b)$, which is a consequence of another short exact sequence 
		\begin{equation}\label{eqn:Connes short exact seq acyclic}
			0\to (C/C_{\mathrm{cyc}},b)\xrightarrow{1-\lambda} (C,b^\prime)\xrightarrow{N}(C_{\mathrm{cyc}},b)\to 0
		\end{equation} and the fact that $ (C,b^\prime)$ is acyclic (since $b^\prime s+sb^\prime=1$). Lemma \ref{lem:Connes exact seq coincide} below says \eqref{eqn:cyclic exact seq Connes} can be identified with \eqref{eqn:cyclic exact seq gysin [[u]]}.
		Finally we mention that $\HC_{[[u]]}^*(C_{\mathrm{nm}})\cong \HC_{[[u]]}^*(C)\cong \HC_{\lambda}^*(C)$, where the first isomorphism follows from Lemma \ref{lem:mixed complex quasi}.
	\end{exmp}
	A subexample of Example \ref{exmp:cocyclic complex} is as follows.
	\begin{exmp}[Cyclic cohomology of dg algebras]\label{exmp:cyclic cohomology of algebra}
		Let $A^*$ be a dg algebra with unit $1_A$. Then $\{\Hom^*(A^{\ot k+1},\mathbb{R})\}_{k\geq0}$ has the structure of a cocyclic cochain complex, where $\delta_i:\Hom^{*}(A^{\ot k},\mathbb{R})\to\Hom^{*}(A^{\ot k+1},\mathbb{R})$, $\sigma_i:\Hom^{*}(A^{\ot k+2},\mathbb{R})\to\Hom^{*}(A^{\ot k+1},\mathbb{R})$ and $\tau_k:\Hom^{*}(A^{\ot k+1},\mathbb{R})\to\Hom^{*}(A^{\ot k+1},\mathbb{R})$ are 
		\begin{align}\nonumber
			\delta_i(\vphi)(a_1\ot\cdots\ot a_{k+1})&:=\begin{cases}
				(-1)^{\varepsilon}\vphi(a_2\ot\cdots\ot a_k\ot a_{k+1}a_1)& (i=0)\\
				\vphi(a_1\ot\cdots\ot a_i a_{i+1}\ot\cdots\ot a_{k+1})& (1\leq i\leq k),
			\end{cases}\\\nonumber
			\sigma_i(\vphi)(a_1\ot\cdots\ot a_{k+1})&:=\vphi(a_1\ot\cdots\ot a_i\ot 1_A\ot a_{i+1}\ot\cdots\ot a_{k+1})\hspace{.33333em}(0\leq i\leq k),\\\label{eqn:cyclic permutation on End}
			\tau_k(\vphi)(a_1\ot\cdots\ot a_{k+1})&:=(-1)^{\varepsilon}\vphi(a_{k+1}\ot a_1\cdots\ot a_k).
		\end{align}
		The associated mixed total complex is denoted by $\CH^*(A,A^\vee)$. For simplicity, denote cyclic homologies of $\CH^*(A,A^\vee)$ by $\HC_{[u^{-1}]}^*(A,A^\vee)$, $\HC_{[[u]]}^*(A,A^\vee)\cong \HC_\lambda^*(A,A^\vee)$, etc. Classically, $\HC_\lambda^*(A,A^\vee)$ is called \emph{(Connes') cyclic cohomology} of $A$. 
		
		Let us also recall that for any dg $A$-bimodule $M^*$, there is a structure of a cosimplicial complex on $\{\Hom^*(A^{\ot k},M)\}_{k\geq0}$, where $\delta_i:\Hom^*(A^{\ot k-1},M)\to\Hom^*(A^{\ot k},M)$, $\sigma_i:\Hom^*(A^{\ot k+1},M)\to\Hom^*(A^{\ot k},M)$ are
		\begin{align*}
			\delta_i(f)(a_1\ot\cdots\ot a_k)&:=\begin{cases}
				(-1)^{\varepsilon}a_1\boldsymbol{\cdot}f(a_2\ot\cdots\ot a_k)& (i=0)\\
				f(a_1\ot\cdots\ot a_i a_{i+1}\ot\cdots\ot a_k)& (1\leq i\leq k-1)\\
				f(a_1\ot\cdots\ot a_{k-1})\boldsymbol{\cdot}a_k&(i=k),
			\end{cases}\\
			\sigma_i(f)(a_1\ot\cdots\ot a_k)&:=f(a_1\ot\cdots\ot a_i\ot 1_A\ot a_{i+1}\ot\cdots\ot\cdots\ot a_{k})\hspace{.33333em}(0\leq i\leq k).
		\end{align*}
		The associated total complex, denoted by $\CH^*(A,M)$, is called the \emph{Hochschild cochain complex}, whose cohomology group, denoted by $\mathrm{HH}^*(A,M)$, is called the \emph{Hochschild cohomology}. Taking $M^*=(A^\vee)^*=\Hom^*(A,\mathbb{R})$ with dg $A$-bimodule structure satisfying
		\begin{equation}\label{eqn:A dual bimodule}
			(d\vphi)(a)+(-1)^{|\vphi|}\vphi(da)=0,\hspace{.5em}\vphi (ab)=(-1)^{(|a|+|\vphi|)|b|}(b\boldsymbol{\cdot}\vphi)(a)=(\vphi\boldsymbol{\cdot} a)(b),
		\end{equation}	
		one sees that the cosimplicial structure on $\{\Hom^*(A^{\ot k},A^\vee)\}_k$ is the same as that on $\{\Hom^*(A^{\ot k+1},\mathbb{R})\}$ described previously, in view of the natural isomorphism
		\begin{equation*}
			\Hom^*(A^{\ot k},A^\vee)\cong\Hom^{*}(A^{\ot k}\ot A,\mathbb{R})=\Hom^{*}(A^{\ot k+1},\mathbb{R})
		\end{equation*}
		from $\Hom$-$\ot$ adjunction. See Example \ref{exmp:end operad} for further discussion.
	\end{exmp}
	\begin{rem}
		We shall use the name ``Connes' version of cyclic \emph{co}homology'' for ``\emph{co}cyclic complex'', even if we work with chain complexes rather than cochain complexes. For a cocyclic chain complex $\big((C(k)_*,\partial),\delta_j,\sigma_i,\tau_k\big)$, Connes' version of cyclic cohomology is $\HC^{\lambda}_*(C):=H_*(C^{\mathrm{cyc}},b)$ where $C^{\mathrm{cyc}}_*:=\mathrm{Ker}(1-\lambda)\subset C_*$, and $\HC^{\lambda}_*(C)$ is isomorphic to $\HC_*^{[[v]]}$ of the mixed chain complex $(C_*,b,B)=\big(\prod_{k\geq0}C(k)_{*+k},\partial+\delta,Ns(1-\lambda)\big)$.
	\end{rem}
	
	\begin{lem}\label{lem:Connes exact seq coincide}
		In the situation of Example \ref{exmp:cocyclic complex}, the isomorphism \eqref{eqn:cyclic Connes isom} and the long exact sequences \eqref{eqn:cyclic exact seq Connes} and \eqref{eqn:cyclic exact seq gysin [[u]]} fit into the following commutative diagram:
		\begin{center}
			\begin{tikzcd}
				\cdots   \HC_\lambda^{*}(C)\arrow[r,"i_{\lambda*}"]\arrow[d,"I_{\lambda*}","\cong"']& H^{*}(C,b)\arrow[r,"B_\lambda"]\arrow[d,"\mathrm{id}",equal] &  \HC_\lambda^{*-1}(C)\arrow[r,"-S_\lambda"]\arrow[d,"I_{\lambda*}","\cong"'] & \HC_\lambda^{*+1}(C)\arrow[d,"I_{\lambda*}","\cong"']  \cdots\\
				\cdots  \HC_{[[u]]}^{*}(C)\arrow[r,"p_{0*}"]& H^{*}(C,b)\arrow[r,"B_*"] & \HC_{[[u]]}^{*-1}(C) \arrow[r,"i^+_*\cdot u"] & \HC_{[[u]]}^{*+1}(C)\cdots
			\end{tikzcd}
		\end{center}
	\end{lem}
	\begin{proof}
		The left square commutes since it commutes at cochain level. To verify commutativity of the other two squares, we need explicit formulas of $B_\lambda$ and $S_\lambda$. Since $(1-\lambda)b=b^\prime(1-\lambda)$, 
		there is a cochain isomorphism $$(C/C_{\mathrm{cyc}},b)\underset{\cong}{\xrightarrow{1-\lambda}}(\mathrm{Im}(1-\lambda),b^\prime).$$ Also note that $C_{\mathrm{cyc}}=\mathrm{Im}N$ and $N|_{C_{\mathrm{cyc}}}=((k+1)\mathrm{id}_{\mathrm{Ker}(1-(-1)^k\tau_k)})_{k\geq0}:C_{\mathrm{cyc}}\to C_{\mathrm{cyc}}$ is a linear isomorphism. By definition, $B_\lambda$  is the composition \begin{equation*}H^*(C,b)\xrightarrow{p_{\lambda*}}H^*(C/C_{\mathrm{cyc}},b)\underset{\cong}{\xrightarrow{1-\lambda}}H^*(\mathrm{Im}(1-\lambda),b^\prime)\underset{\cong}{\xrightarrow{Q_{\lambda*}^{-1}}} \HC_\lambda^{*-1}(C),\end{equation*} where by examining \eqref{eqn:Connes short exact seq acyclic}, $Q_{\lambda*}:\HC_\lambda^{*-1}(C)\xrightarrow{\cong}H^*(\mathrm{Im}(1-\lambda),b^\prime)$ is given on cocycles by
		\begin{equation*}
			Q_\lambda: Z^{*-1}(C_{\mathrm{cyc}},b)\to Z^*(\mathrm{Im}(1-\lambda),b^\prime);\hspace{.5em}x\mapsto (b^\prime\circ(N|_{C_{\mathrm{cyc}}})^{-1})(x).
		\end{equation*}
		Let us calculate that on $Z(C,b)$, 
		\begin{equation}\label{eqn:Connes long exact seq chase}
			Q_\lambda B=Q_\lambda Ns(1-\lambda)=b^\prime s(1-\lambda)=(1-sb^\prime)(1-\lambda)=1-\lambda-s(1-\lambda)b=1-\lambda.
		\end{equation}
		Thus $(B|_{C\to C_{\mathrm{cyc}}})_*=(Q_{\lambda*})^{-1}\circ(1-\lambda)\circ p_{\lambda*}=B_\lambda$, which says the middle square commutes. Similarly, $S_\lambda$ is the composition \begin{equation*}\HC_\lambda^{*-1}(C)\underset{\cong}{\xrightarrow{Q_{\lambda*}}}H^*(\mathrm{Im}(1-\lambda),b^\prime)\underset{\cong}{\xrightarrow{(1-\lambda)^{-1}}} H^*(C/C_{\mathrm{cyc}},b)\xrightarrow{R_{\lambda*}} \HC_\lambda^{*+1}(C),\end{equation*} where $R_{\lambda*}:H^*(C/C_{\mathrm{cyc}},b)\to \HC_\lambda^{*+1}(C)$ is induced by the map
		\begin{equation*}
			R_\lambda:\{y\in C^*\mid b(y)\in C_{\mathrm{cyc}}\}\to Z^{*+1}(C_{\mathrm{cyc}},b);\hspace{.5em}y\mapsto b(y).
		\end{equation*}
		\eqref{eqn:Connes long exact seq chase} also holds on $Z(C/C_{\mathrm{cyc}},b)$, and implies $(1-\lambda)^{-1}Q_{\lambda*}=(B|_{C/C_{\mathrm{cyc}}\to C_{\mathrm{cyc}}})_*^{-1}$.
		
		Therefore for $x\in Z^{*-1}(C_{\mathrm{cyc}},b)$, \begin{equation*}S_\lambda([x])=[b(y)]\in \HC_\lambda^{*+1}(C),\end{equation*}  where $y$ is any choice of elements in $C^*$ satisfying $B(y)=x$ and $(1-\lambda)(b(y))=0$. For such $x$ and $y$, $-b(y)-x\cdot u=(b+uB)(-y)$ is exact in $C^{*+1}[[u]]$, so $I_{\lambda*}\circ(-S_\lambda)=(i^+_*\cdot u)\circ I_{\lambda*}$, i.e. the right square commutes.
	\end{proof}

	\begin{exmp}[$S^1$-equivariant homology theories \cite{Jones cyclic}]\label{exmp:mixed complex equivariant} Let $X$ be a topological $S^1$-space, namely a topological space with a continuous $S^1$-action $F^X:S^1\times X\to X$. Let $(C_*,b)=(S_*(X),\partial)$ be the singular chain complex of $X$, and define the rotation operator $B=J:S_*(X)\to S_{*+1}(X)$ by
		
		\begin{equation}\label{eqn:s1 action J}
			J(a):=F^X_*([S^1]\times a),\hspace{.5em}a\in S_*(X).
		\end{equation}
		Here $[S^1]\in S_1(S^1)$ is the fundamental cycle of $S^1$, namely $[S^1]=\pi^{\Delta^1}_{S^1}:\Delta^1=[0,1]\to\mathbb{R}/\mathbb{Z}=S^1$, and $\times$ is the simplicial cross product induced by standard decomposition of $\Delta^l\times\Delta^k$ into $(k+l)$-simplices (see \cite[page 278]{Hatcher}). 
		Then $\partial J+J\partial=0$ since $$\partial ([S^1]\times a)=\partial[S^1]\times a+(-1)^{\deg[S^1]}[S^1]\times \partial a=-[S^1]\times \partial a.$$ To see $J^2=0$, let us write down the cross product with $[S^1]$ explicitly. For $k\in\mathbb{Z}_{\geq0}$ and $j\in\{0,\dots,k\}$, consider the embeddings $\iota_{k,j}:\Delta^{k+1}\to\Delta^1\times\Delta^k$ defined by
		\begin{equation*}
			\iota_{k,j}(t_1,\dots,t_{k+1}):=(t_{j+1},(t_1,\dots,t_j,t_{j+2},\dots,t_{k+1})),
		\end{equation*}
		then for $(\sigma:\Delta^k\to X)\in S_k(X)$,
		\begin{equation*}
			[S^1]\times\sigma=\sum_{0\leq j\leq k}(-1)^j(\pi^{\Delta^1}_{S^1}\times\sigma)\circ\iota_{k,j}\in S_{k+1}(S^1\times X).
		\end{equation*}
		Let $F^{S^1}:S^1\times S^1\to S^1$, $([t],[t^\prime])\mapsto[t+t^\prime]$ be the rotation $S^1$-action on $S^1$, then
		\begin{equation*}
			F^{S^1}_*([S^1]\times[S^1])=F^{S^1}\circ(\pi^{\Delta^1}_{S^1}\times\pi^{\Delta^1}_{S^1})\circ\iota_{1,0}-F^{S^1}\circ(\pi^{\Delta^1}_{S^1}\times\pi^{\Delta^1}_{S^1})\circ\iota_{1,1}=0.
		\end{equation*}
		From the commutative diagram
		\begin{center}
			\begin{tikzcd}
				S^1\times S^1\times X \arrow[d, "\mathrm{id}_{ S^1}\times F^X"] \arrow[rr, "F^{ S^1}\times\mathrm{id}_X"] & &  S^1\times X \arrow[d, "F^X"]\\
				S^1\times X \arrow[rr, "F^X"]  & & X
			\end{tikzcd}\hspace{.33333em}
			\begin{tikzcd}
				([t],[t^\prime],x) \arrow[d, mapsto] \arrow[r, mapsto]  & ([t+t^\prime],x) \arrow[d, mapsto]\\
				([t], F^X([t^\prime],x)) \arrow[r, mapsto]  &  F^X([t+t^\prime],x)
			\end{tikzcd}
		\end{center}
		we conclude that for any $a\in S_*(X)$,
		\begin{equation*}
			J^2(a)=F^X_*([S^1]\times F^X_*([S^1]\times a))=F^X_*(F^{S^1}_*([S^1]\times[S^1])\times a))=0.
		\end{equation*}
		For the mixed chain complex $(C_*,b,B)=(S_*(X),\partial,J)$, there is a natural isomorphism (\cite[Lemma 5.1]{Jones cyclic}) 
		\begin{equation*}\HC_*^{[v^{-1}]}(S(X))\cong H^{S^1}_*(X):=H_*(X\times_{S^1}E S^1),\end{equation*}
		namely $\HC_*^{[v^{-1}]}(S(X))$ is isomorphic to the $S^1$-equivariant homology of $X$, i.e. homology of the homotopy quotient (Borel construction).
		The other two cyclic homology groups of $(S_*(X),\partial,J)$ are called the \emph{negative} and \emph{periodic $S^1$-equivariant homology} of $X$, and are denoted by 
		\begin{equation*}G^{S^1}_*(X):=\HC_*^{[[v]]}(S(X)),\hspace{.5em}\widehat{H}^{S^1}_*(X):=\HC_*^{[[v,v^{-1}]}(S(X)),\end{equation*} 
		respectively.
		The tautological exact sequence \eqref{eqn:cyclic exact seq tautological} translates into
		\begin{equation}\label{eqn:s1 equivariant tautological}
			\cdots\to G^{S^1}_*(X)\to \wh{H}^{S^1}_*(X)\to H^{S^1}_{*-2}(X)\to G^{S^1}_{*-1}(X)\to\cdots,
		\end{equation}
		and the Connes-Gysin exact sequences \eqref{eqn:cyclic exact seq gysin} translate into
		\begin{subequations}
			\begin{align}\label{eqn:s1 equivariant gysin}
				\cdots\to H_*(X)\to H^{S^1}_*(X)\to H^{S^1}_{*-2}(X)\to H_{*-1}(X)\to\cdots\\\label{eqn:s1 equivariant negative gysin}
				\cdots\to G^{S^1}_{*+2}(X)\to G^{S^1}_*(X)\to H_*(X)\to G^{S^1}_{*+1}(X)\to\cdots.
			\end{align}
		\end{subequations}
		We end this example by mentioning that \eqref{eqn:s1 equivariant gysin} coincides with the Gysin sequence associated to the $S^1$-fibration $X\times E S^1\to X\times_{S^1}E S^1$. 
	\end{exmp}
	\begin{rem}There seems to be no interpretation of $G^{S^1}_*(X)$ and $\wh{H}^{S^1}_*(X)$ as homology groups of some spaces naturally associated to $X$, but there are homotopy theoretic interpretations. For example, when $X$ is a (finite) $S^1$-CW complex, \cite[Lemma 4.4]{ACD} says $G_*^{S^1}(X)$ is naturally isomorphic to the homotopy groups of the homotopy fixed point spectrum $(\mathbf{H}\wedge X_+)^{h S^1}$, where $\mathbf{H}$ is the Eilenberg-MacLane spectrum $\{K(\mathbb{Z},n)\}$.
	\end{rem}
	
	\section{A cocyclic complex and an $\infty$-quasi-isomorphism}\label{section:A cocyclic complex and an infty-quasi-isomorphism}
	Let $X$ be a topological space with $S^1$-action $F^X: S^1\times X\to X$. There is a cocyclic structure on $\{ X\times\Delta^k\}_{k\in\mathbb{Z}_{\geq0}}$, where $\delta_i^{ X\times\Delta^k}:=\mathrm{id}_X\times \delta_i^{\Delta^k}$, $\sigma_i^{ X\times\Delta^k}:=\mathrm{id}_X\times \sigma_i^{\Delta^k}$ ($\delta_i^{\Delta^k}$, $\sigma_i^{\Delta^k}$ are as in \eqref{eqn:cocyclic simplex}), and $\tau_k^{ X\times\Delta^k}: X\times\Delta^k\to  X\times\Delta^k$ is defined by
	\begin{equation*}
		\tau_k^{ X\times\Delta^k}(x,t_1,\dots,t_k):=(F^X([t_1],x),t_2-t_1,\dots,t_k-t_1,1-t_1).\end{equation*} 
	Taking singular chains of the cocyclic space $\{ X\times\Delta^k\}_{k\geq0}$ yields a cocyclic chain complex $\{S_*( X\times\Delta^k)\}_{k\geq0}$. Let us denote the associated mixed complex by
	\begin{equation*}
		(S^{ X\Delta}_*,b,B):=\Big(\prod_{k\geq0}S_{*+k}( X\times\Delta^k),\partial+\delta,Ns(1-\lambda)\Big).
	\end{equation*}
	The $S^1$-action on $X$ extends to $ X\times\Delta^k$ where the $S^1$-action on $\Delta^k$ is trivial, and then the rotation operator $J:S_*(X)\to S_{*+1}(X)$ defined in Example \ref{exmp:mixed complex equivariant} extends component-by-component to $S^{ X\Delta}_*$ by
	\begin{equation*}
		J:S^{ X\Delta}_*\to S^{ X\Delta}_{*+1},\hspace{.5em}(x_k)_{k\geq0}\mapsto (J(x_k))_{k\geq0}.
	\end{equation*}
	By Example \ref{exmp:mixed complex equivariant}, $J^2=0$ and $\partial J+J\partial=0$. Since $S^1$ acts trivially on $\Delta^k$, $J$ commutes with $\delta_i,\sigma_i$. It follows that $\delta J+J\delta=0$ and $J(S^{ X\Delta,\mathrm{nm}})\subset S^{ X\Delta,\mathrm{nm}}$, so $(S^{ X\Delta}_*,b,J)$,  $(S^{ X\Delta,\mathrm{nm}}_*,b,J)$ are also mixed complexes. $J$ also commutes with $\tau_k^{ X\times\Delta^k}$ because of the commutative diagram
	\begin{equation}\label{eqn:tau commute with J}\begin{tikzcd}
			S^1\times X\times\Delta^k \arrow[d, "F^{ X\times\Delta^k}"] \arrow[rr, "\mathrm{id}_{S^1}\times\tau_k^{ X\times\Delta^k}"] & &   S^1\times X\times\Delta^k\arrow[d, "F^{ X\times\Delta^k}"] \\
			X\times\Delta^k \arrow[rr, "\tau_k^{ X\times\Delta^k}"] & &  X\times\Delta^k
	\end{tikzcd}\end{equation}
	\begin{center}\begin{tikzcd}
			([t],x,t_1,\dots,t_k) \arrow[d, mapsto] \arrow[r, mapsto]  & ([t],F^X([t_1],x),t_2-t_1,\dots,1-t_1)\arrow[d, mapsto]\\
			(F^X([t],x),t_1,\dots,t_k) \arrow[r, mapsto] &  (F^X([t+t_1],x),t_2-t_1,\dots,1-t_1),
	\end{tikzcd}\end{center}
	so $JB+BJ=0$.
	We will analyze the relationship between the mixed complexes
	\begin{equation*}
		(S^{ X\Delta}_*,b,B),\hspace{.33333em}(S^{ X\Delta}_*,b,J),\hspace{.33333em}(S_*(X),\partial,J).
	\end{equation*}
	If there is no risk of confusion, we shall write $\delta_i^{ X\times\Delta^k},\sigma_i^{ X\times\Delta^k},\tau_k^{ X\times\Delta^k}$ and the induced maps on singular chain complexes as $\delta_i,\sigma_i,\tau_k$ for short. Note that $\delta_i,\sigma_i$ do not involve $S^1$-action, so if we forget the $S^1$-action, there is still a total complex $(S^{ X\Delta}_*,b=\partial+\delta)$ from the cosimplicial chain complex $\{S_*( X\times\Delta^k)\}_{k\in\mathbb{Z}_{\geq0}}$.
	
	Let us state the main theorem of this section.
	\begin{thm}\label{thm:cocyclic topological isom}
		Let $X$ be a topological $S^1$-space. Then for both of the mixed complex structures $(b,B)$ and $(b,J)$ on $S_*^{ X\Delta}=\prod_{k\geq0}S_{*+k}( X\times\Delta^k)$, there are natural isomorphisms
		\begin{align*}
			\HC_*^{[v^{-1}]}(S^{ X\Delta})&\cong H^{S^1}_*(X)\hspace{.5em}\hbox{as $\mathbb{R}[v^{-1}]$-modules},\\ \HC_*^{[[v,v^{-1}]}(S^{ X\Delta})&\cong \widehat{H}^{S^1}_*(X)\hspace{.5em}\hbox{as $\mathbb{R}[[v,v^{-1}]$-modules},\\ \HC_*^{[[v]]}(S^{ X\Delta})&\cong G^{S^1}_*(X)\hspace{.5em}\hbox{as $\mathbb{R}[[v]]$-modules}.
		\end{align*}
		Furthermore, these isomorphisms throw the (tautological and Connes-Gysin) exact sequences \eqref{eqn:cyclic exact seq tautological}\eqref{eqn:cyclic exact seq gysin} for cyclic homology theories onto the (tautological and Gysin) exact sequences for $S^1$-equivariant homology theories.
	\end{thm}
	\begin{proof}The statement about isomorphisms is a consequence of Lemma \ref{lem:mixed complex quasi}, Corollary \ref{cor:cocyclic singular J quasi} and Proposition \ref{prop:cocyclic exists mixed infty} below. The statement about long exact sequences is then a consequence of Lemma \ref{lem:cyclic long exact seq naturality}.
	\end{proof}
	\begin{cor}\label{cor:cocyclic topological isom connes}For any topological $S^1$-space $X$, Connes' version of cyclic cohomology of the cocyclic chain complex $\{S_*( X\times\Delta^k)\}_{k\in\mathbb{Z}_{\geq0}}$ is naturally isomorphic to the negative $S^1$-equivariant homology of $X$. \qed
	\end{cor}
	\begin{lem}\label{lem:cocyclic singular quasi}
		For any topological space X, the projection chain map
		\begin{equation*}\pr_0:(S^{ X\Delta}_*,b)\to(S_*(X),\partial);\hspace{.5em}(c_k)_{k\geq0}\mapsto c_0\end{equation*} is a quasi-isomorphism.
	\end{lem}
	\begin{proof}
		Since $\pr_0$ is surjective, it suffices to prove $\mathrm{Ker}(\pr_0)_*=\prod_{k\geq1} S_{*+k}( X\times\Delta^k)$ is $b$-acyclic. Let us write $\tilde{S}_*:=\mathrm{Ker}(\pr_0)_*$ and consider the decreasing filtration $\mathcal{F}_p$ $(p\in\mathbb{Z}_{\geq1})$ on $\tilde{S}$ defined by $\mathcal{F}_p\tilde{S}_*:=\prod_{k\geq p} S_{*+k}( X\times\Delta^k)$. The $E_1$-page of the spectral sequence of this filtration is divided into columns indexed by $q\in\mathbb{Z}_{\geq0}$, each of which looks like
		\begin{equation}\label{eqn:total complex E1}
			0\to H_q( X\times\Delta^1)\xrightarrow{\delta_*} H_q(X\times\Delta^2)\xrightarrow{\delta_*} H_q(X\times\Delta^3)\xrightarrow{\delta_*}\cdots.
		\end{equation}
		For each $k\geq1$, the map $p_k:=\sigma_0\sigma_1\cdots\sigma_{k-1}: X\times\Delta^k\to  X\times\Delta^0=X$, $(x,t_1,\dots,t_k)\mapsto(x,0)=x$ is a homotopy equivalence. Since $p_{k+1}=p_{k}\sigma_k$ and $\sigma_j\delta_i=\mathrm{id}$ ($i=j,j+1$), we conclude that for any $k\geq1$ and $0\leq i\leq k$, $$H_*(X\times\Delta^{k-1})\xrightarrow{(\delta_i)_*} H_*( X\times\Delta^k)$$ is an isomorphism such that $(\delta_i)_*=(\sigma_{k-1})_*^{-1}$.
		Then since $$\delta|_{S_q(\Delta^{k-1}\times X)}=(-1)^{q+k}\sum_{i=0}^k(-1)^i\delta_i,$$  \eqref{eqn:total complex E1} is nothing but
		\begin{equation*}
			0\to H_q( X\times\Delta^1)\underset{\cong}{\xrightarrow{(-1)^q(\sigma_1)_*^{-1}}} H_q(X\times\Delta^2)\xrightarrow{0} H_q(X\times\Delta^3)\underset{\cong}{\xrightarrow{(-1)^q(\sigma_3)_*^{-1}}}\cdots.
		\end{equation*}
		Thus all $E_2$-terms vanish. Finally, since the filtration $\mathcal{F}_p$ on $\tilde{S}$ is complete (namely $\tilde{S}=\varprojlim \tilde{S}/\mathcal{F}_p\tilde{S}$) and bounded above (since $\tilde{S}_*=\mathcal{F}_1\tilde{S}_*$), standard convergence theorem \cite[Theorem 5.5.10(2)]{Weibel} gives $H_*(\tilde{S},b)=0$.
	\end{proof}
	\begin{rem}\label{rem:cosimplicial pr quasi}
		The proof of Lemma \ref{lem:cocyclic singular quasi} implies that more generally, for a cosimplicial complex $\{(C(k)_*,\partial)\}_{k\geq0}$, if $\sigma_0\sigma_1\cdots\sigma_{k-1}:C(k)\to C(0)$ is a quasi-isomorphism for each $k\geq1$, then so is $\pr_0:(\prod_{k\geq0}C(k)_{*+k},\partial+\delta)\to(C(0)_*,\partial)$ (\cite[Lemma 8.3]{Irie loop}).
	\end{rem}
	\begin{cor}\label{cor:cocyclic singular J quasi}
		For any topological $S^1$-space $X$,
		$\pr_0:(S^{ X\Delta}_*,b,J)\to(S_*(X),\partial,J)$ is a mixed complex quasi-isomorphism.\qed
	\end{cor}
	Note that $S_*(X)=S_*^{ X\Delta}(0)=S^{ X\Delta,\mathrm{nm}}_*(0)$ by vacuum normalized condition. Since $(S^{ X\Delta,\mathrm{nm}}_*,b)\hookrightarrow(S^{ X\Delta}_*,b)$ is a quasi-isomorphism, Lemma \ref{lem:cocyclic singular quasi} and Corollary \ref{cor:cocyclic singular J quasi} also hold true if $S_*^{ X\Delta}$ is replaced by $S^{ X\Delta,\mathrm{nm}}_*$. In the following, we may use $S^{ X\Delta,\mathrm{nm}}_*$ to simplify calculation involving Connes' operator $B$. One could also stick with $S^{ X\Delta}_*$, though.
	
	Recall the augmentation map $\varepsilon:S_0(X)\to\mathbb{R}$, $\sum\lambda_i\cdot(\Delta^0\xrightarrow{u_i}X)\mapsto\sum\lambda_i$.
	\begin{lem}\label{lem:cocyclic infty quasi for S1}
		Consider the topological $S^1$-space $S^1$ with rotation action on itself.
		\begin{enumerate}
			\item\label{item:lem:cocyclic infty quasi for S1 1} There exists a sequence of elements $\{\xi^n=(\xi^n_k)_{k\geq0}\in S^{ S^1\Delta,\mathrm{nm}}_{2n}\}_{n\in\mathbb{Z}_{\geq0}}$ such that 
			\begin{equation*}
				\varepsilon(\xi^0_0)=1,\hspace{.5em}b(\xi^0)=0,\hspace{.5em}b(\xi^n)=(J-B)(\xi^{n-1})\hspace{.16667em}(n\geq1).
			\end{equation*}
			
			\item \label{item:lem:cocyclic infty quasi for S1 2} Suppose $\{\xi^n\}_{n\geq0}$, $\{\xi^{\prime n}\}_{n\geq0}$ both satisfy conditions in \eqref{item:lem:cocyclic infty quasi for S1 1}. Then there exists a sequence of elements $\{\eta^n=(\eta^n_k)_{k\geq0}\in S^{ S^1\Delta,\mathrm{nm}}_{2n+1}\}_{n\in\mathbb{Z}_{\geq0}}$ such that 
			\begin{equation*}
				\xi^0-\xi^{\prime 0}=b(\eta^0),\hspace{.5em}\xi^n-\xi^{\prime n}=b(\eta^n)-(J-B)(\eta^{n-1})\hspace{.16667em}(n\geq1).
			\end{equation*}
		\end{enumerate}
	\end{lem}
	\begin{proof}
		(1)  Consider the isomorphisms $(\pr_0)_*:H_0(S^{ S^1\Delta,\mathrm{nm}},b)\xrightarrow{\cong}H_0(S^1)$ from Lemma \ref{lem:cocyclic singular quasi} and $\varepsilon_*:H_0(S^1)\xrightarrow{\cong}\mathbb{R}$ induced by augmentation. Choose a 0-cycle $\xi^0$ in the homology class $(\varepsilon_*\circ(\pr_0)_*)^{-1}(1)\in H_0(S^{ S^1\Delta,\mathrm{nm}},b)$, then $\xi^0=(\xi^0_k)_{k\geq0}$ is as desired. 
		Next, since $bB=-Bb$ and $bJ=-Jb$, $B(\xi^0)$ and $J(\xi^0)$ are 1-cycles. We claim that they are homologous. Since $\pr_0$ is a quasi-isomorphism, it suffices to look at $\pr_0(B(\xi^0))$ and $\pr_0(J(\xi^0))$. By definition,
		\begin{equation*}
			\pr_0(B(\xi^0))=(Ns(\xi^0))_0=\sigma_0\tau_1(\xi^0_1),\hspace{.5em}\pr_0(J(\xi^0))=J(\xi^0_0)=F_*^{S^1}([S^1]\times\xi^0_0).
		\end{equation*} By construction, $\xi^0_0\in S_0(S^1)$ is homologous to the map $\Delta^0\owns{0}\mapsto[0]\in S^1$, and $\xi^0_1\in S_1( S^1\times\Delta^1)$ is homologous to the map $\Delta^1\owns{t}\mapsto([0],t)\in S^1\times\Delta^1$. So $\pr_0(B(\xi^0))$, $\pr_0(J(\xi^0))$ are homologous to
		\begin{align*}
			&\Delta^1\to S^1\times\Delta^1\xrightarrow{\tau_1} S^1\times\Delta^1\xrightarrow{\sigma_0}S^1; \hspace{.5em}t\mapsto([0],t)\mapsto ([t],1-t) \mapsto [t],\\
			&\Delta^1\to\Delta^1\times\Delta^0 \rightarrow  S^1\times S^1 \xrightarrow{F^{S^1}} S^1; \hspace{.33333em}\
			t\mapsto  (t,0) \mapsto ([t],[0])\mapsto[t],
		\end{align*}
		respectively. Namely they are both homologous to $[S^1]$. This proves the existence of $\xi^1\in S_2^{ S^1\Delta,\mathrm{nm}}$ satisfying $b(\xi^1)=(J-B)(\xi^0)$. Now suppose $\xi^0,\xi^1,\dots,\xi^n$ ($n\geq1$) have been chosen as desired, to find $\xi^{n+1}$, simply notice that $(J-B)(\xi^n)$ is a $(2n+1)$-cycle:
		\begin{equation*}
			b((J-B)(\xi^n))=-(J-B)(b(\xi^n))=-(J-B)^2(\xi^{n-1})=0,
		\end{equation*}
		where $(J-B)^2=0$ since $J^2=0$, $B^2=0$ and $JB+BJ=0$ (see \eqref{eqn:tau commute with J}).
		Since $$H_{2n+1}(S^{ S^1\Delta,\mathrm{nm}},b)\cong H_{2n+1}(S^1)=0\ \text{ ($n\geq1$)},$$ $(J-B)(\xi^n)$ is exact, i.e. $\xi^{n+1}$ exists.
		
		(2) By construction, $\xi^0$ is homologous to $\xi^{\prime0}$, so $\eta^0$ exists. To inductively find $\eta^n$ for $n\geq1$, simply check that $\xi^n-\xi^{\prime n}+(J-B)(\eta^{n-1})$ is a $2n$-cycle, which is then exact since $H_{2n}(S^1)=0$ ($n\geq1$).
	\end{proof}
	
	\begin{prop}\label{prop:cocyclic exists mixed infty}
		Let $X$ be a topological $S^1$-space. Denote the transposition $X\times S^1\to S^1\times X$ by $\nu$.
		\begin{enumerate}
			\item Choose $\xi=\{\xi^n\}_{n\geq0}$ as in Lemma \ref{lem:cocyclic infty quasi for S1}\eqref{item:lem:cocyclic infty quasi for S1 1}. Define a sequence of linear maps $f^\xi=\{f^{\xi}_n:S_*(X)\to S^{ X\Delta,\mathrm{nm}}_{*+2n}\}_{n\geq0}$ by
			\begin{equation*}
				f^{\xi}_n(a):=\big((F^{X\times\Delta^k}\circ(\nu\times\mathrm{id}_{\Delta^k}))_*(a\times\xi^n_k)\big)_{k\geq0}.
			\end{equation*}
			Then $f^\xi$ is an $\infty$-quasi-isomorphism from $(S_*(X),\partial,J)$ to $(S^{ X\Delta,\mathrm{nm}}_*,b,B)$.
			\item For two choices $\xi,\xi^\prime$, the $\infty$-quasi-isomorphisms $f^{\xi},f^{\xi^\prime}$ are $\infty$-homotopic.
		\end{enumerate}
		
	\end{prop}
	\begin{proof}
		(1) The verification of $f^{\xi}_0\circ\partial-b\circ f^{\xi}_0=0$ is simpler than $f^{\xi}_n\circ\partial-b\circ f^{\xi}_n=B\circ f^{\xi}_{n-1}-f^{\xi}_{n-1}\circ J$ ($n\geq1$), so we omit it. Let us write $F_k:=F^{X\times\Delta^k}\circ(\nu\times\mathrm{id}_{\Delta^k}):X\times S^1\times\Delta^k\to X\times\Delta^k$. For $n\geq1$,
		\begin{align*}
			\big((f^{\xi}_n\circ\partial-b\circ f^{\xi}_n)(a)\big)_k&=(F_k)_*(\partial a\times\xi^n_k-\partial(a\times\xi^n_k)-\delta(a\times\xi^n_{k-1}))\\
			&=(F_k)_*((-1)^{|a|}a\times(-\partial\xi^n_k-\delta\xi^n_{k-1}))\\
			&=(-1)^{|a|}(F_k)_*(a\times(B(\xi^{n-1}_{k+1})-J(\xi^{n-1}_k))),
		\end{align*}
		where the last equality follows from $b(\xi^n)=(J-B)(\xi^{n-1})$.
		Now introduce maps
		\begin{align*}
			G_k:&\hspace{.16667em}X\times S^1\times\Delta^k\to X\times\Delta^k\\
			&(x,[t],t_1,\dots,t_k)\mapsto(F^X([t+t_1],x),t_2-t_1,\dots,1-t_1),\\
			H_k:&\hspace{.16667em}X\times S^1\times S^1\times\Delta^k\to X\times\Delta^k\\
			&(x,[t],[t^\prime],t_1,\dots,t_k)\mapsto(F^X([t+t^\prime],x),t_1,\dots,t_k),
		\end{align*}
		then 
		\begin{align*}
			G_k&=F_k\circ(\mathrm{id}_X\times\tau_k^{ S^1\times\Delta^k})=\tau_k^{X\times\Delta^k}\circ F_k,\\
			H_k&=F_k\circ(\mathrm{id}_X\times F^{ S^1\times\Delta^k})=F_k\circ(F^X\times\mathrm{id}_{ S^1\times\Delta^k})\circ(\nu\times\mathrm{id}_{ S^1\times\Delta^k}).
		\end{align*}
		It follows that
		\begin{align*}
			(-1)^{|a|}(F_k)_*(a\times B(\xi^{n-1}_{k+1}))&=B((F_k)_*(a\times\xi^{n-1}_{k+1})),\\
			(-1)^{|a|}(F_k)_*(a\times J(\xi^{n-1}_{k}))&=(F_k)_*(J(a)\times \xi^{n-1}_{k}).
		\end{align*}
		This implies $(f^{\xi}_n\circ\partial-b\circ f^{\xi}_n)(a)=(B\circ f^{\xi}_{n-1}-f^{\xi}_{n-1}\circ J)(a)$, so $f^{\xi}$ is an $\infty$-morphism. It remains to show that $f^{\xi}_0$ is a quasi-isomorphism. Since $\xi^{0}_0$ is homologous to the 0-chain $\Delta^0\to S^1$, ${0}\mapsto[0]$, $\pr_0\circ f^{\xi}_0$ is chain homotopic to $\mathrm{id}_{S_*(X)}$. Since $\pr_0$ is a quasi-isomorphism, so is $f^{\xi}_0$.
		
		(2) Choose $\eta$ as in Lemma \ref{lem:cocyclic infty quasi for S1}\eqref{item:lem:cocyclic infty quasi for S1 2}. Define a sequence of linear maps $h^{\eta}=\{h^{\eta}_n:S_*(X)\to S^{ X\Delta,\mathrm{nm}}_{*+2n+1}\}_{n\geq0}$ by
		\begin{equation*}
			h^{\eta}_n(a):=(-1)^{|a|}\big((F^{X\times\Delta^k}\circ(\nu\times\mathrm{id}_{\Delta^k}))_*(a\times\eta^n_k)\big)_{k\geq0}.
		\end{equation*}
		Then similar calculation as before shows $h^{\eta}$ is an $\infty$-homotopy between $f^\xi$ and $f^{\xi^\prime}$.
	\end{proof}
	
	\section{The story of differentiable spaces} \label{section:differentiable space}
	\subsection{Differentiable spaces and de Rham chains} Materials in this subsection are collected from Irie \cite{Irie loop}. The notion of differentiable spaces is a modification of that utilized by Chen \cite{Chen differentiable space}, and the notion of de Rham chains is inspired by an idea of Fukaya \cite{Fukaya loop}.
	
	Let $\mathscr{U}:=\coprod_{n\geq m\geq 0}\mathscr{U}_{n,m}$ where $\mathscr{U}_{n,m}$ denotes the set of oriented $m$-dimensional $C^\infty$-submanifolds of $\mathbb{R}^n$. Let $X$ be a set. A \emph{differentialble structure} $\mathscr{P}(X)$ on $X$ is a family of maps $\{(U,\vphi)\}$ called \emph{plots}, such that:
	\begin{itemize}
		\item Every plot is a map $\vphi$ from some $U\in\mathscr{U}$ to $X$;
		\item If $\vphi:U\to X$ is a plot, $U^\prime\in\mathscr{U}$ and $\theta:U^\prime\to U$ is a submersion, then $\vphi\circ\theta:U^\prime\to X$ is a plot.
	\end{itemize}
	A \emph{differentiable space} is a pair of a set and a differentiable structure on it. A map $f:X\to Y$ between differentiable spaces is called \emph{smooth}, if $(U,f\circ\vphi)\in\mathscr{P}(Y)$ for any $(U,\vphi)\in\mathscr{P}(X)$. A subset of a differentiable space and the product of a family of differentiable spaces admit naturally induced differentiable structures (\cite[Example 4.2(iii)(iv)]{Irie loop}).
	\begin{rem}\label{rem:differentiable space topology} Differentiable structures are defined on \emph{sets} rather than topological spaces.  For later purpose, we say a differentiable structure and a topology on a set $X$ are \emph{compatible} if every plot is continuous.	\end{rem}
	
	\begin{exmp}\label{exmp:differentiable space}Here are some important examples of differentiable spaces.
		\begin{enumerate}
			\item Let $M$ be a $C^\infty$-manifold. Consider two differentiable structures on it:
			\begin{enumerate}
				\item \label{item:differentiable structure 1a}Define $(U,\vphi)\in\mathscr{P}(M)$ if $\vphi:U\to M$ is a $C^\infty$-map;
				\item Define $(U,\vphi)\in\mathscr{P}(M_{\mathrm{reg}})$ if $\vphi:U\to M$ is a ($C^\infty$-)submersion.
			\end{enumerate}
			The set-theoretic identity map $\mathrm{id}_M:M_{\mathrm{reg}}\to M$ is smooth, but its inverse is not.
			\item Let $\mathcal{L}M:=C^\infty(\mathbb{S}^1,M)$ be the smooth free loop space of $M$, where $\mathbb{S}^1=\mathbb{R}/\mathbb{Z}$. There is a differentiable structure $\mathscr{P}(\mathcal{L}^M)$ on $\mathcal{L}M$ defined by: $(U,\vphi)\in\mathscr{P}(\mathcal{L})$ iff $U\times\mathbb{S}^1\to M$, $(u,[t])\mapsto \vphi(u)(t)$ is a $C^\infty$-map.
			\item For each $k\in\mathbb{Z}_{\geq0}$, the smooth free loop space of $M$ with $k$ inner marked points, denoted by $\mathcal{L}_{k+1}M$, is defined as
			\begin{equation*}
				\{(\gamma,t_1,\dots,t_k)\in\mathcal{L}M\times\Delta^k\mid\partial_t^m\gamma(0)=\partial_t^m\gamma(t_j)=0\hspace{.16667em}(1\leq\forall j\leq k,\hspace{.16667em}\forall m\geq1)\}.
			\end{equation*}
			It has induced differentiable structure $\mathscr{P}(\mathcal{L}_{k+1}^M)$ as a subspace of $\mathcal{L}^M\times\Delta^k$, where $\Delta^k$ is viewed as a subspace of $\mathbb{R}^k$ with the differentiable structure in \eqref{item:differentiable structure 1a}.
			\item The smooth free Moore path space of $M$, denoted by $\Pi M$, is defined as
			\begin{equation*}
				\{(T,\gamma)\mid T\in\mathbb{R}_{\geq0},\hspace{.16667em}\gamma\in C^\infty([0,T],M),\hspace{.16667em} \partial_t^m\gamma(0)=\partial_t^m\gamma(T)=0\hspace{.16667em}(\forall m\geq1)\}.
			\end{equation*}
			Consider two differentiable structures $\mathscr{P}(\Pi^M),\mathscr{P}(\Pi_{\mathrm{reg}}^M)$ on $\Pi M$:
			\begin{enumerate}
				\item Define $(U,\vphi)\in\mathscr{P}(\Pi^M)$ if $\vphi=(\vphi_T,\vphi_\gamma):U\to\Pi M$ satisfies the following conditions: 
				\begin{itemize}
					\item $\vphi_T:U\to\mathbb{R}_{\geq0}$ is a $C^\infty$-map.
					\item The map \begin{equation*}\tilde{U}:=\{(u,t) \mid u\in U,t\in[0,\vphi_T(u)]\}\to M;\hspace{.5em}(u,t)\mapsto \vphi_\gamma(u)(t)\end{equation*} extends to a $C^\infty$-map from an open neighborhood of $\tilde{U}$ in $U\times\mathbb{R}$ to $M$.
				\end{itemize}
				\item Define $(U,\vphi)\in\mathscr{P}(\Pi_{\mathrm{reg}}^M)$ if: $(U,\vphi)\in\mathscr{P}(\Pi^M)$ and the map $U\to M$, $u\mapsto \vphi_\gamma(u)(t_0)$ is a submersion for $t_0=0,T$.
			\end{enumerate}
			\item For each $k\in\mathbb{Z}_{\geq0}$, the smooth free Moore loop space of $M$ with $k$ inner marked points, denoted by $\mathscr{L}_{k+1}M$, is defined as 
			\begin{align*}
				\{((T_0,\gamma_0),\dots,(T_k,\gamma_k))\in(\Pi M)^{k+1}\mid \hspace{.16667em}&
				\gamma_j(T_j)=\gamma_{j+1}(0)\hspace{.16667em}(0\leq j\leq k-1),\\ & \gamma_k(T_k)=\gamma_0(0)\},\\
				=\{(T,\gamma,t_1,\dots,t_k)\in\Pi M\times\mathbb{R}^k\mid\hspace{.16667em} &0\leq t_1\leq\dots\leq t_k\leq T,\hspace{.16667em}\gamma(0)=\gamma(T),\\& \partial_t^m\gamma(t_j)=0\hspace{.16667em}(1\leq \forall j\leq k,\hspace{.16667em}\forall m\geq1)\}.
			\end{align*}
			Apparently there are two ways to endow the set $\mathscr{L}_{k+1}M$ with differentiable structures, namely as a subset of $(\Pi M)^{k+1}$ or of $\Pi M\times\mathbb{R}^k$. It basically follows from \cite[Lemma 7.2]{Irie loop} that these two ways are equivalent. Let us denote by $\mathscr{L}_{k+1}^M$ (resp. $\mathscr{L}_{k+1,\mathrm{reg}}^M$) the differentiable space obtained from $\Pi^M$ (resp. $\Pi_{\mathrm{reg}}^M$).
		\end{enumerate}
		Note that the inclusion of sets 
		$\mathcal{L}_{k+1}M=\{T=1\}\subset\mathscr{L}_{k+1}M$, induced by the inclusion $\mathcal{L}_1M\times\Delta^k\subset\Pi M\times \mathbb{R}^k$,
		is also an inclusion of differentiable spaces $\mathcal{L}_{k+1}^M\hookrightarrow\mathscr{L}_{k+1}^M$.
	\end{exmp}
	
	The \emph{de Rham chain complex} $(C^{\mathrm{dR}}_*(X),\partial)$ of a differentiable space $X$ is defined as follows. 
	For $n\in\mathbb{Z}$, let $\tilde{C}^{\mathrm{dR}}_n(X):=\bigoplus_{(U,\vphi)\in\mathscr{P}(X)} \Omega_c^{\dim U-n}(U).$ For any $(U,\vphi)\in\mathscr{P}(X)$ and $\omega\in\Omega_c^{\dim U-n}(U)$, denote the image of $\omega$ under the natural inclusion $\Omega_c^{\dim U-n}(U)\hookrightarrow\tilde{C}^{\mathrm{dR}}_n(X)$ by $(U,\vphi,\omega)$. Let $Z_n\subset\tilde{C}^{\mathrm{dR}}_n(X)$ be the subspace spanned by all elements of the form $(U,\vphi,\pi_!\omega^\prime)-(U^\prime,\vphi\circ\pi,\omega^\prime)$,	where $(U,\vphi)\in\mathscr{P}(X)$, $U^\prime\in\mathscr{U}$, $\omega^\prime\in\Omega_c^{\dim U^\prime-n}(U^\prime)$, and $\pi:U^\prime\to U$ is a submersion. Then define $C_n^{\mathrm{dR}}(X):=\tilde{C}_n^{\mathrm{dR}}(X)/Z_n$.	By abuse of notation we still denote the image of $(U,\vphi,\omega)$ under the quotient map $\tilde{C}_n^{\mathrm{dR}}(X)\to C_n^{\mathrm{dR}}(X)$ by $(U,\vphi,\omega)$. Then $\partial:C_*^{\mathrm{dR}}(X)\to C_{*-1}^{\mathrm{dR}}(X)$ is defined by $\partial(U,\vphi,\omega):=(U,\vphi,d\omega)$.
	The homology of $(C_*^{\mathrm{dR}}(X),\partial)$ is denoted by $H_*^{\mathrm{dR}}(X)$.
	\begin{rem}
		For any oriented $C^\infty$-manifold $M$, there exists $n\in\mathbb{Z}_{\geq0}$ and an embedding $\iota:M\hookrightarrow \mathbb{R}^n$. Then $(\iota(M),\iota^{-1})\in\mathscr{P}(M_{\mathrm{reg}})\subset\mathscr{P}(M)$, and $(\iota(M),\iota^{-1},(\iota^{-1})^*\omega)\in C^{\mathrm{dR}}_*(M_{\mathrm{reg}})\subset C^{\mathrm{dR}}_*(M) $ for any $\omega\in\Omega_c(M)$. Such a de Rham chain is independent of choices of $n$ and $\iota$, and by abuse of notation we write it as $(M,\mathrm{id}_M,\omega)$. If $M$ is closed oriented, we call $(M,\mathrm{id}_M,1)$ the \emph{fundamental de Rham cycle} of $M$ (or $M_{\mathrm{reg}}$).
	\end{rem}
	
	Let $X,Y$ be differentiable spaces. The \emph{cross product} on de Rham chains is a chain map $C_k^{\mathrm{dR}}(X)\otimes C_l^{\mathrm{dR}}(Y)\to C_{k+l}^{\mathrm{dR}}(X\times Y)$, defined by
	\begin{equation}\label{eq:cross product de Rham chain}
		(U,\vphi,\omega)\times (V,\eta,\psi):=(-1)^{l\cdot\dim U}(U\times V,\vphi\times\psi,\omega\times\eta).
	\end{equation}
	\subsection{$S^1$-equivariant homology of differentiable $S^1$-spaces}\label{subsection:s1 equiv homology differentiable space}
	Let $X$ be a \emph{differentiable $S^1$-space}, namely $X$ is a differentiable space with a smooth map $F^X: S^1\times X\to X$, where $S^1$ is endowed with the differentiable structure in Example \ref{exmp:differentiable space}\eqref{item:differentiable structure 1a}. Let $(S^1,\mathrm{id}_{S^1},1)\in C^{\mathrm{dR}}_1(S^1)$ be the fundamental de Rham 1-cycle of $S^1$. Define
	\begin{equation*}
		J:C^{\mathrm{dR}}_*(X)\to C^{\mathrm{dR}}_{*+1}(X);\hspace{.5em}a\mapsto F^X_*((S^1,\mathrm{id}_{S^1},1)\times a),
	\end{equation*}
	then $J$ is clearly an anti-chain map. We claim $J^2=0$. Let $g: S^1\times S^1\to S^1$ be the smooth map $([t],[t^\prime])\mapsto[t+t^\prime]$, then by the same arguments as in Example \ref{exmp:mixed complex equivariant}, to see $J^2=0$, it suffices to prove $g_*((S^1,\mathrm{id}_{S^1},1)\times(S^1,\mathrm{id}_{S^1},1))=0\in C^{\mathrm{dR}}_2(S^1)$. This is easy, as we can see the following:
	\begin{align*}
		g_*((S^1,\mathrm{id}_{S^1},1)\times(S^1,\mathrm{id}_{S^1},1))&\overset{\eqref{eq:cross product de Rham chain}}{=\joinrel=\joinrel=}-g_*(( S^1\times S^1,\mathrm{id}_{ S^1\times S^1},1))\\
		&=-( S^1\times S^1,\mathrm{id}_{S^1}\circ g,1)=-(S^1,\mathrm{id}_{S^1},g_!(1))=0.
	\end{align*}
	The middle equality on the second line holds since $g$ is a submersion. Thus $(C^{\mathrm{dR}}_*(X),\partial,J)$ is a mixed chain complex. One can then define the positive (ordinary), periodic and negative ``$S^1$-equivariant de Rham homology'' of $X$ as the $\HC^{[v^{-1}]}_*$, $\HC^{[[v]]}_*$ and $\HC^{[[v,v^{-1}]}_*$ versions of cyclic homology of $(C_*^{\mathrm{dR}}(X),\partial,J)$. 
	
	Consider $\Delta^k$ as a differentiable subspace of $\mathbb{R}^k$. Then the cocyclic maps $\delta_i,\sigma_i,\tau_k$ among $\{X\times\Delta^k\}_{k\in\mathbb{Z}_{\geq0}}$, defined by the same formulas as in Section \ref{section:A cocyclic complex and an infty-quasi-isomorphism}, are smooth maps between differentiable spaces. 
	So $\{X\times\Delta^k\}_{k\in\mathbb{Z}_{\geq0}}$ is a cocyclic differentiable space and $\{C_*^{\mathrm{dR}}(X\times\Delta^k)\}_{k\in\mathbb{Z}_{\geq0}}$ is a cocyclic chain complex, which gives rise to a mixed complex \begin{equation*}(C^{\mathrm{dR},X\Delta}_*,b,B):=\Big(\prod_{k\geq0}C^{\mathrm{dR}}_{*+k}( X\times\Delta^k),\partial+\delta,Ns(1-\lambda)\Big).\end{equation*} The smooth $S^1$-action $ S^1\times X\to X$ also extends trivially to $ S^1\times X\times\Delta^k\to X\times\Delta^k$ and gives a mixed complex $(C^{X\Delta}_*,b,J)$.
	
	There is counterpart of Theorem \ref{thm:cocyclic topological isom} for differentiable $S^1$-spaces, whose proof is also similar. We omit the details since we will not make essential use of it. 	     
	
	The \emph{smooth singular chain complex}  $(C^{\mathrm{sm}}_*(X),\partial)$ of a differentiable space $X$, introduced in \cite[Section 4.7]{Irie loop}, is defined in a similar way as the singular chain complex of topological spaces, except that only ``strongly smooth'' maps $\Delta^k\to X$ are considered. 
	The homology of $(C^{\mathrm{sm}}_*(X),\partial)$ is denoted by $H^{\mathrm{sm}}_*(X)$.
	
	Smooth singular homology is related to singular homology and de Rham homology in the following way.
	\begin{itemize}
		\item Let $X$ be a differentiable space with a fixed compatible topology (Remark \ref{rem:differentiable space topology}). Then every strongly smooth map $\Delta^k\to X$ is continuous, hence there is a natural inclusion $(C^{\mathrm{sm}}_*(X),\partial)\hookrightarrow(S_*(X),\partial)$.
		\item $C_*^{\mathrm{sm}}$, $C_*^{\mathrm{dR}}$ are functors from the category of differentiable spaces to the category of chain complexes. Given a cocycle $u=(u_k)_{k\geq0}\in C_0^{\mathrm{dR},\mathrm{pt}\Delta}=\prod_{k\geq0}C_k^{\mathrm{dR}}(\Delta^k)$ in the class $1\in\mathbb{R}\cong H_0^{\mathrm{dR}}(\mathrm{pt})\cong H_0^{\mathrm{dR},\mathrm{pt}}$, there is a natural transformation $\iota^u:C_*^{\mathrm{sm}}\to C_*^{\mathrm{dR}}$ defined by 
		$\iota^u(X)_k: C_k^{\mathrm{sm}}(X)\to C_k^{\mathrm{dR}}(X)$, $\sigma\mapsto \sigma_*(u_k)$. The homotopy class of $\iota^u(X)$ does not depend on $u$ (since $H_n^{\mathrm{dR},\mathrm{pt}}=0$ when $n>0$).
	\end{itemize}
	\begin{assm}\label{assm:differentiable space compatible}
		$X$ is a differentiable space with a fixed compatible topology, such that the chain maps discussed above induce isomorphisms $H_*(X)\xleftarrow{\cong} H^{\mathrm{sm}}_*(X)\xrightarrow{\cong} H_*^{\mathrm{dR}}(X)$.
	\end{assm}
	\begin{prop}\label{prop:differentiable topological s1 isom}
		Let $X$ be a set which satisfies Assumption \ref{assm:differentiable space compatible} and admits an $S^1$-action that is both smooth (with respect to differentiable structure) and continuous (with respect to topology). Then there are natural isomorphisms 
		$\HC^{[v^{-1}]}_*(C^{\mathrm{dR},X\Delta})\cong H^{S^1}_*(X)$, $\HC^{[[v,v^{-1}]}_*(C^{\mathrm{dR},X\Delta})\cong \widehat{H}^{S^1}_*(X)$, as well as $\HC^{\lambda}_*(C^{\mathrm{dR},X\Delta})\cong \HC^{[[v]]}_*(C^{\mathrm{dR},X\Delta})\cong G^{S^1}_*(X)$, which are compatible with tautological and Connes-Gysin long exact sequences.
	\end{prop}
	\begin{proof}
		Consider the mixed complex $(C^{\mathrm{sm},X\Delta}_*,b,B)$ associated to the cocyclic complex $\{C^{\mathrm{sm}}_{*+k}(X\times\Delta^k\}_k$.
		For any $k\geq0$, $X\times\Delta^k$ is a differentiable and topological $S^1$-space satisfying Assumption \ref{assm:differentiable space compatible}. By construction, the inclusions $(C^{\mathrm{sm}}_*(X\times\Delta^k),\partial)\hookrightarrow(S_*(X\times\Delta^k),\partial)$ commute with the cocyclic maps $\delta_i,\sigma_i,\tau_k$, so there is an inclusion $(C^{\mathrm{sm},X\Delta}_*,b,B)\hookrightarrow(S^{X\Delta}_*,b,B)$. On the other hand, given a choice of $u$, the chain maps $\iota^u(X\times\Delta^k)_*:C^{\mathrm{sm}}_*(X\times\Delta^k)\to C^{\mathrm{dR}}_*(X\times\Delta^k)$ commute with cocyclic maps since $\iota^u$ is a natural transformation, so we obtain a morphism $(C^{\mathrm{sm},X\Delta}_*,b,B)\to(C^{\mathrm{dR},X\Delta}_*,b,B)$, whose homotopy class does not depend on $u$. Now consider the following commutative diagram of chain maps
		\begin{center}\begin{tikzcd}
				(S^{X\Delta}_*,b) \arrow[d, "\pr_0"]& (C^{\mathrm{sm},X\Delta}_*,b) \arrow[l] \arrow[r] \arrow[d, "\pr_0"]& (C^{\mathrm{dR},X\Delta}_*,b) \arrow[d, "\pr_0"]\\
				(S_*(X),\partial) & (C^{\mathrm{sm}}_*(X),\partial) \arrow[l] \arrow[r] & (C^{\mathrm{dR}}_*(X),\partial).
		\end{tikzcd}\end{center}
		By Lemma \ref{lem:cocyclic singular quasi} and Remark \ref{rem:cosimplicial pr quasi}, all vertical arrows are quasi-isomorphisms, and by assumption, the arrows in the second row are quasi-isomorphisms. Thus the arrows in the first row are quasi-isomorphisms. In this way we obtain quasi-isomorphisms of mixed complexes $(S^{X\Delta}_*,b,B)\leftarrow(C^{\mathrm{sm},X\Delta}_*,b,B)\to(C^{\mathrm{dR},X\Delta}_*,b,B)$, and get the desired isomorphisms by Lemma \ref{lem:mixed complex quasi}, Theorem \ref{thm:cocyclic topological isom} and Corollary \ref{cor:cocyclic topological isom connes}. Compatibility with long exact sequences is a consequence of Lemma \ref{lem:cyclic long exact seq naturality} and Lemma \ref{lem:Connes exact seq coincide}.
	\end{proof}
	\begin{exmp}\label{exmp:cyclic s1 de Rham}Let $M$ be a closed oriented $C^\infty$-manifold. It is proved in \cite[Section 5, Section 6]{Irie loop} that Assumption \ref{assm:differentiable space compatible} is satisfied for $M,M_{\mathrm{reg}}$ (with manifold topology) and $\mathcal{L}^M$ (with Fr\'echet topology) in Example \ref{exmp:differentiable space}. Moreover, $\mathcal{L}^M$ is an $S^1$-space that Propositon \ref{prop:differentiable topological s1 isom} applies to.
	\end{exmp}
	\subsection{Application to marked Moore loop spaces} Consider the various versions of smooth loop spaces in Example \ref{exmp:differentiable space}. The following lemma is proved in \cite[Section 7]{Irie loop}.
	
	\begin{lem}\label{lem:loop homology isom}
		For any closed oriented $C^\infty$-manifold $M$ and $k\in\mathbb{Z}_{\geq0}$, the zig-zag of smooth maps between differentiable spaces
		\begin{equation*}
			\mathscr{L}_{k+1,\mathrm{reg}}^M\xrightarrow{\mathrm{id}_{\mathscr{L}_{k+1}M}}\mathscr{L}_{k+1}^M\xhookleftarrow{T=1}\mathcal{L}_{k+1}^M\hookrightarrow\mathcal{L}^M\times\Delta^k
		\end{equation*}
		induces a zig-zag of isomorphisms between de Rham homology groups:
		\begin{equation*}
			H^{\mathrm{dR}}_*(\mathscr{L}_{k+1,\mathrm{reg}}^M)\xrightarrow{\cong} H^{\mathrm{dR}}_*(\mathscr{L}_{k+1}^M)\xleftarrow{\cong} H^{\mathrm{dR}}_*(\mathcal{L}_{k+1}^M)\xrightarrow{\cong} H^{\mathrm{dR}}_*(\mathcal{L}^M\times\Delta^k).
		\end{equation*}
	\end{lem}
	The cocyclic structure on $\{\mathcal{L}^M\times\Delta^k\}_{k}$ restricts to $\{\mathcal{L}_{k+1}^M\}_k$.
	There is also a similar structure of cocyclic set on $\{\mathscr{L}_{k+1}M\}_k$ as follows. Regarding $\mathscr{L}_{k+1}M\subset(\Pi M)^{k+1}$, $\delta_i:\mathscr{L}_kM\to \mathscr{L}_{k+1}M$, $\sigma_i:\mathscr{L}_{k+2}M\to \mathscr{L}_{k+1}M$, $\tau_k:\mathscr{L}_{k+1}M \to \mathscr{L}_{k+1}M$ are
	\begin{subequations}\begin{align}
			\delta_i(T,\gamma,t_1,\dots,t_{k-1})&:=
			\begin{cases}
				(T,\gamma,0,t_1,\dots,t_{k-1}) & (i=0)\\
				(T,\gamma,t_1,\dots,t_i,t_i,\dots,t_{k-1}) & (1\leq i\leq k-1)\\
				(T,\gamma,t_1,\dots,t_{k-1},T) & (i=k),
			\end{cases}\\
			\sigma_i(T,\gamma,t_1,\dots,t_{k+1})&:=
			(T,\gamma,t_1,\dots,\widehat{t_{i+1}},\dots,t_{k+1}) \hspace{.16667em}(0\leq i\leq k),\\\label{eqn:cyclic structure on loop}
			\tau_k(T,\gamma,t_1,\dots,t_k)&:=
			(T,\gamma^{t_1},t_2-t_1,\dots,t_k-t_1,T-t_1),
	\end{align}\end{subequations}
	where $\gamma^{t_1}(t):=\gamma(t+t_1)$. These cocyclic maps are smooth for both $\{\mathscr{L}_{k+1}^M\}_k$ and $\{\mathscr{L}_{k+1,\mathrm{reg}}^M\}_k$. Note that if we view $\mathscr{L}_{k+1}M\subset(\Pi M)^{k+1}$, then \eqref{eqn:cyclic structure on loop} can be written as $\tau_k((T_0,\gamma_0),\dots,(T_k,\gamma_k))=((T_1,\gamma_1),\dots,(T_k,\gamma_k),(T_0,\gamma_0))$.
	
	Let us write $(C^{\mathscr{L}}_*,b,B):=\big(\prod_{k\geq 0}C^{\mathrm{dR}}_{*+k}(\mathscr{L}_{k+1,\mathrm{reg}}^M),\partial+\delta,Ns(1-\lambda)\big)$ for the mixed total complex of the cocyclic chain complex $\{C^{\mathrm{dR}}_*(\mathscr{L}_{k+1,\mathrm{reg}}^M)\}_k$.
	
	\begin{prop}\label{prop:Irie loop s1 equiv}
		For any closed oriented $C^\infty$-manifold $M$, there are natural isomorphisms $\HC^{[v^{-1}]}_*(C^{\mathscr{L}})\cong H^{S^1}_*(\mathcal{L}M)$, $\HC^{[[v,v^{-1}]}_*(C^{\mathscr{L}})\cong \widehat{H}^{S^1}_*(\mathcal{L}M)$, and $\HC^{\lambda}_*(C^{\mathscr{L}})\cong \HC^{[[v]]}_*(C^{\mathscr{L}})\cong G^{S^1}_*(\mathcal{L}M)$, which are compatible with long exact sequences.
	\end{prop}
	\begin{proof}
		The smooth maps $\mathscr{L}_{k+1,\mathrm{reg}}^M\xrightarrow{\mathrm{id}}\mathscr{L}_{k+1}^M\xhookleftarrow{T=1}\mathcal{L}_{k+1}^M\hookrightarrow\mathcal{L}^M\times\Delta^k$ commute with cocyclic maps, inducing a zig-zag of mixed complex morphisms between the mixed total complexes associated to the cocyclic de Rham chain complexes of these cocyclic differentiable spaces. By Lemma \ref{lem:loop homology isom}, this is a zig-zag of mixed complex quasi-isomorphisms. The rest is obvious in view of Lemma \ref{lem:mixed complex quasi}, Proposition \ref{prop:differentiable topological s1 isom} and Example \ref{exmp:cyclic s1 de Rham}. 
	\end{proof}
	
	\section{Preliminaries on operads and algebraic structures}\label{section:operads and algebraic structures}
	Let $V=\{V_i\}_{i\in\mathbb{Z}}$ be a (homologically) graded vector space.
	
	A Lie bracket of degree $n\in\mathbb{Z}$ is a Lie bracket on $V[-n]$, namely a bilinear map $[,]:V\ot V\to V$ of degree $n$ satisfying shifted skew-symmetry and Jacobi identity: \begin{equation*}[a,b]=-(-1)^{(|a|-n)(|b|-n)}[b,a],\hspace{.33333em}\
		[a,[b,c]]=[[a,b],c]+(-1)^{(|a|-n)(|b|-n)}[b,[a,c]].
	\end{equation*}
	Note that in this definition, there is no need to apply sign change \eqref{eqn:sign suspension}.
	
	A structure of \emph{Gerstenhaber algebra} is a Lie bracket of degree 1 and a graded commutative (and associative, by default) product $\boldsymbol{\cdot}$ satisfying the Poisson relation:
	\begin{equation*}
		[a,b\boldsymbol{\cdot} c]=[a,b]\boldsymbol{\cdot} c+(-1)^{(|a|+1)|b|}b\boldsymbol{\cdot}[a,c].
	\end{equation*}
	
	A structure of \emph{Batalin-Vilkovisky (BV) algebra} is a graded commutative product $\boldsymbol{\cdot}$ and a linear map $\Delta:V_*\to V_{*+1}$ (called the BV operator) such that $\Delta^2=0$, and 
	\begin{align}\label{eqn:BV}
		\Delta(a\boldsymbol{\cdot} b\boldsymbol{\cdot} c)= \Delta(a\boldsymbol{\cdot} b)\boldsymbol{\cdot} c+(-1)^{|a|}a\boldsymbol{\cdot}\Delta(b\boldsymbol{\cdot} c)+(-1)^{(|a|+1)|b|}b\boldsymbol{\cdot}\Delta(a\boldsymbol{\cdot} c)\\\nonumber
		-\Delta a\boldsymbol{\cdot} b\boldsymbol{\cdot} c-(-1)^{|a|}a\boldsymbol{\cdot}\Delta b\boldsymbol{\cdot} c-(-1)^{|a|+|b|}a\boldsymbol{\cdot} b\boldsymbol{\cdot}\Delta c.
	\end{align}
	By induction, the defining relation \eqref{eqn:BV} implies that for any $k\geq2$,
	\begin{align}\label{eqn:BV generalized}
		\Delta(a_1\boldsymbol{\cdot}a_2\boldsymbol{\cdot}\cdots\boldsymbol{\cdot} a_k)=\sum_{1\leq i<j\leq k}(-1)^{\varepsilon(i,j)}\Delta(a_i\boldsymbol{\cdot}a_j)\boldsymbol{\cdot}a_1\boldsymbol{\cdot}\cdots\wh{a_i}\cdots\wh{a_j}\cdots \boldsymbol{\cdot}a_k\\\nonumber -(k-2)\sum_{1\leq i\leq k}(-1)^{|a_1|+\cdots+ |a_k|}a_1\boldsymbol{\cdot}\cdots\boldsymbol{\cdot}\Delta a_i\boldsymbol{\cdot}\cdots\boldsymbol{\cdot}a_k,
	\end{align}
	where $\varepsilon(i,j)$ is from the Koszul sign rule. By \cite[Proposition 1.2]{Getzler BV}, a BV algebra is equivalently a Gerstenhaber algebra with a linear map $\Delta:V_*\to V_{*+1}$ such that $\Delta^2=0$ and 
	\begin{equation}\label{eqn:BV induce Gerstenhaber}
		[a,b]=(-1)^{|a|}\Delta(a\boldsymbol{\cdot} b)-(-1)^{|a|}\Delta a\boldsymbol{\cdot} b-a\boldsymbol{\cdot}\Delta b.
	\end{equation} 
	
	Following Getzler \cite{Getzler gravity}, a structure of \emph{gravity algebra} is a sequence of graded symmetric linear maps $V^{\ot k}\to V$ ($k\geq2$) of degree 1, $a_1\ot\cdots\ot a_k\mapsto \{a_1,\dots,a_k\}$ (which we call $k$-th bracket), satisfying the following generalized Jacobi relations:
	\begin{align*}
		\sum_{1\leq i<j\leq k}(-1)^{\varepsilon(i,j)}\{\{a_i,a_j\},a_1,\dots,\widehat{a_i},\dots,\widehat{a_j},\dots,a_k,b_1,\dots,b_l\}\\=\begin{cases}\{\{a_1,\dots,a_k\},b_1,\dots,b_l\} & (l>0)\\\hspace{3em}0 & (l=0).\end{cases}
	\end{align*}
	Note that the relation for $(k,l)=(3,0)$ implies that, with sign change \eqref{eqn:sign suspension}, the second bracket becomes an honest Lie bracket on $V[-1]$.
	
	The following lemma, which goes back to \cite[Theorem 6.1]{Chas-Sullivan string topology}, is well-known to experts.
	\begin{lem}\label{lem:BV induce gravity}
		Let $(V_*,\boldsymbol{\cdot},\Delta)$ be a BV algebra, $W_*$ be a graded vector space, with linear maps $\alpha:W_*\to V_*$, $\beta:V_*\to W_{*+1}$ such that $\Delta=\alpha\circ\beta$ and $\beta\circ\alpha=0$.
		Then:
		\begin{enumerate}
			\item \label{item:lem:BV induce gravity 1} $W_*$ is a gravity algebra where the brackets $W^{\ot k}\to W$ are
			\begin{equation}\label{eqn:BV induces gravity}
				\hspace{.16667em}x_1\ot\cdots\ot x_k\mapsto \{x_1,\dots,x_k\}:=\beta(\alpha(x_1)\boldsymbol{\cdot}\dots\boldsymbol{\cdot}\alpha(x_k))\hspace{.33333em}(k\geq2).
			\end{equation}		
			\item \label{item:lem:BV induce gravity 2} Let $[,]$ be the Gerstenhaber bracket \eqref{eqn:BV induce Gerstenhaber} on $V_*$. Then for any $x_1,x_2\in W$,
			\begin{equation*}
				\alpha(\{x_1,x_2\})=(-1)^{|x_1|}[\alpha(x_1),\alpha(x_2)].
			\end{equation*}
		\end{enumerate}
	\end{lem}
	\begin{proof}
		To prove \eqref{item:lem:BV induce gravity 1}, first note that since $\boldsymbol{\cdot}$ is graded commutative, $\{x_1,\dots,x_k\}$ is graded symmetric in its variables. Next, the generalized Jacobi relations follow from a straightforward calculation based on \eqref{eqn:BV generalized} (see the proof of \cite[Theorem 8.5]{Chen string}), and is omitted. The proof of \eqref{item:lem:BV induce gravity 2} is trivial.
	\end{proof}
	A BV algebra homomorphism between two BV algebras is an algebra homomorphism that commutes with their BV operators. The case of gravity algebras is similar. The following lemma is obvious.
	\begin{lem}\label{lem:BV induce gravity morphism}
		Suppose there is a commutative diagram of linear maps
		\begin{center}
			\begin{tikzcd}
				W_*\arrow[r,"\alpha"]\arrow[d,"f"] & V_*\arrow[r,"\beta"]\arrow[d,"g"] & W_{*+1}\arrow[d,"f"]\\
				W_*^\prime\arrow[r,"\alpha^\prime"] & V_*^\prime\arrow[r,"\beta^\prime"] & W_{*+1}^\prime
			\end{tikzcd}
		\end{center}
		such that $(V_*,W_*,\alpha,\beta)$ and $(V_*^\prime,W_*^\prime,\alpha^\prime,\beta^\prime)$ satisfy the assumptions in Lemma \ref{lem:BV induce gravity}, and $g$ is a BV algebra homomorphism. Then $f$ is a gravity algebra homomorphism (for the induced structures on $W$,$W^\prime$).\qed
	\end{lem}
	Next we need to work in the language of operads. We collect some basics below, and refer the reader to \cite[Section 2]{Irie loop} or standard references \cite{Loday operad book}\cite{MSS operad book} for more details.
	
	Let $(\mathscr{C},\ot,1_{\mathscr{C}})$ be a symmetric monoidal category. 
	A \emph{nonsymmetric operad} (ns operad for short) $\mathcal{O}$ in $\mathscr{C}$ consists of the following data:
	\begin{itemize}
		\item An object $\mathcal{O}(k)$ in $\mathscr{C}$ for each $k\in\mathbb{Z}_{\geq0}$.
		\item Morphisms $\circ_i:\mathcal{O}(k)\ot\mathcal{O}(l)\to\mathcal{O}(k+l-1)$ for each $1\leq i\leq k$ and $l\geq0$,
		called \emph{partial compositions}, that are associative: for $x\in\mathcal{O}(k)$, $y\in\mathcal{O}(l)$, $z\in\mathcal{O}(m)$,
		\begin{subequations}
			\begin{align}\label{eqn:operad associativity 1}
				(x\circ_i y)\circ_{i+j-1} z&=x\circ_i(y\circ_jz)\hspace{.66667em}(1\leq i\leq k,\hspace{.16667em}1\leq j\leq l,\hspace{.16667em}m\geq0),\\\label{eqn:operad associativity 2}
				(x\circ_iy)\circ_{l+j-1}z&=(x\circ_j z)\circ_i y\hspace{.66667em}(1\leq i<j\leq k,\hspace{.16667em}l\geq0,\hspace{.16667em}m\geq0).
			\end{align}
		\end{subequations}
		\item A morphism $1_{\mathcal{O}}:1_{\mathscr{C}}\to\mathcal{O}(1)$, which a two-sided unit for $\circ_i$.
	\end{itemize}
	An \emph{operad} is a ns operad such that each $\mathcal{O}(k)$ admits a right action of the symmetric group $\mathfrak{S}_k$ ($\mathfrak{S}_0$ is the trivial group), in a way compatible with partial compositions.
	
	A (ns) operad in the symmetric monoidal category of dg (resp. graded) vector spaces is called a (ns) \emph{dg (resp. graded) operad}. A Koszul sign $(-1)^{|y||z|}$ should appear in \eqref{eqn:operad associativity 2} in graded and dg cases. 
	Taking homology yields a functor from the category of (ns) dg operads to the category of (ns) graded operads.
	
	\begin{exmp}\label{exmp:operads} Here are some examples of dg operads and graded operads. 
		\begin{enumerate}
			\item (Endomorphism operad $\mathcal{E}nd_V$.) For any dg (resp. graded) vector space $V_*$, there is a dg (resp. graded) operad $\mathcal{E}nd_V$ defined as follows. For each $k\geq0$, $\mathcal{E}nd_V(k)_*:=\Hom_*(V^{\ot k},V)$, where $\Hom_*(V^{\ot 0},V)=\Hom_*(\mathbb{R},V)= V_*$. For $1\leq i\leq k$, $l\geq0$, $f\in\Hom_*(V^{\ot k},V)$, $g\in\Hom_*(V^{\ot l},V)$, and $\sigma\in \mathfrak{S}_k$,
			\begin{align*}
				(f\circ_i g)(v_1\ot\cdots \ot v_{k+l-1})&:=(-1)^{\varepsilon}f(v_1\ot\cdots\ot g(v_i\ot\cdots)\ot\cdots),\\
				(f\cdot\sigma)(v_1\ot\cdots\ot v_k)&:=(-1)^{\varepsilon}f(v_{\sigma^{-1}(1)}\ot\cdots\ot v_{\sigma^{-1}(n)}),\\
				1_{\mathcal{E}nd_V}&:=\mathrm{id}_V\in\Hom_0(V,V).
			\end{align*}
			Let $\mathcal{O}$ be a (ns) graded operad or dg operad. A structure of \emph{algebra over $\mathcal{O}$} on $V$, or say an action of $\mathcal{O}$ on $V$, means a morphism $\mathcal{O}\to\mathcal{E}nd_V$ as (ns) operads.
			\item (Gerstenhaber operad $\mathcal{G}er$, BV operad $\mathcal{BV}$, and gravity operad $\mathcal{G}rav$.) These are graded operads that can be defined in terms of generators subject to the relations defining Gerstenhaber/ BV/ gravity algebras. A Gerstenhaber/ BV/ gravity algebra is exactly an algebra over $\mathcal{G}er$/ $\mathcal{BV}$/ $\mathcal{G}rav$.
			
			\item \label{item:example operad 3}(Ward's construction \cite{Ward}.) There is a dg operad $\mathsf{M}_{\circlearrowleft}$ constructed from certain ``labeled $A_\infty$ trees'', such that $H_*(\mathsf{M}_{\circlearrowleft})\cong \mathcal{G}rav$ as graded operads, and there are explicit homotopies measuring the failure of gravity relations on $\mathsf{M}_{\circlearrowleft}$ (while Jacobi relation for the second bracket strictly holds). 
			For this reason an algebra over $\mathsf{M}_{\circlearrowleft}$ can be viewed as a gravity algebra up to homotopy. $\mathsf{M}_{\circlearrowleft}$ is closely related to the operad of ``cyclic brace operations'' (Section \ref{section:cyclic brace operations}). There are other important properties of $\mathsf{M}_{\circlearrowleft}$ that we will use later (Proposition \ref{prop:cyclic operad induce structure}\eqref{item:statement 5}). Indeed the notation of Ward \cite{Ward} is $\mathsf{M}_{\circlearrowright}$, but we use $\mathsf{M}_{\circlearrowleft}$ for the reason of Remark \ref{rem:cyclic orientation}.
		\end{enumerate}
	\end{exmp}

	\begin{defn}\label{defn:cyclic operad}(\cite[Definition 2.6, 2.9]{Irie loop}.) Let $\mathcal{O}$ be a ns dg operad.
		\begin{enumerate}
			\item A \emph{cyclic structure} $(\tau_k)_{k\geq0}$ on $\mathcal{O}$ is a sequence of morphisms $\tau_k:\mathcal{O}(k)\to\mathcal{O}(k)$ $(k\geq0)$ such that $\tau_k^{k+1}=\mathrm{id}_{\mathcal{O}(k)}$, $\tau_1(1_{\mathcal{O}})=1_{\mathcal{O}}$, and that for any $1\leq i\leq k$, $l\geq0$, $x\in\mathcal{O}(k)$, $y\in\mathcal{O}(l)$,
			\begin{equation*}
				\tau_{k+l-1}(x\circ_i y)=\begin{cases}
					\tau_k x\circ_{i-1} y & (i\geq2)\\
					(-1)^{|x||y|}\tau_l y\circ_l \tau_k x & (i=1,\hspace{.16667em}l\geq1)\\
					\tau_k^2x\circ_k y & (i=1,\hspace{.16667em}l=0).
				\end{cases}
			\end{equation*}
			\item A \emph{multiplication} $\mu$ and a \emph{unit} $\varepsilon$ in $\mathcal{O}$ are elements $\mu\in\mathcal{O}(2)_0$,  $\varepsilon\in\mathcal{O}(0)_0$ satisfying $\partial\mu=0$, $\mu\circ_1\mu=\mu\circ_2\mu$, $\partial\varepsilon=0$ and $\mu\circ_1\varepsilon=\mu\circ_2\varepsilon=1_{\mathcal{O}}$.
		\end{enumerate}
	\end{defn}
	
	\begin{rem}\label{rem:cyclic orientation}
		An operad with a cyclic structure is called a \emph{cyclic operad}. The cyclic relation in Definition \ref{defn:cyclic operad} differs from some authors (in particular, Ward \cite{Ward}) in the orientation of performing cyclic permutation, but they are equivalent. See e.g. \cite[Section 3]{Menichi cyclic}.
	\end{rem}
	
	Let $\mathcal{O}=(\mathcal{O}(k))_{k\geq0}$ be a ns dg operad endowed with a multiplication $\mu$ and a unit $\varepsilon$. Then $\{(\mathcal{O}(k)_*,\partial)\}_{k\ge0}$ is a cosimplicial chain complex where $\delta_i:\mathcal{O}(k-1)_*\to\mathcal{O}(k)_*$, $\sigma_i:\mathcal{O}(k+1)_*\to\mathcal{O}(k)_*$ ($0\leq i\leq k$) are
	\begin{equation}\label{eqn:operad cosimplicial}
		\delta_i(x):=\begin{cases}
			\mu\circ_2 x & (i=0)\\
			x\circ_i\mu & (1\leq i\leq k-1)\\
			\mu\circ_1 x & (i=k),
		\end{cases}\quad \sigma_i(x):=x\circ_{i+1}\varepsilon.
	\end{equation}
	Denote the associated total complex by $(\tilde{\mathcal{O}}_*,b)$. If there is also a cyclic structure $(\tau_k)_{k\geq0}$ on $\mathcal{O}$ such that $\mu$ is cyclically invariant, i.e. $\tau_2(\mu)=\mu$, then $\{(\mathcal{O}(k)_*,\partial),\delta_i,\sigma_i,\tau_k\}_{k\geq0}$ is a cocyclic chain complex.
	Denote the associated mixed complex by $(\tilde{\mathcal{O}}_*,b,B)$.
	
	\begin{prop}\label{prop:cyclic operad induce structure}
		Let $\mathcal{O}=(\mathcal{O}(k)_*,\partial)_{k\ge0}$ be a ns dg operad. Define binary operations $\circ$ and $[,]$ on $\tilde{\mathcal{O}}_*:=\prod_{k\geq0}\mathcal{O}(k)_{*+k}$ by the following: for $x=(x_k)_{k\geq0}$, $y=(y_k)_{k\geq0}$,
		\begin{subequations}
			\begin{align}\label{eqn:cyclic operad total complex pre Lie}
				(x\circ y)_k&:=\sum_{\substack{l+m=k+1\\ 1\leq i\leq l}}(-1)^{(i-1)(m-1)+(l-1)(|y|+m)}x_l\circ_i y_m,\\\label{eqn:cyclic operad total complex bracket}
				[x,y]&:=x\circ y-(-1)^{(|x|-1)(|y|-1)}y\circ x.
			\end{align}
		\end{subequations}	
		Then for $(\tilde{\mathcal{O}}_*,\partial)$, statement \eqref{item:statement 1a} below holds true.
		
		If there is a multiplication $\mu$ and a unit $\varepsilon$ on $\mathcal{O}$, define a binary operation
		$\boldsymbol{\cdot}$ on $\tilde{\mathcal{O}}_*$ by
		\begin{equation}
			\label{eqn:cyclic operad total complex prod}
			(x\boldsymbol{\cdot}y)_k:=\sum_{l+m=k}(-1)^{l|y|}(\mu\circ_1 x_l)\circ_{l+1}y_m.
		\end{equation}
		Then for $(\tilde{\mathcal{O}}_*,b)$, statements \eqref{item:statement 1b}\eqref{item:statement 2a} below hold true.
		
		If there is a cyclic structure $(\tau_k)_{k\ge0}$ on $\mathcal{O}$, then for $\tilde{\mathcal{O}}^{\mathrm{cyc}}_*=\mathrm{Ker}(1-\lambda)\subset\tilde{\mathcal{O}}_*$, statement \eqref{item:statement 3} below holds true.
		
		If there is a multiplication $\mu$, a unit $\varepsilon$ and a cyclic structure $(\tau_k)_{k\ge0}$ on $\mathcal{O}$ such that $\tau_2(\mu)=\mu$, then for $(\tilde{\mathcal{O}}_*,b,B)$ and $\tilde{\mathcal{O}}^{\mathrm{cyc}}_*$, the other statements below hold true.
		\begin{enumerate}
			\item \label{item:statement 1}\begin{enumerate}
				\item \label{item:statement 1a}$(\tilde{\mathcal{O}}_*,\partial,\circ)$ is a dg pre-Lie algebra (with shifted grading) such that $[,]$ is a Lie bracket of degree 1.
				\item \label{item:statement 1b}$(\tilde{\mathcal{O}}_*,b,\circ)$ is a dg pre-Lie algebra (with shifted grading) such that $[,]$ is a Lie bracket of degree 1, and $(\tilde{\mathcal{O}}_*,b,\boldsymbol{\cdot})$ is a dg algebra. 
			\end{enumerate}
			\item\label{item:statement 2} \begin{enumerate}
				\item\label{item:statement 2a}  $\boldsymbol{\cdot}$ and $[,]$ induce a Gerstenhaber algebra structure on $H_*(\tilde{\mathcal{O}},b)$. 
				\item\label{item:statement 2b}  $\boldsymbol{\cdot}$ and Connes' operator $B$ induce a BV algebra structure on $H_*(\tilde{\mathcal{O}},b)$ where the BV operator is $\Delta=B_*$. 
				\item\label{item:statement 2c} The above two structures on $H_*(\tilde{\mathcal{O}},b)$ are related by \eqref{eqn:BV induce Gerstenhaber}.
			\end{enumerate}
			\item\label{item:statement 3}  $\tilde{\mathcal{O}}^{\mathrm{cyc}}_*$ is closed under the operation $[,]$. The restriction of $[,]$ to $\tilde{\mathcal{O}}^{\mathrm{cyc}}_*$ is called the cyclic bracket.
			\item\label{item:statement 4} \begin{enumerate}
				\item\label{item:statement 4a} The BV algebra structure on $H_*(\tilde{\mathcal{O}},b)$ obtained in \eqref{item:statement 2b} naturally induces gravity algebra structures on $\HC^\lambda_*(\tilde{\mathcal{O}})$, $\HC^{[[v]]}_*(\tilde{\mathcal{O}})$ and $\HC^{[v^{-1}]}_*(\tilde{\mathcal{O}})[-1]$.
				\item\label{item:statement 4b} The map $B_{0*}:\HC^{[v^{-1}]}_*(\tilde{\mathcal{O}})[-1]\to \HC^{[[v]]}_*(\tilde{\mathcal{O}})$ in \eqref{eqn:cyclic exact seq tautological} is a gravity algebra homomorphism. The map $I_{\lambda*}:\HC^\lambda_*(\tilde{\mathcal{O}})\cong \HC^{[[v]]}_*(\tilde{\mathcal{O}})$ in \eqref{eqn:cyclic Connes isom} is a gravity algebra isomorphism.
				\item\label{item:statement 4c} The Lie bracket on $\HC^\lambda_{*}(\tilde{\mathcal{O}})[-1]$ induced from \eqref{item:statement 3} coincides with the second bracket of its gravity algebra structure, up to sign change \eqref{eqn:sign suspension}.
			\end{enumerate}
			\item\label{item:statement 5} $\tilde{\mathcal{O}}^{\mathrm{cyc}}_*$ admits an action of the operad $\mathsf{M}_{\circlearrowleft}$ (see Example \ref{exmp:operads}) which covers the cyclic bracket in \eqref{item:statement 3}. Via the isomorphism $H_*(\mathsf{M}_{\circlearrowleft})\cong \mathcal{G}rav$, this induces a gravity algebra structure on $\HC^\lambda_*(\tilde{\mathcal{O}})$ which is the same as that in \eqref{item:statement 4a}.
		\end{enumerate}
	\end{prop}
	\begin{proof}Statements \eqref{item:statement 1}\eqref{item:statement 2a} are exactly \cite[Theorem 2.8 (i)-(iii)]{Irie loop}, which in turn follows from \cite[Lemma 2.32]{Ward}. Statements \eqref{item:statement 2b}\eqref{item:statement 2c} follow from \cite[Theorem 2.10]{Irie loop}, which in turn is a consequence of \cite[Theorem B]{Ward}. Note that \cite[Theorem 2.10]{Irie loop} uses the normalized subcomplex $\tilde{\mathcal{O}}_*^{\mathrm{nm}}$, but there is no difference on homology: as explained in \cite[Section 2.5.4]{Irie loop}, the BV operator is just induced by Connes' operator $B=Ns$ on $\tilde{\mathcal{O}}^{\mathrm{nm}}$. 
		
		Statement \eqref{item:statement 3} is a straightforward consequence of \cite[Corollary 3.3]{Ward}. Alternatively, it is quite handy to use definition of cyclic structures to verify that if $\tau_{k_i}x_i=(-1)^{k_i}x_i$ ($x_i\in\mathcal{O}(k_i)$, $i=1,2$), then $\tau_{k_1+k_2-1}[x_1,x_2]=(-1)^{k_1+k_2-1}[x_1,x_2]$.
		
		Statement \eqref{item:statement 4a} is an application of Lemma \ref{lem:BV induce gravity}\eqref{item:lem:BV induce gravity 1} to a part of the exact sequences \eqref{eqn:cyclic exact seq gysin}\eqref{eqn:cyclic exact seq Connes}. Note that there is a transition between (co)homological gradings.
		\begin{itemize}
			\item For $\HC^{[[v]]}_*$, take $V_*=H_*(\tilde{\mathcal{O}},b)$, $W_*=\HC^{[[v]]}_*(\tilde{\mathcal{O}})$, $\alpha=p_{0*}$, and $\beta=B_*$. 
			\item For $\HC^{[v^{-1}]}_{*-1}$, take $V_*=H_*(\tilde{\mathcal{O}},b)$, $W_*=\HC^{[v^{-1}]}_*(\tilde{\mathcal{O}})[-1]=\HC^{[v^{-1}]}_{*-1}(\tilde{\mathcal{O}})$, $\alpha=B_{0*}$, and $\beta=i_*$.
			\item For $\HC^{\lambda}_*$, take $V_*=H_*(\tilde{\mathcal{O}},b)$, $W_*=\HC^{\lambda}_*(\tilde{\mathcal{O}})$, $\alpha=i_{\lambda*}$, and $\beta=B_\lambda$. Here the condition $\alpha\circ\beta=\Delta$ is satisfied because of Lemma \ref{lem:Connes exact seq coincide}.
		\end{itemize}
		Statement \eqref{item:statement 4b} follows from Lemma \ref{lem:BV induce gravity morphism}, Lemma \ref{lem:cyclic exact seq gysin morphism} and Lemma \ref{lem:Connes exact seq coincide}. Statement \eqref{item:statement 4c} follows from Lemma \ref{lem:BV induce gravity}\eqref{item:lem:BV induce gravity 2} and Statement \eqref{item:statement 2c}.
		
		Statement \eqref{item:statement 5} is a consequence of \cite[Theorem C]{Ward}, where the Maurer-Cartan element $\zeta=(\zeta_k)_{k\geq2}$ is taken as $\zeta_2=-\mu$ and $\zeta_k=0$ $(k\neq2)$. 
		
		To see statement \eqref{item:statement 5} covers statement \eqref{item:statement 3}, we need concrete description of the action of $\mathsf{M}_{\circlearrowleft}$ on $\tilde{\mathcal{O}}^{\mathrm{cyc}}_*$. For arity 2 it is the same as cyclic brace operations (see Example \ref{exmp:cyclic brace algebra}).
	\end{proof}
	\begin{rem}
		The sign in \eqref{eqn:cyclic operad total complex pre Lie} comes from operadic suspension (see Appendix \ref{subsection:operadic suspension}). Indeed, $\tilde{\mathcal{O}}=(\prod_{n\geq0}\mathfrak{s}\mathcal{O}(n))[-1]$.
	\end{rem}
	\begin{rem}
		Indeed, \cite[Theorem 2.8 \& 2.10]{Irie loop} and \cite[Theorem A \& B]{Ward} contain much stronger statements than Proposition \ref{prop:cyclic operad induce structure}\eqref{item:statement 1}\eqref{item:statement 2} which we do not need (e.g. existence of an action of a chain model of the (framed) little 2-disks operad on $\tilde{\mathcal{O}}$ or $\tilde{\mathcal{O}}^{\mathrm{nm}}$). Proposition \ref{prop:cyclic operad induce structure}\eqref{item:statement 1}\eqref{item:statement 2} themselves were known much earlier, e.g. see \cite[Section 1.2 \& Theorem 1.3]{Menichi cyclic}.
	\end{rem}
	
	\begin{exmp}\label{exmp:end operad} Let $(A^*,d,\boldsymbol{\cdot})$ be a dg algebra with unit $1_A$. Then $\mathcal{E}nd_A$ admits a multiplication and a unit given by $\mu(a_1\ot a_2):=a_1\boldsymbol{\cdot}a_2$, $\varepsilon:=1_A$. Viewing $A$ as a dg $A$-bimodule, the cosimplicial maps $\delta_i,\sigma_i$ in Example \ref{exmp:cyclic cohomology of algebra} are the same as \eqref{eqn:operad cosimplicial} for $(\mathcal{E}nd_A,\mu,\varepsilon)$ in \eqref{item:statement 1}. To discuss cyclic structures, suppose  there is a graded symmetric bilinear form $\langle,\rangle:A\times A\to \mathbb{R}$ of degree $m\in\mathbb{Z}$, such that
		\begin{equation*}
			d\langle a,b\rangle=\langle da,b\rangle+(-1)^{|a|}\langle a,db\rangle,\hspace{.66667em}\langle ab,c\rangle=\langle a,bc\rangle\hspace{.5em}(\forall a,b,c\in A).
		\end{equation*}	
		Namely $A$ is a dg version of a Frobenius algebra, but we do not require $\dim_{\mathbb{R}}A$ to be finite or $\langle,\rangle$ to be nondegenerate. Note that since $\langle,\rangle$ is symmetric, the relation $\langle ab,c\rangle=\langle a,bc\rangle$ is equivalent to $\langle,\rangle$ being \emph{cyclic}, i.e. 
		\begin{equation}\label{eqn:<> cyclic invariant}
			\langle ab,c\rangle=(-1)^{|a|(|b|+|c|)}\langle bc,a\rangle.
		\end{equation} Now consider $A^\vee[m]$ with a dg $A$-bimodule structure characterized by \eqref{eqn:A dual bimodule} (the degree of $\vphi\in A^\vee[m]$ is now shifted). The degree 0 map 
		\begin{equation}\label{eqn:bimodule map A to dual}
			\theta:A\to A^\vee[m];\hspace{.5em}\theta(a)(b):=\langle a,b\rangle\hspace{.16667em}(\forall a,b\in A)
		\end{equation}
		is a dg $A$-bimodule map, and $\Hom(-,\theta):\Hom^*(A^{\ot k},A)\to\Hom^*(A^{\ot k},A^\vee[m])$ is a morphism of cosimplicial complexes. $\{\Hom^*(A^{\ot k},A^\vee[m])=\Hom^{*+m}(A^{\ot k+1},\mathbb{R})\}_{k\geq0}$ is moreover cocyclic with cyclic permutations $(\tau_k)_{k\geq0}$ given in Example \ref{exmp:cyclic cohomology of algebra}.
		
		If $\theta$ happens to be an isomorphism, then $\{\Hom(-,\theta)^{-1}\circ\tau_k\circ\Hom(-,\theta)\}_{k\geq0}$ endows $(\mathcal{E}nd_A,\mu,\varepsilon)$ with a cyclic structure.  All statements of Proposition \ref{prop:cyclic operad induce structure} hold for $\mathcal{O}=\mathcal{E}nd_A$. 
		
		If $\theta$ is a quasi-isomorphism, then  $\widetilde{\mathcal{E}nd_A}=\CH^*(A,A)$ and $\CH^*(A,A^\vee[m])$ are quasi-isomorphic through a natural map induced by $\theta$. In this case, let us examine the statements \eqref{item:statement 1}-\eqref{item:statement 5} in Proposition \ref{prop:cyclic operad induce structure} for $\mathcal{O}=\mathcal{E}nd_A$. 
		\begin{enumerate}[label=(\arabic*)$^\prime$]
			\item \label{item:statement A} Statement \eqref{item:statement 1} still holds honestly (it is irrelevant to $\theta$). 
			\item \label{item:statement B} Statement \eqref{item:statement 2a} holds honestly (it is irrelevant to $\theta$). Statements \eqref{item:statement 2b}\eqref{item:statement 2c} ``hold weakly'' in the following sense:
			Connes' operator $B$ on $\CH^*(A,A^\vee[m])$ induces a BV operator on $\mathrm{HH}^*(A,A)\cong\mathrm{HH}^*(A,A^\vee[m])$, making $\mathrm{HH}^*(A,A)$ into a BV algebra, which is compatible with its Gerstenhaber algebra structure. This is proved by Menichi \cite[Theorem 18]{Menichi HH BV}. 
			\item \label{item:statement C} Statement \eqref{item:statement 3} ``holds weakly'' in the sense that the subspace of \emph{weakly cyclic invariants} in $\CH^*(A,A)$, $\Theta^{-1}(\CH_{\mathrm{cyc}}^*(A,A^\vee[m]))$, is closed under the bracket \eqref{eqn:cyclic operad total complex bracket}, and hence is a dg Lie subalgebra. Here $\Theta:\CH^*(A,A)\to\CH^*(A,A^\vee[m])$ is the cochain map induced by \eqref{eqn:bimodule map A to dual}, and $\CH_{\mathrm{cyc}}^*(A,A^\vee[m]):=\mathrm{Ker}(1-\lambda)$ is the subcomplex of cyclic invariants in $\CH^*(A,A^\vee[m])$, with respect to the cocyclic structure on $\{\Hom^{*+m}(A^{\ot k+1},\mathbb{R})\}_{k\geq0}$. This result is rather simple and should be well-known, e.g. it is stated without proof in \cite[Lemma 4]{Penkava}.
			\item \label{item:statement D} Statement \eqref{item:statement 4} ``holds weakly'' in the following sense: there are gravity algebra structures on $\HC_\lambda^*(A,A^\vee[m])\cong \HC_{[[u]]}^*(A,A^\vee[m])\cong \HC_{[[u]]}^*(A,A)$ and on $\HC_{[u^{-1}]}^*(A,A^\vee[m])\cong \HC_{[u^{-1}]}^*(A,A)$, induced by the BV algebra structure on $\mathrm{HH}^*(A,A)\cong\mathrm{HH}^*(A,A^\vee[m])$ described in \ref{item:statement C}.			
			\item \label{item:statement E}Statement \eqref{item:statement 5} ``holds weakly'' by Corollary \ref{cor:homotopy gravity action}, which largely generalizes \ref{item:statement C}. 
		\end{enumerate}  
	\end{exmp}
	\begin{rem}
		Statements \ref{item:statement C}\ref{item:statement E} above hold true even if $\theta:A\to A^\vee[m]$ is not a quasi-isomorphism. If $\theta$ is a quasi-isomorphism, then so is $\Theta:\CH(A,A)\to \CH(A,A^\vee[m])$. If $\Theta$ also restricts to a quasi-isomorphism $\Theta^{-1}(\CH_{\mathrm{cyc}}(A,A^\vee[m]))\to\CH_{\mathrm{cyc}}(A,A^\vee[m])$, then the structures in \ref{item:statement C}\ref{item:statement E} are compatible with those in \ref{item:statement B}\ref{item:statement D}.
	\end{rem}
	\begin{rem}\label{rem:end operad vs cyclic}
		Statement \ref{item:statement C} in Example \ref{exmp:end operad} is irrelevant to the algebra structure on $A$. It holds true when $A$ is just a graded vector space endowed with a symmetric bilinear form $\langle,\rangle:A\times A\to\mathbb{R}$ of degree $m$. In this case, we shall write $\Theta^{-1}(\prod_{k\geq 0}\Hom^{*+m}_{\mathrm{cyc}}(A^{k+1},\mathbb{R}))$ in place of $\Theta^{-1}(\CH_{\mathrm{cyc}}^*(A,A^\vee[m]))$. 
	\end{rem}
	
	\section{Cyclic brace operations}\label{section:cyclic brace operations}
	This section is devoted to the proof of Theorem \ref{thm:main result 3}.  Recall  $\tilde{\mathcal{O}}:=\prod_{k\geq0}\mathcal{O}(k)$ if $\mathcal{O}$ is dg operad, and $\tilde{\mathcal{O}}^{\mathrm{cyc}}:=\mathrm{Ker}(1-\lambda)\subset\tilde{\mathcal{O}}$ if $\mathcal{O}$ is a dg cyclic operad .
	\begin{defn}[Brace operations via concrete formulas]\label{defn:brace operation formula}
		Let $\mathcal{O}$ be a dg operad. For each $n\in\mathbb{Z}_{\geq0}$, define an $(n+1)$-ary operation on $\tilde{\mathcal{O}}$ as follows. When $n=0$, for $a\in\mathcal{O}(r)$, let $a\{\}:=a$. When $n>0$, for $a\in \mathcal{O}(r)$ and $b_j\in\mathcal{O}(t_j)$ ($1\leq j\leq n$), let
		\begin{equation}\label{eqn:brace}
			a\{b_1,\dots,b_n\}:=\sum_{i_1,\dots,i_n}\pm(\cdots((a\circ_{i_1}b_1)\circ_{i_2}b_2)\cdots\circ_{i_n}b_n),
		\end{equation}
		where the summation is taken over tuples $(i_1,\dots,i_n)\in\mathbb{Z}_{\geq1}^n$ satisfying $i_{j+1}\geq i_j+t_j$ and $i_n\leq r-n+1+\sum_{l=1}^{n-1}t_l$. The sign $\pm$ is from iteration of \eqref{eqn:operadic suspension}. 
	\end{defn}
	
	Brace operations were first described by Getzler \cite{Getzler brace} in Hochschild context (generalizing the Gerstenhaber bracket \cite{Gerstenhaber} which corresponds to $n=2$) and later by Gerstenhaber-Voronov \cite{GV} in operadic context. There is also an interpretation of brace operations via planar rooted trees, going back to the ``minimal operad'' of Kontsevich-Soibelman \cite{KS} (see also \cite[Section 7-9]{DW}), which allows for a variation in the cyclic invariant setting (\cite{Ward}). 
	
	Let us fix terminologies about trees before moving to more details. 
	\begin{itemize}
		\item A \emph{tree without tails} is a contractible 1-dimensional finite CW complex. A 0-cell is called a \emph{vertex}; the closure of a 1-cell is called an \emph{edge} (identified with $[0,1]$). 
		\item A \emph{tree with tails} is a
		tree without tails attached with copies of $[0,1)$ called \emph{tails} by gluing each $0\in[0,1)$ to some vertex.
	\end{itemize}
	
	The set of vertices, edges and tails in a tree $T$ is denoted by $V_T$, $E_T$ and $L_T$, respectively. The set of edges and tails at $v\in V_T$ is denoted by $E_v$ and $L_v$, respectively. The \emph{valence} of a vertex $v$ is the number $|E_v\cup L_v|$. The \emph{arity} of a vertex is its valence $-1$.
	\begin{itemize}
		\item An \emph{oriented tree} is a tree with a choice of direction for each edge, from one vertex to the other. Such a choice of directions is called an orientation of the tree.
	\end{itemize}
	\begin{itemize}
		\item A \emph{rooted tree} is a tree with a choice of a distinguished tail called the \emph{root}. 
	\end{itemize}
	
	Every rooted tree is naturally oriented by directions towards the root.
	\begin{itemize}
		\item A \emph{planar tree} is a tree with a cyclic order on $E_v\cup L_v$ for each vertex $v$.
	\end{itemize}
	
	Every planar tree can be embedded into the plane in a way unique up to isotopy, so that at each vertex $v$, the cyclic order on $E_v\cup L_v$ is counterclockwise.
	
	Every planar rooted tree $T$ carries a natural \emph{total order} on $E_T\cup L_T$, which can be obtained by moving counterclockwise along the boundary of a small tubular neighborhood of $T$ in the plane. It starts from the root and is compatible with the cyclic order on $E_v\cup L_v$ for each $v\in V_T$, and also restricts to total orders on $E_T$, $L_T$ and $E_v$, $L_v$ for each $v\in V_T$.
	\begin{itemize}
		\item An $n$-\emph{labeled tree} is a tree $T$ with a bijection between $\{1,2,\dots,n\}$ and $V_T$. If the number of vertices is not specified, it is just called a labeled tree.
	\end{itemize}
	
	The vertex with label $i$ in an $n$-labeled tree $T$ is denoted by $v_i(T)$, with arity $a_i(T)$.
	
	The notion of isomorphisms of trees (with various structures) is obvious. We shall view isomorphic trees as the same.
	
	For $n\in\mathbb{Z}_{\geq1}$, let $\mathsf{B}^s(n)$ be the set of $n$-labeled planar rooted trees without non-root tails, and let $\mathsf{B}(n)$ be the vector space spanned by $\mathsf{B}^s(n)$. Let $\bar{\mathsf{B}}^s(n)$ be the set of $n$-labeled planar rooted trees with tails, and let $\bar{\mathsf{B}}(n)$ be the vector space spanned by $\bar{\mathsf{B}}^s(n)$.
	
	Given $T^\prime\in\bar{\mathsf{B}}(n)$ and $k=(k_1,\dots,k_n)\in\mathbb{Z}_{\geq0}^n$, define a set
	\begin{align}\nonumber
		\mathcal{T}(T^\prime,k):=\{\hspace{.16667em}T^{\prime\prime}\in\bar{\mathsf{B}}^s(n)\mid \text{$T^{\prime\prime}$ can be obtained by attaching tails to $T^\prime$}& \\  \text{so that $a_i(T^{\prime\prime})=k_i$ ($1\leq \forall i\leq n$)}&\hspace{.16667em}\}.\label{eqn:rooted tree attach tails}
	\end{align}
	\begin{defn}[Brace operations via trees]\label{defn:brace operation tree}
		Let $\mathcal{O}$ be a dg operad. Define linear maps \begin{equation*}\kappa_n:\mathsf{B}(n)\to\Hom(\tilde{\mathcal{O}}^{\ot n},\tilde{\mathcal{O}}),\hspace{.5em}\bar{\kappa}_n:\bar{\mathsf{B}}(n)\to\Hom(\tilde{\mathcal{O}}^{\ot n},\tilde{\mathcal{O}})\end{equation*} as follows.
		$\kappa_n$ is the restriction of $\bar{\kappa}_n$. For $T^\prime\in\bar{\mathsf{B}}^s(n)$ and $f_i\in\mathcal{O}(k_i)$ ($1\leq i\leq n$),
		\begin{equation*}
			\bar{\kappa}_n(T^\prime)(f_1,f_2,\dots, f_n):=\sum_{T^{\prime\prime}\in\mathcal{T}(T^\prime,k)}\pm\bar{\kappa}_n(T^{\prime\prime})(f_1,f_2,\dots, f_n),
		\end{equation*}
		where by convention summation over the empty set is zero. If $a_i(T^\prime)=k_i$ ($1\leq i\leq n$), then $\bar{\kappa}_n(T^\prime)(f_1,\dots,f_n)$ is the operadic composition of $f_1,\dots,f_n$ in the obvious way described by $T^\prime$, where $f_i$ is assigned to $v_i(T^\prime)$. The sign $\pm$ is from iteration of \eqref{eqn:operadic suspension}.
	\end{defn}

	Definition \ref{defn:brace operation formula} and Definition \ref{defn:brace operation tree} describe the same operations on $\tilde{\mathcal{O}}$, as explained below. Consider $\beta_n\in\mathsf{B}^s(n+1)$ characterized by: $E_{\beta_n}=\{e_1,\dots,e_n\}$, $V_{\beta_n}=\{v_1,\dots,v_{n+1}\}$ where $v_i$ is labeled by $i$, $L_{\beta_n}=\{l_1\}$, $v_1=l_1\cap e_1\cap\dots\cap e_n$, $v_{i+1}\in e_i$ ($1\leq \forall i\leq n$), and the cyclic order on $E_{v_1}\cup L_{v_1}$ is $(l_1,e_1,\dots,e_n)$. Then $\kappa_{n+1}(\beta_n)\in\Hom(\tilde{\mathcal{O}}^{\ot n+1},\tilde{\mathcal{O}})$ is exactly given by \eqref{eqn:brace}. Moreover, putting $\mathsf{B}(n)$ in degree $n-1$, $\mathsf{B}=\{\mathsf{B}(n)\}$ carries a \emph{reduced} (meaning $\mathsf{B}(0)=0$ and $\mathsf{B}(1)=\mathbb{R}$) graded operad structure (see \cite[Definition 2.11 \& 2.13]{Ward}) so that $\{\beta_n\}_{n\geq0}$ generates $\mathsf{B}$ under operadic compositions and symmetric permutations, and 
	\begin{equation}\label{eqn:brace action}
		\kappa=\{\kappa_n\}:\mathsf{B}\to\mathcal{E}nd_{\tilde{\mathcal{O}}}
	\end{equation}
	is a morphism of operads. $\mathsf{B}$ is called the \emph{brace operad}, which tautologically controls brace operations on \emph{brace algebras}, i.e. algebras over $\mathsf{B}$. We have just seen that $\tilde{\mathcal{O}}$ is naturally a brace algebra. For the purpose of this paper we omitted details of operadic compositions on $\mathsf{B}$  but explained how $\kappa$ is defined.
	
	In \cite[Section 3.2]{Ward}, Ward introduced an operad $\mathsf{B}_{\circlearrowleft}$ which he called \emph{cyclic brace operad}. Let $\mathsf{B}^s_{\circlearrowleft}(n)$ be the set of oriented $n$-labeled planar trees without tails. Then $\mathsf{B}_{\circlearrowleft}(n)$ is the graded vector space spanned by $\mathsf{B}_{\circlearrowleft}^s(n)$ modulo the relation that reversing direction on an edge produces a negative sign. If there is no risk of confusion, we will by abuse of notation not distinguish $T_{\circlearrowleft}\in\mathsf{B}_{\circlearrowleft}^s(n)$ from its image in $\mathsf{B}_{\circlearrowleft}(n)$. 
	There is a morphism of operads $\rho:\mathsf{B}_{\circlearrowleft}\to\mathsf{B}$, which is
	induced by maps
	\begin{equation}\label{eqn:tree add root}
		\rho^s_n:\mathsf{B}^s_{\circlearrowleft}(n)\to \mathsf{B}(n),\hspace{.5em}T_{\circlearrowleft}\mapsto \sum_{T\in\mathcal{R}_1(T_{\circlearrowleft})}(-1)^{\varepsilon(T_{\circlearrowleft},T)} T,
	\end{equation}
	where $\mathcal{R}_1(T_{\circlearrowleft})$ is the set of labeled planar rooted trees that can be obtained by adding a root to the (non-rooted) tree underlying $T_{\circlearrowleft}$, and $\varepsilon(T_{\circlearrowleft},T)$ is the number of edges in $E_{T_{\circlearrowleft}}=E_T$ whose direction from $T_{\circlearrowleft}$ does not agree with the direction from the rooted structure of $T$.
	Here and hereafter, in appropriate contexts we use $f^s$ to denote a set-theoretic map which induces a linear map $f$.
	
	A natural example of cyclic brace algebras, i.e. algebras over $\mathsf{B}_{\circlearrowleft}$, is as follows.
	
	\begin{exmp}[{\cite[Corollary 3.11]{Ward}}]\label{exmp:cyclic brace algebra}
		Let $\tilde{\mathcal{O}}$ be a dg cyclic operad. Consider $\tilde{\mathcal{O}}^{\mathrm{cyc}}\subset\tilde{\mathcal{O}}$ and $\kappa:\mathsf{B}\to\mathcal{E}nd_{\tilde{\mathcal{O}}}$ in \eqref{eqn:brace action}. For $T\in\mathsf{B}^s_\circlearrowleft(n)$, $\rho(T)\in\mathsf{B}(n)$ and $(\kappa\circ\rho)(T)\in\Hom(\tilde{\mathcal{O}}^{\ot n},\tilde{\mathcal{O}})$. Restricting $(\kappa\circ\rho)(T)$ to $\tilde{\mathcal{O}}^{\mathrm{cyc}}$ gives an element in $\Hom((\tilde{\mathcal{O}}^{\mathrm{cyc}})^{\ot n},\tilde{\mathcal{O}})$. Moreover, if $f_i\in\mathcal{O}(k_i)$ ($1\leq i\leq n$) are cyclic invariant, then so is $(\kappa\circ\rho)(T)(f_1,\dots,f_n)$. (Such a claim appears in \cite[Theorem 5.5]{Ward} where it is referred to \cite[Proposition 3.10]{Ward}, but there is no direct proof given in \cite{Ward}. We will give a direct proof in a slightly different situation.) Hence $\kappa\circ\rho$ gives a morphism $\mathsf{B}_\circlearrowleft\to \mathcal{E}nd_{\tilde{\mathcal{O}}^{\mathrm{cyc}}}$.
	\end{exmp}
	\begin{defn}[Cyclic brace operations]\label{defn:cyclic brace operation}
		Let $\mathcal{O}$ be a dg cyclic operad. The cyclic brace operations on $\tilde{\mathcal{O}}^{\mathrm{cyc}}$ are those characterized by the linear maps
		\begin{equation*}
			\kappa_n\circ\rho_n:\mathsf{B}_\circlearrowleft(n)\to\Hom((\tilde{\mathcal{O}}^{\mathrm{cyc}})^{\ot n},\tilde{\mathcal{O}}^{\mathrm{cyc}})
		\end{equation*}
		discussed in Example \ref{exmp:cyclic brace algebra}.
	\end{defn}
	\begin{rem}
		It seems hard to write a direct formula for cyclic brace operations on $\tilde{\mathcal{O}}^{\mathrm{cyc}}$ in terms of operadic compositions, in a way as explicit as \eqref{eqn:brace}.
	\end{rem}
	
	Consider $\bar{\mathsf{B}}^s_{\circlearrowleft}(n)\supset\mathsf{B}^s_{\circlearrowleft}(n)$ and $\bar{\mathsf{B}}_{\circlearrowleft}(n)\supset\mathsf{B}_{\circlearrowleft}(n)$ by extending the definitions to labeled planar trees with tails. There is a forgetful map
	\begin{equation*}
		w_n^s:\bar{\mathsf{B}}^s(n)\to\bar{\mathsf{B}}^s_{\circlearrowleft}(n)\setminus\mathsf{B}^s_{\circlearrowleft}(n)
	\end{equation*}
	forgetting the choice of root but keeping the orientation from rooted structure. Note that $w^s_n$ induces $w_n:\bar{\mathsf{B}}(n)\to\bar{\mathsf{B}}_{\circlearrowleft}(n)/\mathsf{B}_{\circlearrowleft}(n)$.
	There is also a map
	\begin{equation*}
		r^s_n:\bar{\mathsf{B}}^s_{\circlearrowleft}(n)\setminus\mathsf{B}^s_{\circlearrowleft}(n)\to\bar{\mathsf{B}}(n),\hspace{.5em}T_{\circlearrowleft}^\prime\mapsto \sum_{T^\prime\in\mathcal{R}_0(T_{\circlearrowleft}^\prime)}(-1)^{\varepsilon(T_{\circlearrowleft}^\prime,T^\prime)} T^\prime,
	\end{equation*}
	where $\mathcal{R}_0(T_{\circlearrowleft}^\prime)$ is the set of $n$-labeled planar rooted trees obtained by choosing one of the tails in $T_{\circlearrowleft}^\prime$ as the root, and  $\varepsilon(T_{\circlearrowleft}^\prime,T^\prime)$ is defined similar to $\varepsilon(T_{\circlearrowleft},T)$ in \eqref{eqn:tree add root}. 
	
	It is clear that $(w_n\circ r_n)(T_{\circlearrowleft}^\prime)=|\mathcal{R}_0(T_{\circlearrowleft}^\prime)|\cdot T_{\circlearrowleft}^\prime$. To describe $r_n\circ w_n$, consider a map
	\begin{equation*}
		t_n^s:\bar{\mathsf{B}}^s(n)\to\bar{\mathsf{B}}^s(n)
	\end{equation*}
	so that $T^\prime$ and $t_n^s(T^\prime)$ are the same after forgetting the root, and the root of $t_n^s(T^\prime)$ is the first non-root tail of $T^\prime$ (if there are no non-tail roots then $t_n^s(T^\prime)=T^\prime$).
	Then for any $T^\prime\in\mathcal{R}_0(T_{\circlearrowleft}^\prime)$, we have $\mathcal{R}_0(T_{\circlearrowleft}^\prime)=\{T^\prime,t_n^s(T^\prime),\dots,(t_n^s)^p(T^\prime)\}$, where $p=|\mathcal{R}_0(T_{\circlearrowleft}^\prime)|-1$. It follows that 
	\begin{equation}\label{eqn:cyclic permute all}
		(r_n\circ w_n)(T^\prime)=\varepsilon(T_{\circlearrowleft}^\prime,T^\prime)\cdot r_n(T_{\circlearrowleft}^\prime)=\sum_{0\leq i\leq|\mathcal{R}_0(T_{\circlearrowleft}^\prime)|-1}(-1)^{\varepsilon(T^\prime,t_n^i(T^\prime))}\cdot t_n^i(T^\prime).
	\end{equation}
	Here $\varepsilon(T^\prime,t_n^i(T^\prime))$ is the number of edges in $E_{T^\prime}=E_{t_n^i(T^\prime)}$ whose direction towards the root of $T^\prime$ does not agree with the direction towards the root of $t_n^i(T^\prime)$.
	
	Given $k=(k_1,\dots,k_n)\in\mathbb{Z}_{\geq0}^n$, define a map 
	\begin{equation*}
		\nu_k^s:\mathsf{B}^s_{\circlearrowleft}(n)\to\bar{\mathsf{B}}_{\circlearrowleft}(n),\hspace{.5em}T_{\circlearrowleft}\mapsto \sum_{T_{\circlearrowleft}^\prime\in\mathcal{T}(T_{\circlearrowleft},k)} T_{\circlearrowleft}^\prime,
	\end{equation*} 
	where $\mathcal{T}(k,T_{\circlearrowleft})\subset\bar{\mathsf{B}}^s_{\circlearrowleft}(n)$ is defined similar to $\mathcal{T}(k,T)$ in \eqref{eqn:rooted tree attach tails}. 
	
	\begin{lem}\label{lem:cyclic brace action two equiv}
		Let $\mathcal{O}$ be a dg operad. For any $T_\circlearrowleft\in\mathsf{B}^s_{\circlearrowleft}(n)$ and $f_i\in\mathcal{O}(k_i)$ ($1\leq i\leq n$), there holds
		\begin{equation*}
			(\kappa_n\circ\rho_n^s)(T_{\circlearrowleft})(f_1,\dots,f_n)=(\bar{\kappa}_n\circ r_n\circ\nu_k^s)(T_{\circlearrowleft})(f_1,\dots,f_n).
		\end{equation*}
	\end{lem}
	\begin{proof}
		Consider the set of labeled planar rooted trees whose vertices have arities equal to $k=(k_1,\dots,k_n)$ in accordance with the labeling. Such a set can be represented as
		\begin{equation*}  \bigcup_{T\in\mathcal{R}_1(T_{\circlearrowleft})}\mathcal{T}(T,k)=\bigcup_{T_{\circlearrowleft}^\prime\in\mathcal{T}(T_{\circlearrowleft},k)}\mathcal{R}_0(T_{\circlearrowleft}^\prime),\end{equation*}
		and the result follows.
	\end{proof}
	
	In the rest of this section, we take $\mathcal{O}=\mathcal{E}nd_A$, where $A$ is a dg algebra endowed with a symmetric, cyclic, bilinear form $\langle,\rangle$ of degree $m$. Recall from Example \ref{exmp:end operad} that $\langle,\rangle$ induces $\theta:A\to A^\vee[m]$ and $\Theta:\CH(A,A)\to\CH(A,A^\vee[m])$.
	To deal with signs, we may work with $A[1]$ instead of $A$. As explained in Appendix \ref{section:sign rule}, the symmetric bilinear form $\langle,\rangle$ on $A$ becomes anti-symmetric on $A[1]$, and the cyclic permutation ${\tau}_k$ on $\Hom(A^{\ot k+1},\mathbb{R})$ reads as $\tilde{\tau}_k=(-1)^k\tau_k=\lambda$ on $\Hom(A[1]^{\ot k+1},\mathbb{R})$. Since $\mathfrak{s}(\mathcal{E}nd_A)\cong\mathcal{E}nd_{A[1]}$, there is no need to take operadic suspension of $\mathsf{B}_{\circlearrowleft}$, and  $\mathsf{B}_{\circlearrowleft}(n)$ stands in degree 0 when dealing with $A[1]$.
	
	Since the pairing $\langle,\rangle$ 
	is not necessarily nondegenerate, there is not always a cyclic structure on $\mathcal{E}nd_{A}$ compatible with cyclic permutations on $\{\Hom(A^{\ot k+1},\mathbb{R})\}$ via the map $\Hom(A^{\ot k},A)\to\Hom(A^{\ot k},A^\vee[m])$ induced by $\langle,\rangle$, so the discussion of Example \ref{exmp:cyclic brace algebra} does not directly apply here. However, the following is true. 
	
	\begin{prop}\label{prop:cyclic brace action on CHcyc}
		There is a natural action of $\mathsf{B}_{\circlearrowleft}$ on $\Theta^{-1}(\prod_{k\geq 0}\Hom^{*+m}_{\mathrm{cyc}}(A^{\otimes k+1},\mathbb{R}))$.
	\end{prop}
	\begin{proof}
		(This proposition is irrelevant to the multiplication on $A$; compare Remark \ref{rem:end operad vs cyclic}.) Similar to Example \ref{exmp:cyclic brace algebra}, it suffices to show if $T_\circlearrowleft\in\mathsf{B}^s_\circlearrowleft(n)$ and $f_i\in\Hom(A^{\ot k_i},A)$ is weakly cyclic invariant in the sense that $\lambda(\theta\circ f_i)=\theta\circ f_i$, then $(\kappa\circ\rho)(T_\circlearrowleft)(f_1,\dots,f_n)$ is weakly cyclic invariant. 
		This is immediate from Lemma \ref{lem:cyclic brace action two equiv} and Lemma \ref{lem:cyclic brace invariant} below.
	\end{proof}
	\begin{cor}\label{cor:homotopy gravity action}
		The cochain complex $(\Theta^{-1}(\CH_{\mathrm{cyc}}(A,A^\vee[m])),d+\delta)$ admits an action of $\mathsf{M}_{\circlearrowleft}$. Moreover, if $\theta$ is a quasi-isomorphism and $\Theta$ restricts to a quasi-isomrphism from $\Theta^{-1}(\CH_{\mathrm{cyc}}(A,A^\vee[m]))$ to $\CH_{\mathrm{cyc}}(A,A^\vee[m])$, this $\mathsf{M}_{\circlearrowleft}$-action lifts the gravity algebra structure on $\HC_\lambda^*(A,A^\vee[m])$ induced by the $BV$ algebra structure on $\mathrm{HH}^*(A,A^\vee[m])$ (see Example \ref{exmp:end operad} \ref{item:statement D}).
	\end{cor}
	\begin{proof}
		As explained in the proof of \cite[Theorem 5.5]{Ward}, to show $\mathsf{M}_{\circlearrowleft}$ acts on the space of weakly cyclic invariants, it suffices to consider cyclic brace operations, which is nothing but Proposition \ref{prop:cyclic brace action on CHcyc}. In more details, Ward \cite{Ward} defined a dg operad $\mathsf{M}$ using ``$A_\infty$-labeled planar rooted black/white trees'' ($\mathsf{M}$ is isomorphic to the ``minimal operad'' of Kontsevich-Soibelman \cite{KS}), and $\mathsf{M}_{\circlearrowleft}$ is the non-rooted version of $\mathsf{M}$. $\mathsf{M}$ contains $\mathsf{B}$ as a graded suboperad, and acts on $\CH(A,A)$ extending brace operations. What $\mathsf{M}$ does more than $\mathsf{B}$ to $\CH(A,A)$ is generated by the operation $(f,g)\mapsto \mu_A\{f,g\}$ where $\mu_A\in\Hom(A^{\ot 2},A)$ is the multiplication. The action of $\mathsf{M}_{\circlearrowleft}$ on $\Theta^{-1}(\CH_{\mathrm{cyc}}(A,A^\vee[m]))$ comes from an operad morphism $\mathsf{M}_{\circlearrowleft}\to\mathsf{M}$ which extends the morphism $\mathsf{B}_{\circlearrowleft}\to\mathsf{B}$ from \eqref{eqn:tree add root}. Therefore, that $\Theta^{-1}(\CH_{\mathrm{cyc}}(A,A^\vee[m]))$ is closed under the action of $\mathsf{M}_{\circlearrowleft}$ on $\CH(A,A)$ essentially follows from Proposition \ref{prop:cyclic brace action on CHcyc} and weakly cyclic invariance of $\mu_A$, i.e. \eqref{eqn:<> cyclic invariant}.
		
		Now we explain why the $\mathsf{M}_{\circlearrowleft}$-action induces exactly the gravity algebra structure from the BV structure on homology under quasi-isomorphism assumptions; this is just by definition. The gravity algebra structure on $\HC_\lambda^*(A,A^\vee[m])$ follows from Lemma \ref{lem:BV induce gravity} and \eqref{eqn:cyclic exact seq Connes} with $V_*=\mathrm{HH}^{-*}(A,A)=\mathrm{HH}^{-*}(A,A^\vee[m])$, 
		$W_*=\HC_\lambda^{-*}(A,A^\vee[m])$, $\alpha=i_{\lambda*}$, $\beta=B_\lambda$. The product $\boldsymbol{\cdot}$ \eqref{eqn:cyclic operad total complex prod} on $\CH(A,A)$ is just the cup product $(f,g)\mapsto\mu_A\{f,g\}$, so the $k$-th gravity bracket \eqref{eqn:BV induces gravity} on $\HC_\lambda(A,A^\vee[m])$ is induced by the operation $(f_1,f_2,\dots,f_k)\mapsto\mu_A\{f_1,\mu_A\{f_2,\mu_A\{\cdots,\mu_A\{f_{k-1},f_k\}\cdots\}\}\}$ at chain level, which is represented by a certain black/white tree with $k-1$ adjacent black vertices labeled by $\mu_A$. Edges with both black vertices in such a tree should be contracted, creating a new tree with a single black vertex, see \cite[Appendix A, (A.10)]{Ward}. This gives exactly the trees representing the generators of $H_*(\mathsf{M}_{\circlearrowleft})$, see \cite[Definition 5.12, Figure 2]{Ward}.
	\end{proof}
	\begin{lem}\label{lem:cyclic brace invariant}
		Let $T^\prime_{\circlearrowleft}\in\bar{\mathsf{B}}_{\circlearrowleft}(n)$, $T^\prime\in\mathcal{R}_0(T^\prime_{\circlearrowleft})$, $f_i\in\Hom(A[1]^{\ot k_i},A[1])$ where $k_i=a_i(T^\prime)$ ($1\leq i\leq n$). Suppose every $f_i$ is weakly cyclic invariant. Then
		\begin{equation*}
			\theta\circ((\bar{\kappa}_n\circ r_n)(T^\prime_\circlearrowleft)(f_1,\dots,f_n))=\varepsilon(T_{\circlearrowleft}^\prime,T^\prime)\cdot N(\theta\circ(\bar{\kappa}_n(T^\prime)(f_1,\dots,f_n))).
		\end{equation*}	
	\end{lem}
	\begin{proof}
		In view of \eqref{eqn:cyclic permute all}, it suffices to prove the following equality:
		\begin{equation}\label{eqn:cyclic tree permute}
			\theta\circ((\bar{\kappa}_n\circ t_n)(T^\prime)(f_1,\dots,f_n))=\varepsilon(T^\prime,t_n(T^\prime))\cdot\lambda(\theta\circ(\bar{\kappa}_n(T^\prime)(f_1,\dots,f_n))).
		\end{equation}
		If there is only one tail or only one vertex (equivalently, no edges) in $T^\prime$, then $t_n$ acts trivially on $\bar{\mathsf{B}}(n)$, and \eqref{eqn:cyclic tree permute} is obvious. Now suppose there are at least two tails and at least one edge in $T^\prime$. Then there is a unique path in $T^\prime$ connecting the root $l_1$ to the first non-root tail $l_2$, consisting of successive edges with successive vertices $v_{i_1},v_{i_2},\dots,v_{i_k}$ ($k\geq1$), where $v_{i_1}$, $v_{i_k}$ are vertices of $l_1$, $l_2$, respectively. Note that these $k-1$ successive edges are the only edges in $T^\prime$ whose directions towards $l_1$ and $l_2$ disagree, so \begin{equation*}\varepsilon(T^\prime,t_n(T^\prime))=k-1.\end{equation*} 
		
		If $k=1$, \eqref{eqn:cyclic tree permute} simply follows from cyclic invariance of $\theta\circ f_{i_1}$. 
		
		If $k\geq2$, for each $j\in\{1,\dots,k-1\}$, denote the edge joining $v_{i_j}$ to $v_{i_{j+1}}$ by $[v_{i_j},v_{i_{j+1}}]$, which is identified with $[0,1]$. By removing $\frac{1}{2}\in[v_{i_j},v_{i_{j+1}}]$ for all $1\leq j\leq k-1$, $T^\prime$ is cut into $k$ pieces, where the $j$-th ($1\leq j\leq k$) piece $T_j$ contains $v_{i_j}$. By regarding $[v_{i_j},\frac{1}{2})$ and $(\frac{1}{2},v_{i_{j+1}}]$ as tails, these pieces become labeled planar trees, where the planar structures are induced from $T^\prime$, and the vertex labeling is in the same order as $T^\prime$: say the vertices of $T_j$ are labeled by $i_{j,1}<i_{j,2}<\dots<i_{j,n_j}$ in $T^\prime$, then put their labels as $1,2,\dots,n_j$ in $T_j$. Let $T_{j,-}^\prime$ (resp. $T_{j,+}^\prime$) be the labeled planar rooted tree obtained by choosing the tail $(\frac{1}{2},v_{i_j}]$ (resp. $[v_{i_j},\frac{1}{2})$) in $T_j$ as the root, where $(\frac{1}{2},v_{i_1}]$ is indeed $l_1$ and $[v_{i_k},\frac{1}{2})$ is indeed $l_2$. Suppose $(\frac{1}{2},v_{i_{j}}]$ (resp. $[v_{j},\frac{1}{2})$) is the $p_j$-th (resp. $q_j$-th) non-root tail in the total order on $L_{v_{i_{j}}}$ from the planar rooted structure of $T_{j,+}^\prime$ (resp. $T_{j,-}^\prime$). Since $l_2$ is the first non-root tail in $T^\prime$, we have $q_1=\dots=q_k=1$, and
		\begin{equation*}
			p_1+\dots+p_k-k+1=|L_{T^\prime}|-1=k_1+\dots+k_n+1-n=:K.
		\end{equation*}
		Denote
		\begin{equation*}
			F_j^+:=\bar{\kappa}_{n_j}(T_{j,+}^\prime)(f_{i_{j,1}},\dots,f_{i_{j,n_j}}),\hspace{.5em}F_j^-:=\bar{\kappa}_{n_j}(T_{j,-}^\prime)(f_{i_{j,1}},\dots,f_{i_{j,n_j}}).
		\end{equation*}
		Then for any $x_1,\dots,x_{K+1}\in A[1]$, there holds (Koszul sign $(-1)^{\varepsilon}$ is taken with respect to $A[1]$)
		\begin{align}\label{eqn:calculate cyclic brace}
			&\quad \langle(\bar{\kappa}_n\circ t_n)(T^\prime)(f_1,\dots,f_n)(x_1\ot\cdots\ot x_K),x_{K+1}\rangle\\\nonumber
			&=(-1)^{\varepsilon}\langle F_k^+\circ_{p_k}(F_{k-1}^+\circ_{p_{k-1}}(\cdots\circ_{p_2}F_1^+))(x_1\ot\cdots\ot x_K),x_{K+1}\rangle\\\nonumber
			&=(-1)^{\varepsilon}\langle F_k^-(\cdots\ot x_{p_k-1}),F_{k-1}^+\circ_{p_{k-1}}(\cdots\circ_{p_2}F_1^+)(x_{p_k}\ot\cdots)\rangle\\\nonumber
			&=-(-1)^{\varepsilon}\langle F_{k-1}^+\circ_{p_{k-1}}(\cdots\circ_{p_2}F_1^+)(x_{p_k}\ot\cdots),F_k^-(\cdots\ot x_{p_k-1})\rangle,
		\end{align}
		where the second equality follows from cyclic invariance of $\theta\circ f_{i_k}$, and the third equality follows from (anti-)symmetry of $\langle,\rangle$. Iterating the above calculation by cyclic invariance of $\theta\circ f_{i_{k-1}},\theta\circ f_{i_{k-2}},\dots,\theta\circ f_{i_1}$ and (anti-)symmetry of $\langle,\rangle$, we see that \eqref{eqn:calculate cyclic brace} is equal to
		\begin{align*}
			&\quad(-1)^k(-1)^{\varepsilon}\langle x_{p_1+\dots+p_k-k+1},F_1^-\circ_{q_1}(F_2^-\circ_{q_2}(\cdots\circ_{q_{k-1}}F_k^-))(\cdots\ot x_{p_1+\dots+p_k-k})\rangle\\
			&=(-1)^{k-1}(-1)^{\varepsilon}\langle F_1^-\circ_1(F_2^-\circ_1(\cdots\circ_1F_k^-))(x_{K+1}\ot x_1\ot\cdots\ot x_{K-1}),x_{K}\rangle\\
			&=(-1)^{\varepsilon(T^\prime,t_n(T^\prime))}(-1)^{\varepsilon}\langle\bar{\kappa}_n(T^\prime)(f_1,\dots,f_n)(x_{K+1}\ot x_1\ot\cdots),x_K\rangle.
		\end{align*}
		Since $\tilde{\tau}_K$ on $\Hom(A[1]^{\ot K},A[1]^\vee[m])$ corresponds to $\lambda$ on $\Hom(A^{\ot K},A^\vee[m])$, this proves \eqref{eqn:cyclic tree permute}. The proof is now complete.
	\end{proof}
	\begin{rem}
		It is easy to generalize Proposition \ref{prop:cyclic brace action on CHcyc} to $A_\infty$ algebras with cyclic invariant symmetric bilinear forms (not necessarily nondegenerate), and the proof is similar.
	\end{rem}   	
	
	\section{Chain level structures in $S^1$-equivariant string topology}\label{section:application to string topology} In this section we describe chain level structures in $S^1$-equivariant string topology, based on the previous results. Let us first review the initial homology level structures discovered by Chas-Sullivan, and the chain level construction due to Irie. 	
	
	\begin{exmp}[String topology BV algebra and gravity algebra]\label{exmp:string topology gravity} Let $M$ be a closed oriented manifold and $\mathcal{L}M$ be its free loop space. It was discovered by Chas-Sullivan in \cite{Chas-Sullivan string topology}\cite{Chas-Sullivan gravity} that:
		\begin{itemize}
			\item There is a BV algebra structure $(\Delta,\boldsymbol{\cdot})$ on $\mathbb{H}_*(\mathcal{L}M):=H_{*+\dim M}(\mathcal{L}M)$. Here $\Delta$ is induced by the $S^1$-action of rotating loops (i.e. $\Delta=J_*$ where $J$ is defined by \eqref{eqn:s1 action J}), $\boldsymbol{\cdot}$ is induced by concatenation of crossing loops and is called the \emph{loop product}. The associated Gerstenhaber bracket is called the \emph{loop bracket}. We call this BV algebra the \emph{string topology BV algebra}.
			\item There is a gravity algebra structure on $H_{*+\dim M-1}^{S^1}(\mathcal{L}M)$ (as an application of Lemma \ref{lem:BV induce gravity} to a part of the Gysin sequence \eqref{eqn:s1 equivariant gysin}), whose second bracket is the \emph{string bracket} (\cite[Theorem 6.1]{Chas-Sullivan string topology}) up to sign \eqref{eqn:sign suspension}. We call this gravity algebra the \emph{string topology gravity algebra}.
		\end{itemize}
	\end{exmp}
	A similar application of Lemma \ref{lem:BV induce gravity} to a part of the Connes-Gysin sequence \eqref{eqn:cyclic exact seq gysin [[u]]} for the mixed complex $(S_*(\mathcal{L}M),\partial,J)$, together with Lemma \ref{lem:cyclic exact seq gysin morphism} and Lemma \ref{lem:BV induce gravity morphism}, yields the following lemma.
	\begin{lem}\label{lem:gravity algebra negative s1}
		For any closed oriented manifold $M$, there is a gravity algebra structure on $G^{S^1}_{*+\dim M}(\mathcal{L}M)$, such that the natural map $H^{S^1}_{*+\dim M-1}(\mathcal{L}M)\to G^{S^1}_{*+\dim M}(\mathcal{L}M)$ in \eqref{eqn:s1 equivariant tautological} is a morphism of gravity algebras.\qed
	\end{lem}
	
	\begin{exmp}[Irie's construction \cite{Irie loop}] \label{exmp:Irie operad} Given any closed oriented $C^\infty$-manifold $M$, one can associate to $M$ a ns cyclic dg operad $(\mathcal{O}_M,(\tau_k)_{k\geq0},\mu,\varepsilon)$ with a multiplication and a unit, defined by:
		\begin{itemize}
			\item For each $k\in\mathbb{Z}_{\geq0}$, $(\mathcal{O}_M(k)_*,\partial):=\big(C^{\mathrm{dR}}_{*+\dim M}(\mathscr{L}^M_{k+1,\mathrm{reg}}),\partial\big)$.
			\item For each $k\in\mathbb{Z}_{\geq1},k^\prime\in\mathbb{Z}_{\geq0}$ and $j\in\{1,\dots,k\}$, the partial composition
			$
			\circ_j:\mathcal{O}_M(k)\ot \mathcal{O}_M(k^\prime)\to\mathcal{O}_M(k+k^\prime-1)$
			is defined by \begin{equation*}x\circ_j x^\prime:=(\mathrm{con}_j)_*(x\fd{\mathrm{ev}_j}{\mathrm{ev}_0} x^\prime),
			\end{equation*}
			where ${}\fd{\mathrm{ev}_j}{\mathrm{ev}_0}{}$ is the fiber product of de Rham chains with respect to evaluation maps $\mathrm{ev}_j:\mathscr{L}^M_{k+1,\mathrm{reg}}\to M_{\mathrm{reg}}$ and $\mathrm{ev}_0:\mathscr{L}^M_{k^\prime+1,\mathrm{reg}}\to M_{\mathrm{reg}}$ (it is well-defined because of submersive condition), and $\mathrm{con}_j:\mathscr{L}_{k+1}M\fd{\mathrm{ev}_j}{\mathrm{ev}_0} \mathscr{L}_{k^\prime+1}M\to\mathscr{L}_{k+k^\prime}M$ is the concatenation map defined by inserting the second loop into the first loop at the $j$-th marked point.
			\item For each $k\in\mathbb{Z}_{\geq0}$, $\tau_k:\mathcal{O}_M(k)_*\to\mathcal{O}_M(k)_*$ is induced by \eqref{eqn:cyclic structure on loop}.
			\item $1_{\mathcal{O}_M}:=(M,i_1,1)\in\mathcal{O}_M(1)_0$, $\mu:=(M,i_2,1)\in\mathcal{O}_M(2)_0$, $\varepsilon:=(M,i_0,1)\in\mathcal{O}_M(0)_0$. Here for $k\geq0$, $i_k:M\to\mathscr{L}_{k+1}M$ is the map $p\mapsto (0,\gamma_p,0,\dots,0)$, where $\gamma_p$ is the constant loop of length 0 at $p\in M$.
		\end{itemize}
		By \cite[Theorem 3.1(ii)]{Irie loop}, there is an isomorphism $H_*(\tilde{\mathcal{O}}_M,b)\cong \mathbb{H}_*(\mathcal{L}M)$ 
		of BV algebras, where these BV algebra structures are from Proposition \ref{prop:cyclic operad induce structure} and Example \ref{exmp:string topology gravity}, respectively. (The crucial thing about $\tilde{\mathcal{O}}_M$ is the chain level structure which refines the string topology BV algebra, but we do not need to use it.)
		
		Let $(\Omega(M)^*,d,\wedge)$ be the dg algebra of differential forms on $M$. For each $k\geq0$, there is a chain map $I_k:C^{\mathrm{dR}}_{*+\dim M}(\mathscr{L}^M_{k+1,\mathrm{reg}})\to \Hom^{-*}(\Omega(M)^{\ot k},\Omega(M))$, called \emph{iterated integral of differential forms}: for $\eta_1,\dots,\eta_k\in\Omega(M)$,
		\begin{equation*}
			I_k(U,\vphi,\omega)(\eta_1\ot\cdots\ot\eta_k):=(-1)^{\varepsilon_0}(\vphi_0)_!(\omega\wedge\vphi_1^*\eta_1\wedge\cdots\wedge\vphi_k^*\eta_k),
		\end{equation*}
		where $\varepsilon_0:=(\dim U-\dim M)(|\eta_1|+\cdots+|\eta_k|)$ and $\vphi_j:=\mathrm{ev}_j\circ\vphi$ ($0\leq j\leq k$). Moreover, $I=(I_k)_{k\geq0}:\mathcal{O}_M\to \mathcal{E}nd_{\Omega(M)}$ is a morphism of ns dg operads preserving multiplications and units (\cite[Lemma 8.5]{Irie loop}). 
	\end{exmp}
	The paring $\langle\alpha,\beta\rangle:=\int_M\alpha\wedge\beta$ is a graded symmetric bilinear form on $\Omega^*(M)$ of degree $m=-\dim M$, in line with Example \ref{exmp:end operad}. The induced dg $\Omega(M)$-bimodule map $\theta:\Omega^*(M)\to(\Omega(M)^\vee[-\dim M])^*$ is a quasi-isomorphism by Poincar\'e duality, hence induces a quasi-isomorphism
	\begin{equation}\label{eqn:Hochschild pairing differential form}
		\Theta:\CH(\Omega(M),\Omega(M))\xrightarrow{\simeq}\CH(\Omega(M),\Omega(M)^\vee[-\dim M]).
	\end{equation}

	\begin{lem}\label{lem:iterated integral cocyclic}
		The composition
		\begin{equation}\label{eqn:iterated integral composition}
			\theta\circ I_k:\mathcal{O}_M(k)_*\to\Hom^{-*}(\Omega(M)^{\ot k},\Omega(M)^\vee[-\dim M])\hspace{.5em}(k\geq0)
		\end{equation}
		is a morphism of cocyclic complexes. 
	\end{lem}
	\begin{proof}
		The composition $\{\theta\circ I_k\}_k$ is one of cosimplicial maps, so it suffices to check $\tau_k\circ\theta\circ I_k=\theta\circ I_k\circ\tau_k$, which is a simple computation by definition.
	\end{proof}
	According to Lemma \ref{lem:iterated integral cocyclic}, $\Theta\circ I:\tilde{\mathcal{O}}_M\to\CH(\Omega(M),\Omega(M)^\vee[-\dim M])$ preserves cyclic invariants. Moreover, the following is true.
	\begin{lem}\label{lem:iterated integral cyclic brace morphism}
		The chain map
		\begin{equation}\label{eqn:iterated integral cyclic brace algebra}
			I:(\tilde{\mathcal{O}}_M^{\mathrm{cyc}})_*\to \Theta^{-1}(\CH^{-*}_{\mathrm{cyc}}(\Omega(M),\Omega(M)^\vee[-\dim M]))
		\end{equation}
		is a morphism of $\mathsf{M}_{\circlearrowleft}$-algebras.
	\end{lem}
	\begin{proof}
		First, \eqref{eqn:iterated integral cyclic brace algebra} is a morphism of $\mathsf{B}_{\circlearrowleft}$-algebras since $I:\mathcal{O}_M\to\mathcal{E}nd_{\Omega(M)}$ is a morphism of ns operads, and the (cyclic) brace operations on the associated complexes are defined using operadic compositions. Then by the proof of Corollary \ref{cor:homotopy gravity action}, to show \eqref{eqn:iterated integral cyclic brace algebra} is a morphism of $\mathsf{M}_{\circlearrowleft}$-algebras, it suffices to show $I_2(\mu)=\wedge$, where $\mu=(M,i_2,1)\in \mathcal{O}_M(2)$ is the multiplication in $\mathcal{O}_M$. But this is obvious from definition.
	\end{proof}
	\begin{thm}\label{thm:chain level gravity string topology}
		For any closed oriented $C^\infty$-manifold $M$, the ns dg operad $\mathcal{O}_M$ with $(\tau_k)_{k\geq0}$, $\mu$, $\varepsilon$ in Example \ref{exmp:Irie operad} gives rise to the following data:
		\begin{enumerate}
			\item \label{item:thm chain level 1}A chain complex $\tilde{\mathcal{O}}^{\mathrm{cyc}}_M:=\mathrm{Ker}(1-\lambda)\subset\tilde{\mathcal{O}}_M$ which is an algebra over $\mathsf{M}_{\circlearrowleft}$. In particular, $H_*(\tilde{\mathcal{O}}^{\mathrm{cyc}}_M)$ carries a gravity algebra structure.
			\item \label{item:thm chain level 2}An isomorphism $H_*(\tilde{\mathcal{O}}^{\mathrm{cyc}}_M)\cong G^{S^1}_{*+\dim M}(\mathcal{L}M)$ of gravity algebras, where the gravity algebra structure on $G^{S^1}_{*+\dim M}(\mathcal{L}M)$ is as in Lemma \ref{lem:gravity algebra negative s1}.
			\item \label{item:thm chain level 3}A morphism $I:(\tilde{\mathcal{O}}^{\mathrm{cyc}}_M)_{*}\to\Theta^{-1}(\CH_{\mathrm{cyc}}^{-*}(\Omega(M),\Omega(M)^\vee[-\dim M]))$ of $\mathsf{M}_{\circlearrowleft}$-algebras, such that the induced map in homology 
			fits into the following commutative diagram of gravity algebra homomorphisms:
			\begin{equation}\label{eqn:iterated integral gravity}			
				\begin{tikzcd}				
					H^{S^1}_{*+\dim M-1}(\mathcal{L}M)\arrow[r,"1"]\arrow[dd,"2"] & \HC^{-*+1}_{[u^{-1}]}(\CH(\Omega(M),\Omega(M)^\vee[-\dim M]))\arrow[d,"3"] \\
					& \HC^{-*}_{[[u]]}(\CH(\Omega(M),\Omega(M)^\vee[-\dim M]))  \arrow[d, equal, "\eqref{eqn:cyclic Connes isom}"]\\
					G^{S^1}_{*+\dim M}(\mathcal{L}M)\arrow[r,"4"] & \HC^{-*}_{\lambda}(\CH(\Omega(M),\Omega(M)^\vee[-\dim M])).
				\end{tikzcd}		
			\end{equation}
			Here arrows 1, 4 are induced by \eqref{eqn:iterated integral composition}, arrow 2 is as in \eqref{eqn:s1 equivariant tautological}, arrow 3 is as in \eqref{eqn:cyclic exact seq tautological}. The gravity algebra structures are those on the (negative) $S^1$-equivariant homology of $\mathcal{L}M$ (Example \ref{exmp:string topology gravity}, Lemma \ref{lem:gravity algebra negative s1}) and on (negative) cyclic cohomology of $\Omega(M)$ (Example \ref{exmp:end operad}) in view of Poincar\'e duality.
		\end{enumerate}
		
		\begin{proof}
			Statement \eqref{item:thm chain level 1} follows from Proposition \ref{prop:cyclic operad induce structure}\eqref{item:statement 5}. Statement \eqref{item:thm chain level 2} follows from Proposition \ref{prop:cyclic operad induce structure}\eqref{item:statement 4} and Proposition \ref{prop:Irie loop s1 equiv}. As for statement \eqref{item:thm chain level 3}, $I$ is defined in Lemma \ref{lem:iterated integral cyclic brace morphism}; Arrows 2, 3 are gravity algebra homomorphisms by Proposition \ref{prop:cyclic operad induce structure}\eqref{item:statement 4} and Example \ref{exmp:end operad}; Arrows 1, 4 are gravity algebra homomorphisms by Lemma \ref{lem:iterated integral cocyclic}, Lemma \ref{lem:cyclic long exact seq naturality} and Lemma \ref{lem:BV induce gravity morphism}; The diagram \eqref{eqn:iterated integral gravity} commutes by Lemma \ref{lem:cyclic long exact seq naturality}. Strictly speaking, since the grading of $\mathcal{O}_M(k)_*$ has been shifted by $\dim M$ from $C^{\mathrm{dR}}_{*}(\mathscr{L}_{k+1,\mathrm{reg}}^M)$, there is a minor sign change for $\delta$ \eqref{eqn:cocyclic complex delta defn} in $\tilde{\mathcal{O}}_M$ compared to $C^{\mathscr{L}}$ (the same thing happens in \cite[Lemma 8.4]{Irie loop}), and thus we should repeat the proof of Proposition \ref{prop:Irie loop s1 equiv} under new signs and use new isomorphisms to make the diagram commute without question of signs, but this is straightforward.
		\end{proof}
	\end{thm}
	
	\begin{rem}
		Theorem \ref{thm:chain level gravity string topology}\eqref{item:thm chain level 1} is an easy combination of work of Irie and Ward, so it is not new. But it was not known before whether the chain level structures in statement \eqref{item:thm chain level 1} correctly fit with known homology level structures in string topology (it was not even known whether $H_*(\tilde{\mathcal{O}}^{\mathrm{cyc}}_M)$ is isomorphic to the $S^1$-equivariant homology of $\mathcal{L}M$), so statement \eqref{item:thm chain level 2} is new. As for statement \eqref{item:thm chain level 3}, some (perhaps not all) of the homology level statements are known, see the discussion after Theorem \ref{thm:main result 3}; the chain level statement is more crucial, and is new because the result that $\mathsf{M}_{\circlearrowleft}$ (nontrivially) acts on $\Theta^{-1}(\CH_{\mathrm{cyc}}^{-*}(\Omega(M),\Omega(M)^\vee[-\dim M]))$ is new (Corollary \ref{cor:homotopy gravity action}).
	\end{rem}
	
	\begin{rem}	It is known that if $M$ is simply connected, then the iterated integral map $I:(\tilde{\mathcal{O}}_M)_*\to\CH^{-*}(\Omega(M),\Omega(M))$ is a quasi-isomorphism  (proved by K. T. Chen \cite{Chen iterated integral} and improved by Getzler-Jones-Petrack \cite{GJP}). In this case Lemma \ref{lem:mixed complex quasi} implies that arrows 1, 4 in \eqref{eqn:iterated integral gravity} are isomorphisms of gravity algebras. 
	\end{rem}   	
	
	Note that arrow 4 in \eqref{eqn:iterated integral gravity} is not exactly induced by $I$, but is the composition $\Theta_*\circ I_*$. The author does not know an answer to the following question.
	\begin{conj}\label{conj:weakly cyclic invariants homology}
		For any closed oriented $C^\infty$-manifold $M$, the quasi-isomorphism \eqref{eqn:Hochschild pairing differential form} restricts to a quasi-isomorphism on (weakly) cyclic invariants,
		\begin{equation*}
			\Theta^{-1}(\CH^{-*}_{\mathrm{cyc}}(\Omega(M),\Omega(M)^\vee[-\dim M]))\xrightarrow{\simeq}\CH^{-*}_{\mathrm{cyc}}(\Omega(M),\Omega(M)^\vee[-\dim M]).
		\end{equation*}
	\end{conj}
	
	\appendix
	\section{Sign rules}\label{section:sign rule}
	\subsection{Koszul sign rule}\label{subsection:Koszul sign rule} Compared to ungraded formulas, a sign $(-1)^{|a||b|}$ is produced in graded setting whenever a symbol $a$ travels across another symbol $b$. For example if $A,B$ are graded vector spaces, the graded tensor product of graded linear maps $f:A\to B$ and $g:C\to D$, $f\otimes g:A\otimes C\to B\otimes D$, is defined by $(f\otimes g)(v\otimes w)=(-1)^{|g||v|}f(v)\otimes g(w)$.
	
	\subsection{Sign change rule for (de)suspension}\label{subsection:sign change rule for suspension} Let $C=\{C^i\}_{i\in\mathbb{Z}}$ be a graded vector space. For any $n\in\mathbb{Z}$, define a shifted graded vector space $C[n]=\{C[n]^i\}_{i\in\mathbb{Z}}$ by $C[n]^i:=C^{i+n}$. (In homological grading this turns into $C[-n]_{-i}:=C_{-i-n}$.) $C[-1]$ is often denoted by $\Sigma C$ and is called the suspension of $C$. Let $s:C\to C[-1]$; $x\mapsto sx$ be the shifted identity map which is of degree $1$. By the Kozsul sign rule, for $x_1,\dots,x_k\in C$,
	\begin{equation}\label{eqn:sign suspension}
		s^{\ot k}(x_1\ot\cdots\ot x_k)=(-1)^{\sum_{i=1}^k(k-i)|x_i|}sx_1\ot\cdots\ot sx_k.
	\end{equation}
	Here $|x_i|$ denotes the degree of $x_i$ in $C$, and the sign $(-1)^{(k-i)|x_i|}$ comes from exchanging positions of $k-i$ copies of $s$ with that of $x_i$. The sign change \eqref{eqn:sign suspension} identifies the graded exterior algebra of $C$ with the graded symmetric algebra of $C[-1]$, as	
	\begin{equation*}	
		(s\ot s)((-1)^{|x_1||x_2|}x_2\ot x_1)	=(-1)^{|x_1|}((-1)^{1+|sx_1||sx_2|}sx_2\ot sx_1).
	\end{equation*}
	The same rule applies to the sign change between $A$ and $A[1]$, the desuspension of $A$.
	\subsection{Operadic suspension}\label{subsection:operadic suspension} Let $(\mathcal{O},\circ_i)$ be a dg operad in cohomological grading. The operadic suspension of $\mathcal{O}$ is a dg operad $(\mathfrak{s}\mathcal{O})$ with partial compositions $\tilde{\circ}_i$ satisfying
	\begin{align}\nonumber
		&\mathfrak{s}\mathcal{O}(n)=\mathcal{O}(n)[1-n],\\
		&a\tilde{\circ}_i b=(-1)^{(i-1)(m-1)+(n-1)|b;\mathcal{O}(m)|},\label{eqn:operadic suspension}
	\end{align}
	where $a\in \mathfrak{s}\mathcal{O}(n)$, $b\in\mathfrak{s}\mathcal{O}(m)$, $|b;\mathcal{O}(m)|$ is the degree of $b$ in $\mathcal{O}(m)$. For an explanation of signs (which comes from Koszul sign rule), see e.g. \cite[Section 2.5.4]{Irie loop}. 
	
	When $\mathcal{O}=\mathcal{E}nd_A$ is the endomorphism operad of a dg algebra $A$, there is an isomorphism of dg operads $\mathfrak{s}\mathcal{O}=\mathfrak{s}(\mathcal{E}nd_A)\cong\mathcal{E}nd_{A[1]}$. Therefore, for signs related to $\mathfrak{s}(\mathcal{E}nd_A)$, one may alternatively use Koszul sign rule for $A[1]$ and perform \eqref{eqn:sign suspension} when necessary.
	
	\subsection{Cyclic permutation}\label{subsection:cyclic permutation}
	If $(\mathcal{O},\circ_i)$ is a cyclic dg operad
	, then $(\mathfrak{s}\mathcal{O},\tilde{\circ}_i)$ also carries a cyclic structure where $\tilde{\tau}_k=(-1)^k\tau_k$ under the naive identification $\mathfrak{s}\mathcal{O}(k)=\mathcal{O}(k)$. On the other hand, let $A$ be a dg algebra, consider the cocyclic complex $\{\Hom(A^{\ot k+1},\mathbb{R}),\tau_k\}$ and the operation $\tilde{\tau}_k$ on $\Hom(A[-1]^{\ot k+1},\mathbb{R})$ induced by $\tau_k$ under the linear isomorphism $s:A\to A[-1]$. Then the following equality says $\tilde{\tau}_k=(-1)^k\tau_k$ after sign change \eqref{eqn:sign suspension}:
	\begin{equation*}
		\tilde{\tau}_k\circ s^{\ot k+1}=\tilde{\tau}_k\circ (s^{\ot k}\ot s)=(-1)^k(s\ot s^{\ot k})\circ\tau_k,
	\end{equation*}
	where $s^{\ot k}$ applies to $x_1\ot\cdots\ot x_k\in A^{\ot k}$ and $s$ applies to $x_{k+1}\in A$.
	Therefore, when discussing cyclic homology theories of $A$ under the naive identification $A=A[-1]$, the subspace of cyclic invariants in $C(k)$ is $\mathrm{Ker}(1-\lambda)=\mathrm{Ker}(1-\tilde{\tau}_k)$ and the operator $N$ satisfies $N|_{C(k)}=\sum_{i=0}^k\lambda^i=\sum_{i=0}^k\tilde{\tau}_k^i$.


\begin{thebibliography}{99}
		\bibitem{ATZ}
		H.~Abbaspour, T.~Tradler and M.~Zeinalian, Algebraic string bracket as a Poisson bracket. \emph{J. Noncommut. Geom.} \textbf{4} (2010), no.~3, 331--347 
		
		\bibitem{ACD}
		A.~Adem, R.~L. Cohen and W.~G. Dwyer, Generalized Tate homology, homotopy fixed points and the transfer. \emph{Algebraic topology (Evanston, IL, 1988)}, 1--13, Contemp. Math., 96, Amer. Math. Soc. 
		
		\bibitem{Chas-Sullivan string topology}
		M.~Chas and D.~Sullivan, String topology. arxiv.org/abs/math/9911159, 1999
		
		\bibitem{Chas-Sullivan gravity}
		M.~Chas and D.~Sullivan, Closed string operators in topology leading to Lie bialgebras and higher string	algebra. In \emph{The legacy of Niels Henrik Abel}, pp. 771--784, Springer-Verlag, Berlin, 2004 
		
		\bibitem{Chen iterated integral}
		K.~T. Chen, Iterated integrals of differential forms and loop space homology. \emph{Ann. of Math. (2)} \textbf{97} (1973), 217--246
		
		\bibitem{Chen differentiable space}
		K.~T. Chen, On differentiable spaces. In \emph{Categories in continuum physics (Buffalo, N.Y., 1982)}, 38--42. Lecture Notes in Mathematics, vol 1174. Springer-Verlag, Berlin, 1986
		
		\bibitem{Chen string}
		X.~Chen, An algebraic chain model of string topology. \emph{Trans. Amer. Math. Soc.} \textbf{364} (2012), no.~5, 2749--2781
		
		\bibitem{CV cyclic}
		K.~Cieliebak and E.~Volkov, Eight flavors of cyclic homology. \emph{Kyoto J. Math.} \textbf{61} (2021), no.~2, 495--541
		
		\bibitem{DW}
		V.~Dolgushev and T.~Willwacher, Operadic twisting -- with an application to Deligne’s conjecture. \emph{J. Pure Appl. Algebra} \textbf{219} (2015), no.~5, 1349--1428
		
		\bibitem{Fukaya loop}
		K.~Fukaya, Application of Floer homology of Lagrangian submanifolds to symplectic topology. In \emph{Morse Theoretic Methods in Nonlinear Analysis and in Symplectic Topology}, 231--276. NATO Sci. Ser. II Math. Phys. Chem., 217. Springer, Dordrecht, 2006 
		
		\bibitem{Fukaya cyclic} 
		K.~Fukaya, Cyclic symmetry and adic convergence in Lagrangian Floer theory. \emph{Kyoto J. Math.} \textbf{50} (2010), no.~3, 521--590
		
		\bibitem{Gerstenhaber}
		M.~Gerstenhaber, The cohomology structure of an associative ring. \emph{Ann. of Math. (2}) \textbf{78} (1963), 267--288
		
		\bibitem{GV}
		M.~Gerstenhaber and A.~A. Voronov, Homotopy G-algebras and moduli space operad. \emph{Int. Math. Res. Not.} (1995), no.~3, 141--153
		
		\bibitem{Getzler brace}
		E.~Getzler, Cartan homotopy formulas and the Gauss-Manin connection in cyclic homology. In \emph{Quantum deformations of algebras and their representations (Ramat-Gan, 1991/1992; Rehovot, 1991/1992)}, 65--78. Israel Math. Conf. Proc., 7. Bar-Ilan Univ., Ramat Gan, 1993
		
		\bibitem{Getzler BV}
		E.~Getzler, Batalin-Vilkovisky algebras and two-dimensional topological field theories. \emph{Comm. Math. Phys.} \textbf{159} (1994), no.~2, 265--285
		
		\bibitem{Getzler gravity}
		E.~Getzler, Two-dimensional topological gravity and equivariant cohomology. \emph{Comm. Math. Phys.} \textbf{163} (1994), no.~3, 473--489.
		
		\bibitem{GJP}
		E.~Getzler, J.~D. S. Jones, and S.~Petrack, Differential forms on loop spaces and the cyclic bar complex, \emph{Topology} \textbf{30} (1991), no.~3, 339--371
		
		\bibitem{Goodwillie}
		T.~G. Goodwillie, Cyclic homology, derivations and the free loop space. \emph{Topology} \textbf{24} (1985), no.~2, 187--215
		
		\bibitem{Hatcher}
		A.~Hatcher, \emph{Algebraic Topology}. Cambridge University Press, Cambridge, 2002 
		
		\bibitem{Irie loop}
		K.~Irie, A chain level Batalin-Vilkovisky structure in string topology via de Rham chains. \emph{Int. Math. Res. Not.} (2018), no.~15, 4602--4674 
		
		\bibitem{Jones cyclic}
		J.~D. S. Jones, Cyclic homology and equivariant homology. \emph{Invent. Math.} \textbf{87} (1987), no.~2, 403--423
		
		\bibitem{Kassel}
		C.~Kassel, Cyclic homology, comodules, and mixed complexes. \emph{J. Algebra} \textbf{107} (1987), no.~1, 195--216
		
		\bibitem{K}
		M.~Khalkhali, \emph{Basic noncommutative geometry}, second edition. EMS Ser. Lect. Math., European Mathematical Society (EMS), Zürich, 2013
		
		\bibitem{KS}
		M.~Kontsevich and Y.~Soibelman, Deformations of algebras over operads and the Deligne conjecture. In \emph{Conf\'erence Mosh\'e Flato 1999, Vol. I (Dijon)}, 255--307, Math. Phys. Stud., 21, Kluwer Academic Publishers Group, Dordrecht, 2000 
		
		\bibitem{Loday cyclic book}
		J.-L.~Loday, \emph{Cyclic Homology}. Appendix E by María O. Ronco, Chapter 13 by the author in collaboration with Teimuraz Pirashvili, second edition, Grundlehren Math. Wiss. 301 [Fundamental Principles of Mathematical Sciences], Springer-Verlag, Berlin, 1998 
		
		\bibitem{Loday operad book}
		J.-L.~Loday and B.~Vallette, \emph{Algebraic Operads}. Grundlehren Math. Wiss., 346 [Fundamental Principles of Mathematical Sciences], Springer, Heidelberg, 2012 
		
		\bibitem{MSS operad book}
		M.~Markl, S.~Shnider and J.~Stasheff, \emph{Operads in Algebra, Topology and Physics}. Math. Surveys Monogr., 96, American Mathematical Society, Providence, RI, 2002 
		
		\bibitem{Menichi cyclic}
		L.~Menichi, Batalin-Vilkovisky algebras and cyclic cohomology of Hopf algebras. \emph{K-Theory} \textbf{32} (2004), no.~3, 231--251
		
		\bibitem{Menichi HH BV}
		L.~Menichi, Batalin-Vilkovisky algebra structures on Hochschild cohomology. \emph{Bull. Soc. Math. France} \textbf{137} (2009), no.~2, 277--295
		
		\bibitem{Penkava}
		M.~Penkava, Infinity algebras, cohomology and cyclic cohomology, and infinitesimal deformations. https://arxiv.org/abs/math/0111088, 2001
		
		\bibitem{Ward}
		B.~C. Ward, Maurer-Cartan elements and cyclic operads. \emph{J. Noncommut. Geom.} \textbf{10} (2016), no.~4, 1403--1464 
		
		\bibitem{Weibel}
		C.~A. Weibel, \emph{An Introduction to Homological Algebra}. Cambridge Stud. Adv. Math., 38, Cambridge University Press, Cambridge, 1994
		
		\bibitem{Westerland}
		C.~Westerland, Equivariant operads, string topology, and Tate cohomology. \emph{Math. Ann.} \textbf{340} (2008), no.~1, 97--142
		
		\bibitem{Zhao cyclic}
		J.~Zhao, Periodic symplectic cohomologies. \emph{J. Symplectic Geom.} \textbf{17} (2019), no.~5, 1513--1578 
		
		
		
		
		
		
		
	\end{thebibliography}
\end{document}